%% file: main_v5.tex
\documentclass[11pt,twoside]{article}

\usepackage{xcolor}

\usepackage{fullpage}

\usepackage{epsf}
\usepackage{fancyhdr}
\usepackage{graphics}
\usepackage{graphicx}
\usepackage{psfrag}
\usepackage{comment}

\usepackage[linesnumbered,ruled]{algorithm2e}% http://ctan.org/pkg/algorithm2e
\DontPrintSemicolon	

\usepackage{color}
\usepackage{amsthm}
\usepackage{amsfonts}
\usepackage{amsmath}
\usepackage{bm}
\usepackage{amssymb,bbm}
\usepackage{natbib}
\usepackage{algorithmic}
\usepackage{url}
\usepackage[colorlinks=True,linkcolor=magenta,citecolor=blue,urlcolor=blue,pagebackref=true,backref=true]
{hyperref}
\renewcommand*{\backref}[1]{\ifx#1\relax \else Page #1 \fi}
\renewcommand*{\backrefalt}[4]{%
    \ifcase #1 \footnotesize{(Not cited.)}%
    \or        \footnotesize{(Cited on page~#2.)}%
    \else      \footnotesize{(Cited on pages~#2.)}%
    \fi}

\usepackage{nicefrac}

\def\half{\hbox{$1\over2$}}

\usepackage{chngpage}

 \usepackage{tabularx}%

\usepackage{enumitem}
\usepackage{booktabs}
\usepackage{pbox}

\usepackage{caption,subcaption}

\usepackage{mathtools}

\usepackage{fullpage}

\input{final_macros}

\begin{document}
\begin{center}

{\bf{\LARGE{Multivariate Smoothing via the Fourier Integral Theorem and Fourier Kernel}}}
  
\vspace*{.2in}
{\large{
\begin{tabular}{cc}
Nhat Ho$^{\diamond}$ & Stephen G. Walker$^{\diamond, \flat}$ \\
\end{tabular}
}}

\vspace*{.2in}

\begin{tabular}{c}
Department of Statistics and Data Sciences, , University of Texas at Austin$^\diamond$, \\
Department of Mathematics, University of Texas at Austin$^\flat$ \\
\end{tabular}

\today

\vspace*{.2in}

\begin{abstract}
Starting with the Fourier integral theorem, we present natural Monte Carlo estimators of multivariate functions including densities, mixing densities, transition densities, regression functions, and the search for modes of multivariate density functions (modal regression). Rates of convergence are established and, in many cases, provide superior rates to current standard estimators such as those based on kernels, including kernel density estimators and kernel regression functions.
Numerical illustrations are presented. 
\end{abstract}
\end{center}


\section{Introduction}
Nonparametric function estimation allows for a data driven form for the estimator with little to no constraints on shape. Early work included kernel density estimators; \cite{Rosen1956} and \cite{Parzen62} and regression estimators; \cite{Nadaraya64} and \cite{Watson64}. Other nonparametric estimators include those of a mixing density, \cite{Laird78}, hazard and cumulative hazard functions and other related functions. 

While there are many approaches to function estimation, such as polynomials, basis functions and splines for regression functions; see  \cite{DonJohn98}, \cite{Fan93}, \cite{FanGijb96}, \cite{GreenSilv94}, \cite{Stone85}, \cite{Tib14}, \cite{Tsy09}, \cite{Wahba90}, and \cite{Wass06},
kernel methods remain popular.

The main contribution of the present paper is multivariate kernel estimation, and, in particular, for regression functions. The Gaussian kernel is used almost exclusively and a number of authors advocate the use of a multivariate Gaussian kernel. However, even in the bivariate case, a number of issues arise regarding the covariance matrix; see \cite{WJones93}. Some authors advocate a diagonal matrix, e.g., \cite{Wand94}, though for regression function estimation such a plan is problematic. On the other hand, selecting a bandwidth covariance matrix is also a non trivial problem; see \cite{Wand92}, \cite{Stanis93} and \cite{Chacon18}.

In the one dimensional case, a number of authors have considered various kernels, $K(u)$, the main condition being that
$$\int_{-\infty}^\infty K(u)\,d u=1.$$
With this in mind, the Fourier kernel is given by $K(u)=\pi^{-1}\sin(u)/u$, and has been mentioned and looked in early work by 
\cite{Parzen62} and \cite{Davis75} for density estimation.

For reasons unclear to us, there is no, as far as we can ascertain, use of the Fourier kernel for regression smoothing. The Gaussian kernel dominates here due to the possibility of incorporating  a covariance structure in the multivariate case when there are multiple predictor variables. There seems little room for such a covariance matrix within the Fourier kernel. However, as we shall highlight, there is no need for one in the multivariate case; a product of Fourier kernels suffice, which is not so for the Gaussian kernel. First we will introduce the key idea lightly and then be more formal.

The unique aspect of the Fourier kernel is that is satisfies the Fourier integral theorem; i.e., for all suitable functions $m(x)$, with $x\in\mathbb{R}^d$,
\begin{equation}\label{fitt}
m(y)=\frac{1}{\pi}\lim_{R\to\infty} \int \prod_{j=1}^d \frac{\sin (R(y_j-x_j))}{y_j-x_j}\,m(x)\,d x.
\end{equation}
So the product of kernels over dimensions preserves any covariance or dependence structure, automatically, lying within $m(x)$. There is no need to seek out a covariance or dependence structure, as there is with the Gaussian kernel which does not satisfy equation (\ref{fitt}). Hence, for the Gaussian kernel to preserve good approximations, the product of independent kernels over the dimensions would need some attention, such as the inclusion of a covariance structure.

It could well be that the lack of ability of placing a covariance structure suitably within the Fourier kernel is the reason why it has not been looked at in multidimensional problems. However, we have just argued, through (\ref{fitt}), it is not required. 

To be more formal, consider the Fourier integral theorem in one dimension,
\begin{equation}\label{finthem}
m(y)=\frac{1}{2\pi}\lim_{\radius \to\infty }\int_{-R}^R\int_{-\infty}^\infty \cos(s(y-x))\,m(x)\,dx\,ds,
\end{equation}
for $m \in \mathbb{L}_1(\mathbb{R})$. This is an application of the Fourier and Fourier inverse transforms; see for example \cite{Wiener33}.
Hence, an approximation based on the choice of a finite $\radius$, and integrating over $s$, yields
$$m_{\radius}(y)=\frac{1}{\pi}\int_{-\infty}^\infty \frac{\sin(R(y-x))}{y-x}\,m(x)\,dx.$$
In particular, if $m=p_{0}$ is a density function, and $X_1,\ldots,X_n$ are an i.i.d. sample from $p_{0}$, then a Monte Carlo estimate of the density is
$$\funcest_{n,R}(x) =\frac{1}{n\pi}\sum_{i=1}^n \frac{\sin(\radius(x-X_i)}{x-X_i}.$$
The extension to higher dimensions $\mathbb{R}^d$ is a simple procedure, based on
\begin{equation}\label{eq:Fourier_inverstion}
p_{0}(y)=\frac{1}{(2\pi)^d}\lim_{\radius \to\infty}\int_{-\radius}^{\radius}\ldots\int_{-\radius}^{\radius} \int_{\mathbb{R}^d}\cos(s^{\top}(y-x))\,p(x)\,dx\,ds,
\end{equation}
where now $x=(x_1,\ldots,x_d)$.
Proceeding along similar lines, and making multiple use of the expansion of $\cos(A+B)$, we get
\begin{equation}
 p_{0}(y)=\lim_{\radius \to \infty} \frac{1}{\pi^{d}} \int_{\mathbb{R}^{d}} \prod_{j = 1}^{d}\frac{\sinfunc(\radius(y_{j} - x_{j}))}{y_{j} - x_{j}} p(x)  dx \label{eq:Fourier_integral}
\end{equation}
and
\begin{equation}
\label{eq:multivariate_density_estimator}
\funcest_{n, \radius}(x)=\frac{1}{n\pi^d}\sum_{i=1}^n \prod_{j=1}^d\frac{\sin(R(x-X_{ij}))}{x-X_{ij}},
\end{equation}
where $x = (x_{1}, \ldots, x_{d})$ and $X_{i} = (X_{i1},\ldots, X_{id})$ for all $i$.
So note the natural use of the product of one dimensional Fourier kernels. We call the estimator $\funcest_{n, \radius}$ as \emph{Fourier density estimator}.

%Assume that $X_{1}, \ldots, X_{n} \in \mathbb{R}^{d}$ are i.i.d. samples from the unknown distribution $P_{0}$ admitting the density function $p_{0}$ supported on $\mathcal{X} \subseteq \mathbb{R}^{d}$. We denote the following Fourier density estimator of $p_{0}$:
%\begin{align}
 %   \funcest_{n, \radius}(x) = \frac{1}{n \pi^{d}} \sum_{ i = 1}^{n} \prod_{j = 1}^{d} \frac{\sinfunc(\radius(x_{j} - X_{ij}))}{x_{j} - X_{ij}}, \label{eq:multivariate_density_estimator}
%\end{align}
%where $x = (x_{1}, \ldots, x_{d})$ and $X_{i} = (X_{i1},\ldots, X_{id})$ for all $i$. Here, $\radius > 0$ can be considered as \emph{bandwidth}. If $p_{0}$ is continuous and its Fourier transform $\widehat{p}_{0}$ is in $\mathbb{L}_{1}(\mathbb{R}^{d})$, from Fourier integral theorem we have
%\
%for all $x \in \mathcal{X}$. The above result comes from the following inverse Fourier transform:
%\begin{align}
 %   p_{0}(x) = \frac{1}{(2\pi)^{d}} \int_{\mathbb{R}^{d}} \int_{\mathbb{R}^{d}} p_{0}(t) \cosfunc(s^{\top} (x - t)) dt d s. 
%\end{align}
%Note that, if we only have $\widehat{p}_{0}$ to be in $\mathbb{L}_{1}(\mathbb{R}^{d})$ (without the assumption that $p_{0}$ is continuous), both the Fourier integral theorem and inverse Fourier transform still hold for almost surely $x \in \mathcal{X}$. However, the Fourier inversion in equation~\eqref{eq:Fourier_inverstion} also means that $p_{0}$ is continuous almost surely.

The same basic idea equally applies to nonparametric kernel regression; so suppose we observe $(X_i,Y_i)_{i=1}^n$ such that
$Y_i=m(X_i)+\sigma\epsilon_i$, with $\mbox{E}\,\epsilon = 0$ and $\mbox{Var} \,\epsilon = 1.$
Then, as before, $m$ satisfies equation~(\ref{finthem}), and we can again approximate one side with the following term;
$$m_R(y)=\frac{1}{\pi}\int_{-\infty}^\infty \frac{\sin(R(y-x))}{y-x}\,m(x)\,dx.$$
The Monte Carlo estimate of the right side then yields
$$\widehat{m}_{n,R}(x)=\frac{\sum_{i=1}^n Y_i\,K_R(x-X_i)}{\sum_{i=1}^n K_R(x-X_i)},$$
where $K_R(u)=\sin(Ru)/u$. This estimator can be considered as the Fourier version of the Nadaraya--Watson kernel estimator for nonparametric regression.

Again, the extension to the multivariate case (mutiple predictors) follows along the same lines which led to (\ref{eq:multivariate_density_estimator}). That is,
\begin{equation}\label{multreg}
\widehat{m}_{n,R}(x) =\frac{\sum_{i=1}^n Y_i\,\prod_{j=1}^d K_R(x_j-X_{ij})}
{\sum_{i=1}^n \prod_{j=1}^d K_R(x_j-X_{ij})}.
\end{equation}
There is no need for any setting of a covariance structure between variables.

\paragraph{Contribution.} Motivated by equation (\ref{fitt}), the aim of the paper is to study the Monte Carlo estimators of the integral identities (or approximations once we have set a finite $R$), such as those in equations~(\ref{eq:multivariate_density_estimator}) and (\ref{multreg}). 
%under the settings of estimating multivariate (mixing) density functions, nonparametric regression, and nonparametric mode estimation of multivariate density functions. 
While noting the sufficiency of the product of kernels, we demonstrate that when the data density function is suitably smooth, the mean (integrated) square errors of the Monte Carlo estimators have faster convergence rates than those from standard kernel density estimators. Improved rates for other types of functions is also demonstrated.

\paragraph{Organization.} The paper is organized as follows. In Section~\ref{sec:Fourier_density}, we study the mean integrated square error (MISE) of the Fourier density estimator and its derivatives under various tail conditions of the true density function. Then, we also provide (uniform) confidence interval of the true density function based on the Fourier density estimator. In Section~\ref{sec:mixing_measure_estimation}, we study an application of Fourier integral theorem to estimate mixing density under the deconvolution settings. We further extend the idea of Fourier integral theorem to the nonparametric regression, mode hunting applications, and dependent data in Sections~\ref{subsec:mode_clustering}-\ref{sec:dependent_data}. Illustrations with the proposed Monte Carlo estimators are in Section~\ref{sec:illustrations}. Proofs of key results are in Section~\ref{sec:proof} while the remaining proofs are in Appendix~\ref{sec:proof_remaining_result}. We end the paper with some discussion with future work in Section~\ref{sec:discussion}.

\paragraph{Notation.} For any $n \in \mathbb{N}$, we denote $[n] = \{1, \ldots, n\}$. For any set $\mathcal{X}$, we denote $\text{Diam}(\mathcal{X})$ the diameter of set $\mathcal{X}$. For any vector $x = (x_{1}, \ldots, x_{d}) \in \mathbb{R}^{d}$, we denote 
$$\|x\|_{\max} = \max_{1 \leq i \leq d} \{|x_{1}|, \ldots, |x_{d}|\},$$ the maximal norm of $x$. For any $r \geq 1$ and any set $\mathcal{X}$, we denote $\mathcal{C}^{r}(\mathcal{X})$ the set of functions on $\mathcal{X}$ that have bounded integrable continuous derivatives up to the $r$-th order. For any $x \in \mathbb{R}^{d}$ and subset $A$ of $\mathbb{R}^{d}$, we define $d(x, A) = \inf_{y \in A} \enorm{x - y}$. For any symmetric matrix $M \in \mathbb{R}^{d \times d}$, we denote $\lambda_{i}(M)$ as the $i$-th largest eigenvalue of $M$, i.e., $\lambda_{1}(M) \geq \lambda_{2}(M) \geq \ldots \geq \lambda_{d}(M)$. For any subset $A$ of $\mathbb{R}^{d}$ and $r > 0$, we denote $A \oplus r = \{y: \min_{x \in A} \enorm{x - y} \leq r\}$. The notation $X \overset{p}{\to} Y$ and $X \overset{d}{\to} Y$ respectively mean $X$ converge to $Y$ in probability and distribution. For any sequence $a_{n}$ and $b_{n}$, the notation $a_{n} = \mathcal{O}(b_{n})$ means that $a_{n} \leq C b_{n}$ for all $n \geq 1$ where $C$ is some universal constant. Furthermore, the notation $a_{n} = o(b_{n})$ means that $a_{n}/ b_{n} \to 0$ as $n \to \infty$. Finally, we denoted by $p_0$ a true (density) function.

\section{Fourier density estimator}
\label{sec:Fourier_density}

Recall that, we assume $X_{1}, \ldots, X_{n} \in \mathcal{X} \subset \mathbb{R}^{d}$ are an i.i.d. sample from $p_{0}$ and we would like to estimate the density function $p_{0}$ based on the Fourier density estimator $\funcest_{n, \radius}$. Given equation~\eqref{eq:Fourier_integral}, we know that when $\radius$ goes to infinity, the bias of $\funcest_{n, \radius}(x)$ goes to 0. However, we would like to investigate the vanishing rate of the bias. To do that, we first define two important tail behaviors on Fourier transform of density function $p_{0}$, which serve as sufficient conditions for the Fourier transform to be integrable and to obtain the vanishing rate of the bias.
\begin{definition}\label{def:tail_Fourier} 
\noindent
(1) We say that $p_{0}$ is upper-supersmooth (lower-supersmooth) of order $\alpha$ if there exist universal constants $C, C', C_{1}, C_{2}$ such that as long as we have the following inequalities for  $x \in \mathbb{R}^{d}$:
\begin{align*}
\text{Upper-supersmooth:} \quad \quad \abss{ \widehat{p_{0}}(x)} & \leq C \exp \parenth{ -C_{1} \parenth{ \sum_{j = 1}^{d} |x_{j}|^{\alpha}} }, \\
\text{Lower-supersmooth:}  \quad \quad \abss{ \widehat{p_{0}}(x)} & \geq  C' \exp \parenth{ -C_{2} \parenth{ \sum_{j = 1}^{d} |x_{j}|^{\alpha}} }.
\end{align*}
\noindent
(2) The density $p_{0}$ is upper-ordinary smooth (lower-ordinary smooth) of order $\beta$ if there exist universal constants $c, c'$ such that for $x \in \mathbb{R}^{d}$, we have
\begin{align*}
    \text{Upper-ordinary smooth:} \quad \quad \abss{ \widehat{p_{0}}(x)} & \leq  c \cdot \prod_{j = 1}^{d}\frac{1}{(1 + |x_{j}|^{\beta})}, \\
    \text{Lower-ordinary smooth:} \quad \quad  \abss{ \widehat{p_{0}}(x)} & \geq c' \cdot \prod_{j = 1}^{d}\frac{1}{(1 + |x_{j}|^{\beta})}
\end{align*}
\end{definition}

\noindent
Popular examples of upper-and lower-supersmooth densities include multivariate Gaussian, multivariate Cauchy distributions, and their mixtures.  The examples of upper-ordinary smooth densities include continuous density functions that have continuous and integrable partial derivatives or product of univariate Laplace distributions. For the lower-ordinary smooth densities, the examples include multivariate Laplace distribution and univariate Beta distribution. Finally, we would like to note that under the univariate setting of the density $p_{0}$, we can slightly relax the upper-ordinary smooth condition in Definition~\ref{def:tail_Fourier} as follows;
%\begin{align*}
   $ |\widehat{p_{0}}(x)| \leq c/|x|^{\beta}$
%\end{align*}
for almost all $x \in \mathbb{R}$. This relaxation allows the upper-ordinary smooth definition to cover more popular univariate distributions, such as Beta distribution. Later, our results for upper-ordinary smooth univariate settings can be understood to also hold under this relaxation as well. 

\subsection{Risk analysis with Fourier density estimator}
Based on the smoothness definitions of $p_{0}$, we have the following result regarding the bias and variance of the Fourier density estimator $\funcest_{n, \radius}$:
\begin{theorem} \label{theorem:bias_variance_Fourier_density}
(a) Assume that $p_{0}$ is an upper--supersmooth density function of order $\alpha > 0$ and $\|p_{0}\|_{\infty} < \infty$. Then, there exist universal constants $C$ and $C'$ such that while $R \geq C'$, for almost all $x$ we find that
\begin{align*}
    \abss{\Exs \brackets{\funcest_{n, \radius}(x)} - p_{{0}}(x)}
    & \leq C R^{\max \{1 - \alpha, 0\}} \exp \parenth{-C_{1} \radius^{\alpha} }, \\
    \var \brackets{\funcest_{n, \radius}(x)} & \leq \frac{\|p_{0}\|_{\infty}}{\pi^{d}} \cdot \frac{R^{d}}{n},
\end{align*}
where $C_{1}$ is the universal constant associated with the supersmooth density function $p_{0}$ from Definition~\ref{def:tail_Fourier}.

\noindent
(b) Assume that $p_{0}$ is an upper--ordinary smooth density function of order $\beta > 1$ and $\|p_{0}\|_{\infty} < \infty$. Then, there exists a universal constants $c$ such that for almost all $x$ we obtain
\begin{align*}
    \abss{\Exs \brackets{\funcest_{n, \radius}(x)} - p_{{0}}(x)}
    & \leq \frac{c}{R^{\beta - 1}}, \\
    \var \brackets{\funcest_{n, \radius}(x)} & \leq \frac{\|p_{0}\|_{\infty}}{\pi^{d}}  \cdot \frac{R^{d}}{n}.
\end{align*}
\end{theorem}

\vspace{0.1in}
\noindent
The proof of Theorem~\ref{theorem:bias_variance_Fourier_density} is given in Section~\ref{subsec:proof:theorem:bias_variance_Fourier_density}.
Given the result of Theorem~\ref{theorem:bias_variance_Fourier_density}, we have the following upper bound on the mean integrated squared errors (MISE) of the Fourier density estimator $\funcest_{n, \radius}$:

\noindent
(i) When $p_{0}$ is an upper--supersmooth density function of order $\alpha > 0$, we have
\begin{align*}
    \text{MISE}(\funcest_{n, \radius}) & = \int \brackets{ \parenth{ \Exs \brackets{\funcest_{n, \radius}(x)} - p_{0}(x)}^2 + \var \brackets{\funcest_{n, \radius}(x)}} dx  \\
    & \leq C^2 \radius^{\max \{2 - 2\alpha, 0\}} \exp \parenth{-2 C_{1} \radius^{\alpha} } + \frac{\|p_{0}\|_{\infty}}{\pi^{d}} \cdot \frac{\radius^{d}}{n},
\end{align*}
where $C$ and $C_{1}$ are the constants in part (a) of Theorem~\ref{theorem:bias_variance_Fourier_density}. The choice of $R$ that minimizes the upper bound of MSE is the solution of the equation
\begin{align*}
    C^2 \radius^{\max \{2 - 2\alpha, 0\}} \exp(- 2 C_{1} \radius^{\alpha}) = \frac{\|p_{0}\|_{\infty}}{\pi^{d}} \cdot \frac{\radius^{d}}{n}.
\end{align*}
Therefore, we can choose $\radius$ such that $2C_{1}\radius^{\alpha} = \log n$. With this choice of $R$, we have $$\text{MISE}(\funcest_{n, \radius}(x)) \leq \left(C^2 + \frac{\|p_{0}\|_{\infty}}{\pi^{d}}\right) \frac{(\log n)^{\max \{d/ \alpha, 2/\alpha - 2 \}}}{ n},$$ which is better than the well-known MISE rate $n^{-4/(4+d)}$ for the kernel density estimator (KDE), when the density function $p_0$ has bounded second derivatives~\citep{Tsy09}. 

\noindent
(ii) When $p_{0}$ is an upper--ordinary smooth density function of order $\beta > 1$, we find that
\begin{align*}
    \text{MISE}(\funcest_{n, \radius}(x)) \leq \frac{c^2}{R^{2(\beta - 1)}} + \frac{\|p_{0}\|_{\infty}}{\pi^{d}} \cdot \frac{R^{d}}{n},
\end{align*}
where $c$ is the constant given in part (b) of Theorem~\ref{theorem:bias_variance_Fourier_density}. Hence, by choosing $\radius$ such that $R^{d + 2(\beta - 1)} = c^2 \pi^{d} n/ \|p_{0}\|_{\infty}$, we obtain $\text{MISE}(\funcest_{n, \radius}(x)) \leq \bar{C} n^{-\frac{2(\beta - 1)}{2(\beta - 1) + d}}$, where $\bar{C}$ is a positive constant depending on $c$ and $\|p_{0}\|_{\infty}$. As long as $\beta > 3$, the MISE rate of $\funcest_{n, \radius}$ is better than the rate $n^{-4/(4+d)}$ of the KDE, when the density function has bounded second derivatives.
\subsection{Concentration of Fourier density estimator}
\label{subsec:concentration_Fourier_density_estimator}
In this section, we first provide concentration bounds for Fourier density estimator $\funcest_{n, \radius}(x)$ under various smoothness assumptions of the true density function $p_{0}$. 
%at almost surely $x \in \mathcal{X}$.
\begin{proposition} \label{theorem:high_prob_Fourier_density}
For almost all $x \in \mathcal{X}$, there exist universal constants $C$ and $c$ such that:

\noindent
(a) If $p_{0}$ is an upper--supersmooth density function of order $\alpha > 0$ and $\|p_{0}\|_{\infty} < \infty$, then for any $R \geq C'$ where $C'$ is some universal constant, we obtain 
\begin{align*}
    \Prob \parenth{\abss{ \funcest_{n, \radius}(x) - p_{0}(x)} \geq C \parenth{ R^{\max \{1 - \alpha, 0\}} \exp \parenth{-C_{1} \radius^{\alpha} } + \sqrt{\frac{\radius^{d} \log(2/ \delta)}{n}} }} \leq \delta. 
\end{align*}
Here, $C_{1}$ is universal constant given in part (a) of Theorem~\ref{theorem:bias_variance_Fourier_density}.

\noindent
(b) If $p_{0}$ is an upper--ordinary smooth density function of order $\beta > 1$ and $\|p_{0}\|_{\infty} < \infty$, then 
\begin{align*}
    \Prob \parenth{\abss{ \funcest_{n, \radius}(x) - p_{0}(x)} \geq c \parenth{R^{1 - \beta} + \sqrt{\frac{\radius^{d} \log(2/ \delta)}{n}}}} \leq \delta.
\end{align*}
\end{proposition}
\begin{proof}
An application of the triangle inequality yields
\begin{align*}
    \abss{ \funcest_{n, \radius}(x) - p_{0}(x)} \leq \abss{\funcest_{n, \radius}(x) - \Exs \brackets{\funcest_{n, \radius}(x)}} + \abss{\Exs \brackets{\funcest_{n, \radius}(x)} - p_{0}(x)}.
\end{align*}
Denote $Y_{i} = \frac{1}{\pi^{d}}\prod_{j = 1}^{d} \frac{\sin(\radius(x_{j} - X_{ij})}{x_{j} - X_{ij}}$ for all $i \in [n]$. It is clear that $|Y_{i}| \leq \radius^{d}$ for all $i \in [n]$ and $\var(Y_{i}) \leq C \radius^{d}$ (cf. Theorem~\ref{theorem:bias_variance_Fourier_density}) where $C > 0$ is some universal constant. For any $t \in (0, C]$, an application of Bernstein's inequality shows that
\begin{align*}
    \Prob \parenth{\abss{\frac{1}{n} \sum_{i = 1}^{n} Y_{i} - \Exs \brackets{Y_{1}}} \geq t} \leq 2 \exp \parenth{- \frac{n t^2}{2 C \radius^{d} + 2 \radius^{d} t/ 3}}.
\end{align*}
By choosing $t = \bar{C} \sqrt{\radius^{d} \log(2/ \delta)/n}$, where $\bar{C}$ is some universal constant, we find that
\begin{align*}
    \Prob \parenth{\abss{\frac{1}{n} \sum_{i = 1}^{n} Y_{i} - \Exs \brackets{Y_{1}}} \geq t} \leq \delta.
\end{align*}
Combining the above probability bound with the upper bounds of $\abss{\Exs \brackets{\funcest_{n, \radius}(x)} - p_{0}(x)}$ from Theorem~\ref{theorem:bias_variance_Fourier_density}, we reach the conclusion of the theorem.
\end{proof}
The results of Proposition~\ref{theorem:high_prob_Fourier_density} only hold for point--wise $x \in \mathcal{X}$. In certain applications, such as mode estimation, it is desirable to establish the uniform concentration bound for the Fourier density estimator $\funcest_{n, \radius}$, namely, $\sup_{x \in \mathcal{X}} |\funcest_{n, \radius}(x) - p_{0}(x)|$. Our next result provides such a uniform concentration bound when $\mathcal{X}$ is bounded and the density function $p_{0}$ is continuous. Note, the assumption that $p_{0}$ is continuous is to guarantee that the bounds of the bias in Theorem~\ref{theorem:bias_variance_Fourier_density} hold for all $x \in \mathcal{X}$.
\begin{theorem} \label{theorem:uniform_bound_Fourier_density}
Assume that $\mathcal{X}$ is a bounded subset of $\mathbb{R}^{d}$. Then, there exist universal constants $C$ and $c$ such that the following holds:

\noindent
(a) When $p_{0}$ is a continuous upper--supersmooth density function of order $\alpha > 0$ and $\|p_{0}\|_{\infty} < \infty$, for any $\radius \geq C'$ where $C'$ is some universal constant we have 
\begin{align*}
    \Prob \parenth{ \sup_{x \in \mathcal{X}} \abss{ \funcest_{n, \radius}(x) - p_{0}(x)} \geq C \parenth{ R^{\max \{1 - \alpha, 0\}} \exp \parenth{-C_{1} \radius^{\alpha} } + \sqrt{\frac{\radius^{d} \log \radius \parenth{\log(2/ \delta)} }{n}}} } \leq \delta. 
\end{align*}
Here, $C_{1}$ is universal constant given in part (a) of Theorem~\ref{theorem:bias_variance_Fourier_density}.

\noindent
(b) When $p_{0}$ is a continuous upper--ordinary smooth density function of order $\beta > 1$ and $\|p_{0}\|_{\infty} < \infty$, we obtain
\begin{align*}
    \Prob \parenth{ \sup_{x \in \mathcal{X}} \abss{ \funcest_{n, \radius}(x) - p_{0}(x)} \geq c \parenth{ R^{1 - \beta} + \bar{C} \sqrt{\frac{\radius^{d} \log \radius (\log(2/ \delta))}{n}}} } \leq \delta.
\end{align*}
\end{theorem}
\begin{proof}
By the triangle inequality, we have
\begin{align*}
    \sup_{x \in \mathcal{X}} \abss{ \funcest_{n, \radius}(x) - p_{0}(x)} \leq \sup_{x \in \mathcal{X}} \abss{ \funcest_{n, \radius}(x) - \Exs \brackets{\funcest_{n, \radius}(x)}} + \sup_{x \in \mathcal{X}} \abss{ \Exs \brackets{\funcest_{n, \radius}(x)} - p_{0}(x)}.
\end{align*}
In order to bound $\sup_{x \in \mathcal{X}} \abss{ \funcest_{n, \radius}(x) - \Exs \brackets{\funcest_{n, \radius}(x)}}$, we use Bernstein's inequality along with the bracketing entropy under $\mathbb{L}_{1}$ norm of the functions in the space $\mathcal{X}$~\citep{Wainwright_nonasymptotic}. In particular, by denoting $Y_{i} = \frac{1}{\pi^{d}} \prod_{j = 1}^{d} \frac{\sin(\radius(x_{j} - X_{ij})}{x_{j} - X_{ij}}$ for all $i \in [n]$, we have $|Y_{i}| \leq \radius^{d}$ and $\Exs(|Y_{i}|) \leq 1$ for all $i \in [n]$. Therefore, when $t \leq 2 C \radius^{d}$, we find that
\begin{align*}
    \Prob \parenth{\sup_{x \in \mathcal{X}} \abss{ \funcest_{n, \radius}(x) - p_{0}(x)} > t} \leq 4 \mathcal{N}_{[]} \parenth{ t/8, \mathcal{F}', \mathbb{L}_{1}(P)} \exp \parenth{ - \frac{96 n t^2}{76 \radius^{d}}},
\end{align*}
where $\mathcal{F}' = \{f_{x}: \mathbb{R}^{d} \to \mathbb{R}: f_{x}(t) = \frac{1}{\pi^{d}} \prod_{i = 1}^{d} \frac{\sin(\radius(x_{i} - t_{i}))}{x_{i} - t_{i}} \ \text{for all} \ x \in \mathcal{X}, t \in \mathbb{R}^{d} \}$ and $\mathcal{N}_{[]} \parenth{ t/8, \mathcal{F}', \mathbb{L}_{1}(P)}$ is the bracketing number of the functional space $\mathcal{F}'$ under $\mathbb{L}_{1}(P)$. For any functions $f_{x_{1}}$ and  $f_{x_{2}}$ in $\mathcal{F}$, we can check that
\begin{align*}
    \abss{ f_{x_{1}}(y) - f_{x_{2}}(y)} \leq d \radius^{d + 1} \enorm{x_{1} - x_{2}},
\end{align*}
for all $y \in \mathbb{R}^{d}$. Since $\mathcal{X}$ is a bounded subset of $\mathbb{R}^{d}$, we obtain that
\begin{align*}
    \mathcal{N}_{[]} \parenth{ t/8, \mathcal{F}', \mathbb{L}_{1}(P)} \leq \parenth{ \frac{ 4 d \sqrt{d} \cdot \text{Diam}(\mathcal{X}) \radius^{d + 1}}{t}}^{d}.
\end{align*}
Putting the above results together, by choosing 
$$t = \bar{C} \sqrt{\radius^{d} \parenth{\log(2/ \delta) + d (d + 1) \log \radius + d (\log d + \text{Diam}(\mathcal{X})}/n},$$ 
where $\bar{C}$ is some universal constant, we have
\begin{align*}
    \Prob \parenth{\sup_{x \in \mathcal{X}} \abss{ \funcest_{n, \radius}(x) - p_{0}(x)} > t} \leq \delta.
\end{align*}
The above uniform concentration bound of $\sup_{x \in \mathcal{X}} \abss{ \funcest_{n, \radius}(x) - \Exs \brackets{\funcest_{n, \radius}(x)}}$ and the upper bounds of $\sup_{x \in \mathcal{X}} \abss{ \Exs \brackets{\funcest_{n, \radius}(x)} - p_{0}(x)}$ in Theorem~\ref{theorem:bias_variance_Fourier_density} lead to the conclusion of the theorem.
\end{proof}

\subsection{Derivatives of Fourier density estimator}
In this section, we provide the risk analysis for the derivatives of the Fourier density estimator $\funcest_{n, \radius}$. For any $r \geq 1$, the mean integrated squared errors of the $r$-th derivatives of the Fourier density estimators are defined as follows:
\begin{align*}
    \text{MISE}(\nabla^{r} \funcest_{n, \radius}) & = \int \Exs \brackets{ \enorm{ \nabla^{r}\funcest_{n, \radius}(x) - \nabla^{r} p_{0}(x)}^2} dx \\
    & = \int \enorm{\Exs \brackets{\nabla^{r} \funcest_{n, \radius}(x)} - \nabla^{r} p_{0}(x)}^2 dx + \int \Exs \brackets{ \enorm{ \nabla^{r}\funcest_{n, \radius}(x) - \Exs \brackets{\nabla^{r}\funcest_{n, \radius}(x)}}^2} dx.
\end{align*}
The first term can be thought as mean-squared bias while the second term can be thought of as the mean-squared variance. The following result provides upper bounds for the mean-squared bias and variance of $\nabla^{r} \funcest_{n, \radius}(x)$ for estimating $\nabla^{r} p_{0}(x)$.
\begin{theorem} 
\label{theorem:bias_variance_Fourier_density_derivatives}
For any given $r \geq 1$, assume that $p_{0} \in \mathcal{C}^{r}(\mathcal{X})$. Then, the following holds:

\noindent
(a) When $p_{0}$ is an upper--supersmooth density function of order $\alpha > 0$, there exist universal constants $\{C'_{i}\}_{i = 1}^{r}$ and $\{\bar{C}_{i}\}_{i = 1}^{r}$  such that while $R \geq C'$, where $C'$ is some universal constant and $1 \leq i \leq r$, we find that
\begin{align*}
    \sup_{x \in \mathcal{X}} \| \Exs \brackets{\nabla^{i} \funcest_{n, \radius}(x)} - \nabla^{i} p_{{0}}(x)\|_{\max}
    & \leq C_{i}' R^{\max \{1 + i - \alpha, 0\}} \exp \parenth{-C_{1} \radius^{\alpha} }, \\
    \sup_{x \in \mathcal{X}} \Exs \brackets{ \enorm{ \nabla^{i}\funcest_{n, \radius}(x) - \Exs \brackets{\nabla^{i}\funcest_{n, \radius}(x)}}^2} & \leq \bar{C}_{i} \cdot \frac{R^{2i + d}}{n},
\end{align*}
where $C_{1}$ is the universal constant associated with the supersmooth density function $p_{0}$ from Definition~\ref{def:tail_Fourier}.

\noindent
(b) When $p_{0}$ is an upper--ordinary smooth density function of order $\beta > 1 + r$, there exist universal constants $\{c_{i}\}_{i = 1}^{r}$ such that for any $1 \leq i \leq r$ we obtain
\begin{align*}
    \sup_{x \in \mathcal{X}} \| \Exs \brackets{\nabla^{i} \funcest_{n, \radius}(x)} - \nabla^{i} p_{{0}}(x)\|_{\max}
    & \leq \frac{c_{i}}{\radius^{\beta - (i + 1)}}, \\
    \sup_{x \in \mathcal{X}} \Exs \brackets{ \enorm{ \nabla^{i}\funcest_{n, \radius}(x) - \Exs \brackets{\nabla^{i}\funcest_{n, \radius}(x)}}^2} & \leq \bar{C}_{i} \cdot \frac{R^{2i + d}}{n}.
\end{align*}
\end{theorem}
\noindent
The proof of Theorem~\ref{theorem:bias_variance_Fourier_density_derivatives} is in Section~\ref{subsec:proof:theorem:bias_variance_Fourier_density_derivatives}. Given the results in Theorem~\ref{theorem:bias_variance_Fourier_density_derivatives}, we obtain the following results with the MISE of $\nabla^{r} \funcest_{n, \radius}$:

\noindent
(i) When $p_{0}$ is an upper--supersmooth density function of order $\alpha > 0$, the result in part (a) in Theorem~\ref{theorem:bias_variance_Fourier_density_derivatives} demonstrates that
    \begin{align*}
        \text{MISE}(\nabla^{r} \funcest_{n, \radius}) \leq (C_{r}')^2 \radius^{\max \{2(1 + r - \alpha), 0\}} \exp \parenth{ -2 C_{1} \radius^{\alpha}} + \bar{C}_{r} \radius^{2 r + d}/n,
    \end{align*}
    where $C_{r}'$ and $\bar{C}_{r}$ are given constants in part (a). This upper bound suggests that we can choose $\radius$ such that $2 C_{1} \radius^{\alpha} = \log n$. Then, we have
    \begin{align*}
        \text{MISE}(\nabla^{r} \funcest_{n, \radius}) \leq C\,n^{-1 }\,(\log n)^{\max \{(d + 2r)/ \alpha, 2(1 + r - \alpha)/ \alpha \}},
    \end{align*}
    where $C$ is some universal constant.

\noindent
(ii) When $p_{0}$ is an upper--ordinary smooth density function of order $\beta > 1 + r$, we have
    \begin{align*}
        \text{MISE}(\nabla^{r} \funcest_{n, \radius}) \leq \frac{c_{r}^2}{R^{2(\beta - (r + 1))}} + \bar{C}_{r} \radius^{2 r + d}/n,
    \end{align*}
    where $c_{r}$ is given constant in part (b). By choosing $\radius$ such that $\radius^{d + 2( \beta - 2)} = c_{r}^2 n/ \bar{C}_{r}$, we obtain $\text{MISE}(\nabla^{r} \funcest_{n, \radius}) \leq c n^{-\frac{2(\beta - r - 1)}{d + 2(\beta - 1)}}$, where $c$ is some universal constant. When $\beta > r + 3$, then the MISE rate of the $r$-th order derivatives of the Fourier density estimator is better than the MISE rate $n^{-4/(d + 2r + 4)}$ of the KDE estimator when the density function $p_{0} \in \mathcal{C}^{r}(\mathcal{X})$~\citep{Chacon_2011}.
%We would like to remark 

Thus, we have provided the uniform upper bounds for the difference between $\Exs \brackets{\nabla^{r} \funcest_{n, \radius}(x)}$ and $\nabla^{r} p_{{0}}(x)$. In certain applications, such as mode estimation (cf. Section~\ref{subsec:mode_clustering}), it is also important to understand the concentration bounds of $\nabla^{r} \funcest_{n, \radius}(x)$ around $\nabla^{r} p_{{0}}(x)$ uniformly for all $x \in \mathcal{X}$. The following result provides these bounds when $\mathcal{X}$ is a bounded subset of $\mathbb{R}^{d}$.
\begin{theorem} \label{theorem:uniform_bound_derivatives}
For any given $r \geq 1$, assume that $p_{0} \in \mathcal{C}^{r}(\mathcal{X})$ and $\mathcal{X}$ is a bounded subset of $\mathbb{R}^{d}$. Then, there exist universal constants $C$ and $c$ such that the following holds:

\noindent
(a) When $p_{0}$ is an upper--supersmooth density function of order $\alpha > 0$, as long as $R \geq C'$ where $C'$ is some universal constant and $1 \leq i \leq r$, we find that
\begin{align*}
    \Prob \biggr(\sup_{x \in \mathcal{X}} \left\|\nabla^{i} \funcest_{n, \radius}(x) - \nabla^{i} p_{0}(x)\right\|_{\max} & \geq C \biggr( R^{\max \{1 + i - \alpha, 0\}} \exp \parenth{-C_{1} \radius^{\alpha} } \nonumber \\
    & \hspace{6 em} + \sqrt{\frac{\radius^{(d + 2i)} \log \radius \parenth{\log(2/ \delta)}}{n}} \biggr) \biggr) \leq \delta,
\end{align*}
where $C_{1}$ is the universal constant in part (a) of Theorem~\ref{theorem:bias_variance_Fourier_density_derivatives}.

\noindent
(b) When $p_{0}$ is an upper--ordinary smooth density function of order $\beta > r + 1$, for any $1 \leq i \leq 2$ we obtain
\begin{align*}
    \Prob \parenth{\sup_{x \in \mathcal{X}} \left\|\nabla^{i} \funcest_{n, \radius}(x) - \nabla^{i} p_{0}(x)\right\|_{\max} \geq c \parenth{ \radius^{- \beta + (i + 1)} + \sqrt{\frac{\radius^{(d + 2i)} \log \radius \parenth{\log(2/ \delta)}}{n}}} } \leq \delta.
\end{align*}
\end{theorem}
\noindent
The proof of Theorem~\ref{theorem:uniform_bound_derivatives} is in Section~\ref{subsec:proof:theorem:uniform_bound_derivatives}.

Based on the result of Theorem~\ref{theorem:uniform_bound_derivatives}, we can choose the radius $\radius$ similar to those in the discussion after Theorem~\ref{theorem:bias_variance_Fourier_density_derivatives} and obtain the similar uniform upper bounds for the concentration of $\nabla^{r} \funcest_{n, \radius}(x)$ around $\nabla^{r} p_{{0}}(x)$ for any $r \in \mathbb{N}$.

\subsection{Confidence interval and band of Fourier density estimator}
\label{sec:uniform_confidence_Fourier_density}
In this section, we study the confidence interval and band of $p_{0}$ based on the Fourier density estimator. 
\subsubsection{Confidence interval}
\label{sec:point_wise_CI}
In order to establish the point-wise confidence interval for $p_{0}(x)$ for each $x \in \mathcal{X}$, we first study the asymptotic property of the following term as $n \to \infty$:
\begin{align}
    \frac{\funcest_{n, \radius}(x) - p_{0}(x)}{\sqrt{\var(\funcest_{n, \radius}(x))}} = \frac{\funcest_{n, \radius}(x) - \Exs \brackets{ \funcest_{n, \radius}(x)}}{\sqrt{\var(\funcest_{n, \radius}(x))}} + \frac{\Exs \brackets{ \funcest_{n, \radius}(x)} - p_{0}(x)}{\sqrt{\var(\funcest_{n, \radius}(x))}} : = A_{1} + A_{2}. \label{eq:pointwise_CI_Fourier}
\end{align}
For the term $A_{1}$, from the central limit theorem, as $n \to \infty$ we obtain
\begin{align}
    A_{1} = \frac{\sqrt{n} \parenth{\funcest_{n, \radius}(x) - \Exs \brackets{\funcest_{n, \radius}(x)}}}{\sqrt{\var(Y)}} \overset{d}{\to} \mathcal{N}(0, 1), \label{eq:clt_fourier_density_estimator}
\end{align}
where $Y = \frac{1}{\pi^{d}} \prod_{j = 1}^{d} \frac{\sinfunc(\radius(x_{j} - X_{.j}))}{x_{j} - X_{.j}}$ and $X = (X_{.1},\ldots,X_{.d}) \sim p_{0}$. From the result of Theorem~\ref{theorem:bias_variance_Fourier_density}, $\var(Y) \to 0$ as $\radius \to \infty$. The non-asymptotic upper bound on the variance of $Y$ in Theorem~\ref{theorem:bias_variance_Fourier_density} provides a tight dependence on $\radius$ but not on other constants. To obtain a tight asymptotic behavior of $\var(Y)$, we assume that $p_{0} \in \mathcal{C}^{1}(\mathcal{X})$ and $\mathcal{X}$ is a bounded subset of $\mathbb{R}^{d}$. Then, simple algebra shows that $\Exs^2(Y) \leq \|p_{0}\|_{\infty}^2$. Furthermore, from the Taylor expansion up to first order we have
\begin{align*}
 %$$  
 \Exs(Y^2)  = \frac{\radius^{d}}{\pi^{2d}} \int_{\mathbb{R}^{d}} \prod_{j = 1}^{d} \frac{\sinfunc^2(t_{j})}{t_{j}^2} p_{0}\parenth{x - \frac{t}{\radius}} dt 
 &   = \frac{\radius^{d}}{\pi^{2d}} \int_{\mathbb{R}^{d}} \prod_{j = 1}^{d} \frac{\sinfunc^2(t_{j})}{t_{j}^2} \parenth{p_{0}(x) + \mathcal{O}\parenth{\frac{t}{\radius}}} dt   \\
 & = \frac{p_{0}(x) \radius^{d}}{\pi^{d}} + \mathcal{O}(\radius^{d - 1}).
 %$$   
\end{align*}
Collecting the above results, we find that $\lim_{\radius \to \infty} \var(Y)/\radius^{d} = p_{0}(x)/\pi^{d}$. Combining this result with the central limit theorem result in equation~\eqref{eq:clt_fourier_density_estimator}, when $p_{0} \in \mathcal{C}^1(\mathcal{X})$ and $\radius \to \infty$ we obtain that
\begin{align}
    \sqrt{\frac{n}{\radius^{d}}} \parenth{\funcest_{n, \radius}(x) - \Exs \brackets{\funcest_{n, \radius}(x)}} \overset{d}{\to} \mathcal{N} \parenth{ 0, \frac{p_{0}(x)}{\pi^{d}}}. \label{eq:pointwise_clt_fourier_density_estimator}
\end{align}
%\paragraph{Supersmooth setting of $p_{0}$:} 
\noindent
For the term $A_{2}$, when $p_{0}$ is an upper--supersmooth density function of order $\alpha > 0$, the result of part (a) of Theorem~\ref{theorem:bias_variance_Fourier_density} shows that
\begin{align*}
     A_{2} \leq \frac{C \cdot \radius^{\max \{1 - \alpha, 0\}} \exp \parenth{-C_{1} \radius^{\alpha} }}{\var(\funcest_{n, \radius}(x))},
\end{align*}
where $C$ and $C_{1}$ are some constants. By choosing the radius $\radius$ such that $2 C_{1}\radius^{\alpha} = \log n$ and the MISE rate of $\funcest_{n,\radius}$ is at the order $n^{-1}$ (up to some logarithmic factor), we have $A_{2} \to 0$ as $n \to \infty$. Putting the above results together, we obtain the following asymptotic result of equation~\eqref{eq:pointwise_CI_Fourier} when $p_{0}$ is a supersmooth density function.

\begin{proposition}
\label{theorem:CLT_point_wise_supersmooth_Fourier_density}
Assume that $p_{0}$ is an upper--supersmooth density function of order $\alpha > 0$ and $p_{0} \in \mathcal{C}^{1}(\mathcal{X})$ where $\mathcal{X}$ is a bounded subset of $\mathbb{R}^{d}$. Then, for each $x \in \mathcal{X}$, by choosing the radius $\radius$ such that $\radius^{\alpha} = C \log n$ where $C$ is some universal constant, as $n \to \infty$ we have
\begin{align*}
    \sqrt{\frac{n}{\radius^{d}}} \parenth{\funcest_{n, \radius}(x) - p_{0}(x)} \overset{d}{\to} \mathcal{N} \parenth{ 0, \frac{p_{0}(x)}{\pi^{d}}}.
\end{align*}
\end{proposition}

\noindent
The result of Proposition~\ref{theorem:CLT_point_wise_supersmooth_Fourier_density} suggests that we can choose the radius $\radius$ such that the MISE rate of $\funcest_{n, \radius}$ obtains the best possible rate $n^{-1}$ (up to some logarithmic factor) and no bias term in the limit of $\funcest_{n, \radius}$ to $p_{0}(x)$. It is different from the standard kernel density estimator when we essentially need to undersmooth the estimator, i.e., we choose the bandwidth to trade-off the MISE rate and the bias term~\citep{Wand_kernel}. It shows the benefit of using Fourier density estimator for estimating the density function $p_{0}$ when it is upper--supersmooth.

Based on the result of Theorem~\ref{theorem:CLT_point_wise_supersmooth_Fourier_density}, for any $\tau \in (0,1)$ we can construct the $1 - \tau$ point-wise confidence interval for $p_{0}(x)$ as follows:
\begin{align*}
    \funcest_{n, \radius}(x) \pm z_{1 - \tau/2} \sqrt{\frac{\radius^{d} p_{0}(x)}{n \pi^{d}}},
\end{align*}
where $z_{1 - \tau/2}$ stands for critical value of standard Gaussian distribution at the tail area $\tau/2$. Note that, since $p_{0}(x)$ is generally unknown, we can replace the above confidence interval by the following plug-in confidence interval:
\begin{align}
    \text{CI}_{1 - \tau}(x) = \funcest_{n, \radius}(x) \pm z_{1 - \tau/2} \sqrt{\frac{\radius^{d} \max \{\funcest_{n, \radius}(x), 0\}}{n \pi^{d}}}. \label{eq:point_wise_CI_supersmooth}
\end{align}
Since $\max \{\funcest_{n, \radius}(x), 0\}$ is a consistent estimate of $p_{0}(x)$ as $\radius^{\alpha} = \mathcal{O}(\log n)$ and $n \to \infty$, the confidence interval $\text{CI}_{1 - \tau}(x)$ in equation~\eqref{eq:point_wise_CI_supersmooth} satisfies
\begin{align*}
    \lim_{n \to \infty} \Prob \parenth{p_{0}(x) \in \text{CI}_{1 - \tau}(x)} \geq 1 - \tau.
\end{align*}
Therefore, $\text{CI}_{1 - \tau}(x)$ is also a valid $1 - \tau$ confidence interval of $p_{0}(x)$ for each $x \in \mathcal{X}$.
%\paragraph{Ordinary smooth setting of $p_{0}$:} 

When $p_{0}$ is an upper--ordinary smooth density function of order $\beta > 1$, the result of part (b) of Theorem~\ref{theorem:bias_variance_Fourier_density} leads to the following bound of $A_{2}$:
\begin{align*}
    A_{2} \leq \frac{c}{\radius^{\beta - 1} \var(\funcest_{n, \radius}(x))},
\end{align*}
where $c$ is some universal constant. If we choose the optimal radius $\radius^{d + 2(\beta - 1)} = \mathcal{O}(n)$ such that the MISE of $\funcest_{n, \radius}(x)$ obtains the best possible rate (cf. the discussion after Theorem~\ref{theorem:bias_variance_Fourier_density}), $A_{2}$ goes to $\bar{c}(x)$ as $n \to \infty$ where $\bar{c}(x)$ is some universal constant depending on $p_{0}(x)$ and can be possibly different from 0. Plugging this result and the result~\eqref{eq:pointwise_clt_fourier_density_estimator} into equation~\eqref{eq:pointwise_CI_Fourier}, as $\radius^{d + 2(\beta - 1)} = \mathcal{O}(n)$ we have
\begin{align*}
    \sqrt{\frac{n}{\radius^{d}}} \parenth{\funcest_{n, \radius}(x) - p_{0}(x)} \overset{d}{\to} \mathcal{N}\left(c(x), \frac{p_{0}(x)}{\pi^{d}}\right). 
\end{align*}

\noindent
Therefore, under the upper-ordinary smooth setting of $p_{0}$, we need to undersmooth the Fourier density estimator, i.e., we choose $\radius^{d + 2(\beta - 1)} = o(n)$, as the standard kernel density estimator to make sure that $c(x) = 0$. It can be undesirable as the MISE rate is not optimal if we choose sub-optimal radius, which means that the Fourier density estimator becomes less precise. As a consequence, under this case of $p_{0}$, we may only use the asymptotic result for $A_{1}$ in equation~\eqref{eq:pointwise_clt_fourier_density_estimator} to obtain a point-wise confidence interval for the expectation of $\funcest_{n, \radius}$.

\subsubsection{Confidence band}
\label{sec:uniform_CI}
In this section, we establish the confidence band of $p_{0}$  based on the bootstrap approach, which has been widely employed to construct the confidence band based on the standard kernel density estimator; see Section 3 in~\citep{Chen_2017} for a summary of this method. We will only focus on the upper--supersmooth setting of $p_{0}$ since the argument is similar for the upper-ordinary smooth case of $p_{0}$. We first define a Gaussian process used to approximate the uniform error $\sup_{x \in \mathcal{X}} \abss{\funcest_{n, \radius}(x) - \Exs \brackets{ \funcest_{n, \radius}(x)}}$. We denote the function class
\begin{align}
    \mathcal{F} = \left\{f_{x}: \mathbb{R}^{d} \to \mathbb{R}: f_{x}(t) = \frac{1}{\pi^{d}}\prod_{i = 1}^{d} \frac{\sin(\radius(x_{i} - t_{i}))}{\radius(x_{i} - t_{i})} \ \text{for all} \ x \in \mathcal{X}, t \in \mathbb{R}^{d} \right\}. \label{eq:empirical_process}
\end{align}
Then, we define a Gaussian process $\mathbb{B}$ on $\mathcal{F}$ with the covariance matrix given by:
\begin{align}
    \text{cov}(\mathbb{B}(f_{1}, f_{2})) = \Exs \brackets{f_{1}(X) f_{2}(X) } -  \Exs \brackets{f_{1}(X)} \Exs \brackets{ f_{2}(X) }, \label{eq:Gaussian_process}
\end{align}
for any $f_{1}, f_{2} \in \mathcal{F}$. We denote the maximum of the Gaussian process $\mathbb{B}$ as follows: $\bold{B} : = \sqrt{\radius^{d}} \sup_{f \in \mathcal{F}} \abss{\mathbb{B}(f)}$. 
We have the following result regarding the approximation of $$\sup_{x \in \mathcal{X}} \abss{\funcest_{n, \radius}(x) - \Exs \brackets{ \funcest_{n, \radius}(x)}}$$ based on $\bold{B}$.
\begin{proposition}
\label{prop:approx_Gaussian_process}
Assume that $\mathcal{X}$ is a bounded subset of $\mathbb{R}^{d}$ and $p_{0}$ is upper-supersmooth density function of order $\alpha > 0$. Then, as $\radius^{\alpha} = C \log n$ where $C$ is some universal constant depending on $d$ and $n \to \infty$ we have
\begin{align*}
    \sup_{t \geq 0} \abss{\Prob \parenth{\sqrt{\frac{n}{\radius^{d}}}\sup_{x \in \mathcal{X}} \abss{\funcest_{n, \radius}(x) - \Exs \brackets{ \funcest_{n, \radius}(x)}} < t} - \Prob \parenth{\bold{B} < t}} \leq  C' \frac{(\log n)^{(7+d)/8}}{n^{1/8}},
\end{align*}
where $C'$ is some universal constant.
\end{proposition}
\begin{proof}
The proof of Proposition~\ref{prop:approx_Gaussian_process} is based on the tools developed from the seminal works~\citep{Victor_2014, Victor_2014b}. For the simplicity of the presentation, given the functional space $\mathcal{F}$ defined in \eqref{eq:empirical_process}, we define the following empirical process:
\begin{align}
    \mathbb{G}_{n}(f) = \frac{1}{\sqrt{n}} \parenth{ \sum_{i = 1}^{n} f(X_{i}) - \Exs \brackets{f(X_{1})}},
\end{align}
for any $f \in \mathcal{F}$. We first show that $\mathcal{F}$ is a VC-type class of functions. Indeed, for any $x_{1}, x_{2} \in \mathcal{X}$ we have
%\begin{align*}
   $ \abss{f_{x_{1}}(t) - f_{x_{2}}(t)} \leq d \radius \enorm{x_{1} - x_{2}}$,
%\end{align*}
for all $t \in \mathbb{R}^{d}$. Since $\mathcal{X}$ is a bounded subset of $\mathbb{R}^{d}$, we have
\begin{align*}
    \sup_{P} \mathcal{N}_{2} \parenth{ t/8, \mathcal{F}, P} \leq \sup_{P} \mathcal{N}_{[]} \parenth{ t/8, \mathcal{F}, \mathbb{L}_{2}(P)} \leq \parenth{ \frac{ 4 d \sqrt{d} \cdot \text{Diam}(\mathcal{X}) \radius^2}{t}}^{d},
\end{align*}
where $\mathcal{N}_{2} \parenth{ t/8, \mathcal{F}, P}$ is the $t/8$-covering of $\mathcal{F}$ under $\mathbb{L}_{2}$ norm. Since the envelope function of $\mathcal{F}$ is $1/ \pi^{d}$, it shows that $\mathcal{F}$ is a VC-type class of functions. 

In order to facilitate the ensuing discussion, we denote 
$A = \half d \sqrt{d} \text{Diam}(\mathcal{X}) \radius^2$. Direct calculation shows that $\sup_{f \in \mathcal{F}} \Exs \brackets{f^2(X)} \leq 1/ \radius^{d} = \sigma^2$. Furthermore, we can choose the envelope function of $\mathcal{F}$ to be 1. Then, for any $\gamma \in (0, 1)$, an application of Corollary 2.2 in~\citep{Victor_2014b} shows that 
\begin{align}
    \Prob \parenth{\abss{\sup_{f \in \mathcal{F}} \mathbb{G}_{n}(f) - \sup_{f \in \mathcal{F}} \abss{\mathbb{B}(f)}} > \frac{K_{n}}{\gamma^{1/2}n^{1/4}} + \frac{\sqrt{\sigma}K_{n}^{3/4}}{\gamma^{1/2}n^{1/4}} + \frac{\sigma^{2/3} K_{n}^{2/3}}{\gamma^{1/3}n^{1/6}}} \leq C \parenth{\gamma + \frac{\log n}{n}},
\end{align}
where $C$ is some universal constant. Here, $K_{n} = c d (\log n \vee \log(A/\sigma))$ where $c$ is some universal constant. Since $\radius = \mathcal{O}(\log n)$, as $n$ is sufficiently large, we find that
\begin{align*}
    \Prob \parenth{\abss{\sup_{f \in \mathcal{F}} \mathbb{G}_{n}(f) - \sup_{f \in \mathcal{F}} \abss{\mathbb{B}(f)}} > C_{1} \frac{(\log n)^{2/3}}{\gamma^{1/3}\radius^{d/3}n^{1/6}}} \leq C_{2} \gamma,
\end{align*}
where $C_{1}$ and $C_{2}$ are some universal constants depending on $d$. The above result is also equivalent to
\begin{align}
    \Prob \parenth{\abss{\sqrt{\frac{n}{\radius^{d}}}\sup_{x \in \mathcal{X}} \abss{\funcest_{n, \radius}(x) - \Exs \brackets{ \funcest_{n, \radius}(x)}} - \bold{B}} > C_{1} \frac{\radius^{d/6}(\log n)^{2/3}}{\gamma^{1/3}n^{1/6}}} \leq C_{2}\gamma. \label{eq:Gaussian_approximation}
\end{align}
Combining the above result~\eqref{eq:Gaussian_approximation} with the result of Lemma 2.3 in~\citep{Victor_2014b}, for any $\gamma \in (0,1)$, when $n$ is sufficiently large we obtain that
\begin{align*}
    \sup_{t \geq 0} \abss{\Prob \parenth{\sqrt{\frac{n}{\radius^{d}}}\sup_{x \in \mathcal{X}} \abss{\funcest_{n, \radius}(x) - \Exs \brackets{ \funcest_{n, \radius}(x)}} < t} - \Prob \parenth{\bold{B} < t}} \leq C_{3} \Exs \brackets{\bold{B}} \frac{\radius^{d/6}(\log n)^{2/3}}{\gamma^{1/3}n^{1/6}} + C_{4} \gamma,
\end{align*}
where $C_{3}$ and $C_{4}$ are some universal constants. From Dudley's inequality for Gaussian process, we have $\Exs \brackets{\bold{B}} \leq C_{5} \sqrt{\log n}$ where $C_{5}$ is some universal constant. Putting the above results together, by choosing $\gamma = \radius^{d/8}(\log n)^{7/8}/n^{1/8}$, we obtain the conclusion of the proposition. 
\end{proof}

The distribution of $\bold{B}$ depends on the knowledge of the unknown density function $p_{0}$. Therefore, it is non-trivial to construct confidence band for $\Exs \brackets{\funcest_{n, \radius}}$ based on the result of Proposition~\ref{prop:approx_Gaussian_process}. To account for this issue, we utilize bootstrap idea. In particular, we denote $X_{1}^{*}, \ldots, X_{n}^{*}$ the i.i.d. sample from the empirical distribution $P_{n} = \frac{1}{n} \sum_{i = 1}^{n} \delta_{X_{i}}$. Then, we construct a Fourier density estimator $\funcest_{n,\radius}^{*}$ based on $X_{1}^{*}, \ldots, X_{n}^{*}$. Our next result provides the asymptotic behavior of $\sup_{x \in \mathcal{X}} \abss{\funcest_{n, \radius}^{*}(x) -  \funcest_{n, \radius}(x)}$ given the data $X_{1}, \ldots, X_{n}$.
\begin{proposition}
\label{prop:boostrap_Gaussian_process}
Assume that $\mathcal{X}$ is a bounded subset of $\mathbb{R}^{d}$ and $p_{0}$ is upper-supersmooth density function of order $\alpha > 0$. Then, as $\radius^{\alpha} = C \log n$ where $C$ is some universal constant depending on $d$ and $n \to \infty$ we have
\begin{align*}
\sup_{t \geq 0} \abss{\Prob \parenth{\sqrt{\frac{n}{\radius^{d}}}\sup_{x \in \mathcal{X}} \abss{\funcest_{n, \radius}^{*}(x) - \funcest_{n, \radius}(x)} < t \; \big| \; X_{1}, \ldots, X_{n}} - \Prob \parenth{\bold{B} < t}} = \mathcal{O}_{P} \parenth{\frac{(\log n)^{(7+d)/8}}{n^{1/8}}}.
\end{align*}
\end{proposition}
\noindent
The proof of Proposition~\ref{prop:boostrap_Gaussian_process} is in Appendix~\ref{subsec:proof:prop:boostrap_Gaussian_process}.

The results of Propositions~\ref{prop:approx_Gaussian_process} and~\ref{prop:boostrap_Gaussian_process} suggest the bootstrap procedure in Algorithm~\ref{Algorithm:Boostrap_Fourier_estimator} for constructing the confidence interval $\text{UCI}_{1 - \alpha}(x)$ in equation~\eqref{eq:uniform_CI_Fourier_estimator} for $\Exs \brackets{\funcest_{n, \radius}(x)}$ uniformly for all $x \in \mathcal{X}$. The following result showing that $\text{UCI}_{1 - \alpha}(x)$ is a valid $1 - \alpha$ confidence band for $p_{0}$:
\begin{corollary}
\label{cor:uniform_CI_Fourier_estimator}
Assume that $p_{0}$ is an upper--smooth density function of order $\alpha > 0$ and $\mathcal{X}$ is a bounded subset of $\mathbb{R}^{d}$. When $\radius^{\alpha} = C \log n$ where $C$ is some universal constant, for any $\tau \in (0, 1)$ we obtain that
\begin{align*}
    \lim_{n \to \infty} \Prob \left( p_{0}(x) \in \text{UCI}_{1 - \tau}(x) \ \text{for all} \ x \in \mathcal{X}\right) \geq 1 - \tau.
\end{align*}
\end{corollary}
\noindent
The proof of Corollary~\ref{cor:uniform_CI_Fourier_estimator} is a direct consequence of Propositions~\ref{prop:approx_Gaussian_process} and~\ref{prop:boostrap_Gaussian_process} and the fact that $\sup_{x \in \mathcal{X}} A_{2} \to 0$ in equation~\eqref{eq:pointwise_CI_Fourier} as $n \to \infty$ when $\radius^{\alpha} = \mathcal{O}(\log n)$; therefore, it is omitted. 
\begin{algorithm}[!t]
\caption{\textsc{Bootstrap\_Fourier\_estimator}} \label{Algorithm:Boostrap_Fourier_estimator}
\begin{algorithmic}
\STATE \textbf{Input:} Data $X_{1}, \ldots, X_{n}$.
\STATE \textbf{Step 1.} Drawing $B$ bootstrap samples $(X_{1}^{* (1)}, \ldots, X_{n}^{* (1)}), \ldots, (X_{1}^{* (B)}, \ldots, X_{n}^{* (B)})$ from the empirical measure $P_{n} = \frac{1}{n} \sum_{i = 1}^{n} \delta_{X_{i}}$.
\STATE \textbf{Step 2.} Contructing Fourier density estimators $\funcest_{n, \radius}^{* (1)}, \ldots, \funcest_{n, \radius}^{* (B)}$ from the $B$ bootstrap samples.
\STATE \textbf{Step 3.} Computing $T_{i} = \sqrt{\frac{n}{\radius^{d}}} \sup_{x \in \mathcal{X}} \abss{ \funcest_{n, \radius}^{* (i)}(x) - \funcest_{n, \radius}(x)}$ for $i \in [B]$.
\STATE \textbf{Step 4.} Choosing $\eta_{1 - \tau}(x)$ such that $\frac{1}{B} \sum_{i = 1}^{B} \bold{1}_{\{T_{i} > \eta_{\tau}(x)\}} = \tau$ for each $x \in \mathcal{X}$ and $\tau \in (0, 1)$.
\STATE \textbf{Step 5.} Constructing the uniform confidence interval for $p_{0}(x)$ as follows:
\begin{align}
    \text{UCI}_{1 - \tau}(x) = \funcest_{n, \radius}(x) \pm \eta_{1 - \tau}(x) \sqrt{\frac{\radius^{d} }{n}}. \label{eq:uniform_CI_Fourier_estimator}
\end{align}
\STATE \textbf{Output:} $\text{UCI}_{1 - \tau}(x)$.  
\end{algorithmic}
\end{algorithm}
%%%%%%%%%%%%%%%%%%%%%%%%%%%%%%%%%%%%%%%%%%%%%%%%%%%%%%%%%%%%%%%%%%%%%%%%%%%%%%%%%%%%%%%%%%%%%%%%%%%%%%%%%%%

\section{Estimating a mixing density with deconvolution}
\label{sec:mixing_measure_estimation}
In this section we employ the idea of Fourier density estimator to the deconvolution problem. For previous works on estimating a mixing density via maximum likelihood, see the works~\citep{Laird78} and~\cite{Lind83}, and for deconvolution approaches~\citep{Hall88, Zhang-90, Stefanski_1990}. These latter papers only consider the one--dimensional case 
%which have arised in several applications~\genecomment{Cite several references here.}, 
and we demonstrate improved rates of estimating mixing densities. Specifically, throughout this section, we assume that $p_{0}(x) = \int_{\Theta} f(x - \theta) \deconv(\theta) d \theta$, i.e., $X_{1}, \ldots, X_{n}$ are i.i.d. samples from $p_{0}$ which is the convolution between $f$ and $\deconv$. Here, $\Theta$ is a given subset of $\mathbb{R}^{d}$. In the deconvolution setting, the function $f$ is corresponding to the density function of ``noise'' on $\mathbb{R}^{d}$, which is assumed fully specified. Popular examples of $f$ include multivariate Gaussian or Laplace distributions with a given covariance matrix. The mixing density $\deconv$ is unknown and to be estimated. Finally, we assume throughout this section that $\mathcal{X} = \mathbb{R}^{d}$ and $f$ is a symmetric density function around 0, namely, $f(x) = f(-x)$ for all $x \in \mathbb{R}^{d}$. This assumption is to guarantee that the Fourier transform $\widehat{f}(s)$ of the function $f$ only takes real values. 

Using the insight from the Fourier integral theorem, we define the following \emph{Fourier deconvolution estimator} of $g$ as follows:
\begin{align}
    \deconvest_{n,\radius}(\theta) = \frac{1}{n (2 \pi)^{d}} \sum_{i = 1}^{n} \int_{[-\radius, \radius]^{d}} \frac{\cosfunc(s^{\top}( \theta - X_{i}))}{\widehat{f}(s)} ds. \label{eq:decon_mixture}
\end{align}
Since $\widehat{f}(s) \in \mathbb{R}$ for all $s \in \mathbb{R}^{d}$, the Fourier density estimator $\deconvest_{n,\radius}(\theta) \in \mathbb{R}$ for all $\theta \in \Theta$. As long as $\widehat{p}_{0}(s)/ \widehat{f}(s)$ is integrable, from the inverse Fourier transform we find that
\begin{align}
    \deconv(\theta) = \frac{1}{(2\pi)^{d}} \int_{\mathbb{R}^{d}} \int_{\mathbb{R}^{d}} p_{0}(x) \frac{\cosfunc(s^{\top}( \theta - x))}{\widehat{f}(s)} dx ds. \label{eq:Fourier_inversion_deconvolution}
\end{align}
for almost surely $\theta \in \Theta$. Note that, when we further assume that $g$ is continuous, the inverse Fourier transform in equation~\eqref{eq:Fourier_inversion_deconvolution} holds for all $\theta \in \Theta$. In summary, under these assumptions, we have $\lim_{\radius \to \infty} \Exs \brackets{ \deconvest_{n,\radius}(\theta)} = \deconv(\theta)$ where the outer expectation is taken with respect to $X$ that has density function $p_{0}$.

\subsection{Risk analysis with Fourier deconvolution estimator}
\label{subsec:risk_Fourier_deconvolution}
Similar to Section~\ref{sec:Fourier_density}, we would like to study upper bounds on the bias and variance of $\deconvest_{n,\radius}(\theta)$ under various smoothness settings of the density functions $f$ and $\deconv$. We first consider the setting when $f$ is a lower--supersmooth density function. Under this setting, to guarantee that $\widehat{p}_{0}(s)/ \widehat{f}(s)$ is integrable, $\widehat{f}$ needs to be lower--supersmooth density function with a certain condition on its growth.
\begin{theorem} \label{theorem:deconvolution_bias_variance_supersmooth}
Assume that $f$ is a lower--supersmooth density function of order $\alpha_{1} > 0$ and $\deconv$ is an upper--supersmooth density function of order $\alpha_{2} > 0$ such that $\alpha_{2} \geq \alpha_{1}$ and $\|\deconv\|_{\infty} < \infty$. Then, there exist universal constants $C$ and $C$' such that while $R \geq C'$, we have
\begin{align*}
    \abss{\Exs \brackets{\deconvest_{n,\radius}(\theta)} - \deconv(\theta)}
    & \leq C R^{\max \{1 - \alpha_{2}, 0\}} \exp \parenth{-C_{1} \radius^{\alpha_{2}} }, \\
    \var \brackets{\deconvest_{n,\radius}(\theta)} & \leq C \cdot \frac{\radius^{2d} \exp(2 C_{2} d \radius^{\alpha_{1}})}{n},
\end{align*}
for almost all $\theta \in \Theta$ where $C_{1}$ and $C_{2}$ are constants given in Definition~\ref{def:tail_Fourier}.
\end{theorem}

\noindent
Based on the result of Theorem~\ref{theorem:deconvolution_bias_variance_supersmooth}, when $f$ and $g$ are respectively lower--supersmooth and upper--supersmooth density functions of order $\alpha_{1}$ and $\alpha_{2}$, the MISE of the Fourier deconvolution estimator $\deconvest_{n,\radius}$ satisfies the following bound:
\begin{align}
    \text{MISE}(\deconvest_{n,\radius}) \leq C^2 R^{\max \{2 - 2\alpha_{2}, 0\}} \exp \parenth{- 2 C_{1} \radius^{\alpha_{2}} } + \|g\|_{\infty} \cdot \frac{\radius^{2d} \exp(2 C_{2} d R^{\alpha_{1}})}{n}, \label{eq:MISE_supersmooth_supersmooth}
\end{align}
where $C, C_{1}, C_{2}$ are given in part (a) of Theorem~\ref{theorem:deconvolution_bias_variance_supersmooth}. When $\alpha_{2} \geq \alpha_{1}$, the bound of MISE in equation~\eqref{eq:MISE_supersmooth_supersmooth} suggests that if we choose $\radius$ such that $(2C_{1} + 2 C_{2} d)\radius^{\alpha_{2}} = \log n$, the MISE rate of $\deconvest_{n,\radius}$ becomes $\bar{C} n^{-\frac{C_{1}}{C_{1}+C_{2} d}}$ (up to some logarithmic factor) where $\bar{C}$ is some universal constant depending on $d$. It suggests that when $\alpha_{2} \geq \alpha_{1}$, the MSE rate is polynomial in $n$, which is much faster than the known non-polynomial rate $1/ (\log n)^{\gamma}$ of estimating mixing density when the noise function $f$ is supersmooth~\citep{Zhang-90, Fan-91} where $\gamma > 0$ is some constant. A simple and popular deconvolution setting when $\alpha_{2} \geq \alpha_{1}$ is when $f$ is multivariate Gaussian distribution and $g$ is continuous Gaussian mixtures, i.e., $g(\theta) = \int f(\theta|\mu, \Sigma) d H(\mu, \Sigma)$ where $f(.|\mu, \Sigma)$ is multivariate Gaussian distribution with location and covariance $\mu$ and $\Sigma$ and $H$ is a prior distribution on $(\theta, \Sigma)$.
\begin{proof}
We first compute $\Exs \brackets{\deconvest_{n,\radius}(\theta)} - \deconv(\theta)$ for each $\theta \in \Theta$. Direct calculation shows that
\begin{align*}
    \Exs \brackets{\deconvest_{n, \radius}(\theta)} - g(\theta) & = \frac{1}{(2\pi)^{d}} \int_{\mathbb{R}^{d} \backslash [-\radius, \radius]^{d}} \int_{\mathbb{R}^{d}} \frac{\cos(s^{\top}(\theta - x))}{\widehat{f}(s)} p_{0}(x) dx ds \\
    & = \frac{1}{(2\pi)^{d}} \int_{\mathbb{R}^{d} \backslash [-\radius, \radius]^{d}} \int_{\mathbb{R}^{d}} \frac{\cos(s^{\top}(\theta - x))}{\widehat{f}(s)} \parenth{ \int f(x - \theta') \deconv(\theta') d\theta'} dx ds \\
    & = \frac{1}{(2\pi)^{d}} \int_{\mathbb{R}^{d} \backslash [-\radius, \radius]^{d}} \int_{\Theta} \frac{g(\theta')}{\widehat{f}(s)} \parenth{ \int_{\mathbb{R}^{d}}\cos(s^{\top}(\theta - x)) f(x - \theta') d x} d\theta' ds.
\end{align*}
By defining $\bar{\theta} = \theta - \theta'$, we obtain
\begin{align*}
    \int_{\mathbb{R}^{d}} \cos(s^{\top}(\theta - x)) f(x - \theta') d x & = \int_{\mathbb{R}^{d}} \cos(s^{\top} (x - \bar{\theta})) f(x) d x = \cos(s^{\top} \bar{\theta}) \widehat{f}(s).
\end{align*}
Putting the above results together, we obtain
\begin{align*}
    \Exs \brackets{\deconvest_{n, \radius}(\theta)} - g(\theta) & = \frac{1}{(2\pi)^{d}} \int_{\mathbb{R}^{d} \backslash [-\radius, \radius]^{d}} \int_{\Theta} \cos(s^{\top}(\theta - \theta')) g(\theta') d\theta' ds.
\end{align*}
The above term is similar to that in the proof of Theorem~\ref{theorem:bias_variance_Fourier_density}; therefore, the upper bound for its absolute value under the upper-supersmooth assumption of $\deconv$ is direct from the proof of Theorem~\ref{theorem:bias_variance_Fourier_density}. 

Moving to the variance of $\deconvest_{n,\radius}(\theta)$, simple algebra shows that
\begin{align*}
    \var \brackets{\deconvest_{n,\radius}(\theta)} \leq \frac{1}{n (2 \pi)^d} \Exs \brackets{\parenth{\int_{[-\radius, \radius]^{d}} \frac{\cosfunc(s^{\top}( \theta - X))}{\widehat{f}(s)} ds}^2} \leq \frac{\|g\|_{\infty}R^{2d}}{n \min_{s \in [-\radius,\radius]^{d}} \widehat{f}^2(s)}. 
\end{align*}
Based on the assumptions with the lower-supersmoothness of $f$, the above bound directly leads to the conclusion of the theorem with the variance of $\deconvest_{n,\radius}$.
\end{proof}
Our next result is when $f$ is a lower--ordinary smooth density function, such as multivariate Laplace distribution.
  
\begin{theorem} \label{theorem:deconvolution_bias_variance_ordinary}
Assume that $f$ is a lower--ordinary smooth density function of order $\beta_{1} > 0$. Then, the following holds:

\noindent
(a) When $\deconv$ is an upper--supersmooth density function of order $\alpha > 0$ and $\|\deconv\|_{\infty} < \infty$, there exist universal constants $C$ and $C'$ such that as long as $R \geq C'$, we have
\begin{align*}
    \abss{\Exs \brackets{\deconvest_{n,\radius}(\theta)} - \deconv(\theta)}
    & \leq C \radius^{\max \{1 - \alpha, 0\}} \exp \parenth{-C_{1} \radius^{\alpha}}, \\
    \var \brackets{\deconvest_{n,\radius}(\theta)} & \leq C \cdot \frac{\radius^{(2+2\beta_{1})d}}{n},
\end{align*}
for almost surely $\theta \in \Theta$ where $C_{1}$ is a constant given in Definition~\ref{def:tail_Fourier}. 

\noindent
(b) When $\deconv$ is upper-ordinary smooth density function of order $\beta_{2} > 1$ and $\|\deconv\|_{\infty} < \infty$, there exists universal constants $c$ such that for almost surely $\theta \in \Theta$ we obtain
\begin{align*}
    \abss{\Exs \brackets{\deconvest_{n,\radius}(\theta)} - \deconv ( \theta)}
    \leq \frac{c}{\radius^{\beta_{2} - 1}}, \quad \quad 
    \var \brackets{\deconvest_{n,\radius}(\theta)} \leq C \cdot \frac{\radius^{(2+2\beta_{1})d}}{n}.
\end{align*}
\end{theorem}
\noindent
The proof of Theorem~\ref{theorem:deconvolution_bias_variance_ordinary} follows the same argument as that of Theorem~\ref{theorem:deconvolution_bias_variance_supersmooth}; therefore, it is omitted. Based on the results of Theorem~\ref{theorem:deconvolution_bias_variance_ordinary}, we have the following bounds with the MISE of the Fourier deconvolution estimator:

\noindent
(i) When $f$ is lower-ordinary smooth function of order $\beta_{1}$ and $\deconv$ is upper-smooth function of order $\alpha > 0$, we obtain
    \begin{align*}
        \text{MISE}(\deconvest_{n,\radius}) \leq C^2 R^{\max \{2 - 2\alpha, 0\}} \exp \parenth{- 2 C_{1} \radius^{\alpha} } + C \cdot \frac{\radius^{(2+\beta_{1})d}}{n}, %\label{eq:MISE_ordinary_supersmooth}
    \end{align*}
    where $c, C, C_{1}$ are given in part (a) of Theorem~\ref{theorem:deconvolution_bias_variance_ordinary}. By choosing the bandwidth $\radius$ such that $2C_{1}\radius^{\alpha} = \log n$, the MISE rate of $\deconvest_{n,\radius}$ becomes $\bar{C} n^{-1}\parenth{ \log n}^{\max\{(2 + \alpha)d/\alpha, (2 - 2 \alpha)/ \alpha \}}$ where $\bar{C}$ is some universal constant. It is also faster than the best known polynomial rate of estimating mixing density function $g$ when $f$ is ordinary smooth function~\citep{Fan-91}. A popular example for this setting is when $f$ is a multivariate Laplace distribution, which is a lower--ordinary smooth density function of second order, and $g$ is a multivariate Gaussian distribution, which is an upper--supersmooth density function of second order.
    
\noindent
(ii) When $f$ is lower-ordinary smooth function of order $\beta_{1}$ and $\deconv$ is upper-ordinary smooth function of order $\beta_{2} > 0$, the upper bound for MISE of $\deconvest_{n,\radius}$ becomes
    \begin{align*}
        \text{MISE}(\deconvest_{n,\radius}) \leq \frac{c^2}{\radius^{2(\beta_{2} - 1)}} + C \cdot \frac{\radius^{(2+2\beta_{1})d}}{n}, %\label{eq:MISE_ordinary_supersmooth}
    \end{align*}
    where $c$ and $C$ are constants in part (b) of Theorem~\ref{theorem:deconvolution_bias_variance_ordinary}. With the choice of $\radius$ such that $\radius^{2(\beta_{2} - 1) + (2 + 2\beta_{1})d} = n$, we obtain $\text{MISE}(\deconvest_{n,\radius}) \leq \bar{c} n^{-\frac{2(\beta_{2} - 1)}{2(\beta_{2} - 1) + (2 + 2 \beta_{1}) d}}$ where $\bar{c}$ is some universal constant. Examples of this setting include when both $f$ and $g$ are multivariate Laplace distributions. 
%%%%%%%%%%%%%%%%%%%%%%%%%%%%%%%%%%%%%%%%%%%%%%%%%%%%%%%%%%%%%%%%%%%%%%%%%%%%%%%%%%%%%%%%%%%%%%%%%%%%%

\subsection{Derivatives of Fourier deconvolution estimator}
\label{sec:derivative_Fourier_deconvolution_estimator}
Similar to the Fourier density estimator, we also would like to investigate the MISE of the derivatives of the Fourier deconvolution estimator, which is useful for our study with mode estimation of mixing density function (see Section~\ref{sec:mode_clustering_mixing_density} for an example). We first start with the upper bounds for the mean-squared variance and bias of $\nabla^{r} \deconvest_{n,\radius}$ when $f$ is lower-supersmooth density function.

\begin{theorem}
\label{theorem:deconvolution_bias_variance_supersmooth_derivatives}
Assume that $f$ and $g$ satisfy the assumptions of Theorem~\ref{theorem:deconvolution_bias_variance_supersmooth}. Furthermore, $g \in \mathcal{C}^{r}(\Theta)$ for given $r \in \mathbb{N}$. Then, there exist universal constants $\{C_{i}'\}_{i = 1}^{r}$ and $\{\bar{C}_{i}\}_{i = 1}^{r}$ such that as long as $\radius \geq C'$ where $C' > 0$ is some universal constant and $i \in [r]$, we have
\begin{align*}
    \sup_{\theta \in \Theta} \| \Exs \brackets{\nabla^{i} \deconvest_{n, \radius}(\theta)} - \nabla^{i} g(\theta)\|_{\max} \leq C_{i}' \radius^{\max \{i + 1 - \alpha_{2}, 0\}} \exp \parenth{- C_{1} \radius^{\alpha_{2}}}, \\
    \sup_{\theta \in \Theta} \Exs \brackets{ \enorm{ \nabla^{i} \deconvest_{n,\radius}(\theta) - \Exs \brackets{ \nabla^{i} \deconvest_{n,\radius}(\theta)}}^2} \leq \bar{C}_{i} \frac{\radius^{2(i + d)} \exp(2C_{2}d\radius^{\alpha_{1}})}{n},
\end{align*}
where $C_{1}$ and $C_{2}$ are constants associated with supersmooth density functions given in Definition~\ref{def:tail_Fourier}.
\end{theorem}
\noindent 
The proof of Theorem~\ref{theorem:deconvolution_bias_variance_supersmooth_derivatives} is in Section~\ref{subsec:proof:theorem:deconvolution_bias_variance_supersmooth_derivatives}.
The results of Theorem~\ref{theorem:deconvolution_bias_variance_supersmooth_derivatives} demonstrate that the MISE of $\nabla^{r} \deconvest_{n, \radius}$ for any $r \in \mathbb{N}$ can be upper bounded as follows:
\begin{align*}
    \text{MISE}(\nabla^{r} \deconvest_{n, \radius}(\theta)) \leq C_{r}' \radius^{\max \{2(r + 1 - \alpha_{2}), 0\}} \exp \parenth{- 2 C_{1} \radius^{\alpha_{2}}} + \bar{C}_{r} \radius^{2(r + d)} \exp(2C_{2}d\radius^{\alpha_{1}}).
\end{align*}
Therefore, by choosing the radius $\radius$ such that $(2C_{1} + 2C_{2}d)\radius^{\alpha_{2}} = \log n$, the MISE rate of $\nabla^{r} \deconvest_{n, \radius}$ becomes $\bar{C} n^{-\frac{C_{1}}{C_{1}+C_{2} d}} \parenth{\log n}^{\max \{2(r + 1 - \alpha_{2})/ \alpha_{2}, 2(d+ r)/ \alpha_{2} \}}$, which is still polynomial up to some logarithmic factor, where $\bar{C}$ is some universal constant. 

We now move to our next result with the upper bounds of variance and bias of $\nabla^{r} \deconvest_{n,\radius}$ when $f$ is lower-ordinary smooth density function.
\begin{theorem}
\label{theorem:deconvolution_bias_variance_ordinarysmooth_derivatives}
Assume that $f$ is a lower--ordinary smooth density function of order $\beta_{1} > 0$ and $g \in \mathcal{C}^{r}(\Theta)$ for given $r \in \mathbb{N}$. Then, for any $1 \leq i \leq r$, the following holds:

\noindent
(a) When $\deconv$ is an upper--supersmooth density function of order $\alpha > 0$, there exist universal constants $\{C'_{i}\}_{i = 1}^{r}$ and $\{\bar{C}_{i}\}_{i = 1}^{r}$ such that as long as $R \geq C'$ where $C'$ is some universal constant, we have
\begin{align*}
    \sup_{\theta \in \Theta} \| \Exs \brackets{\nabla^{i} \deconvest_{n, \radius}(\theta)} - \nabla^{i} g(\theta)\|_{\max} \leq C_{i}' \radius^{\max \{i + 1 - \alpha, 0\}} \exp \parenth{- C_{1} \radius^{\alpha}}, \\
    \sup_{\theta \in \Theta} \Exs \brackets{ \enorm{ \nabla^{i} \deconvest_{n,\radius}(\theta) - \Exs \brackets{ \nabla^{i} \deconvest_{n,\radius}(\theta)}}^2} \leq \bar{C}_{i} \frac{\radius^{(2+2\beta_{1})d + 2i}}{n},
\end{align*}
where $C_{1}$ is a given constant with upper-smooth density function from Definition~\ref{def:tail_Fourier}.

\noindent
(b) When $\deconv$ is an upper--ordinary smooth density function of order $\beta_{2} > 1 + r$, there exist universal constants $\{c'_{i}\}_{i = 1}^{r}$ such that
\begin{align*}
    \sup_{\theta \in \Theta} \| \Exs \brackets{\nabla^{i} \deconvest_{n, \radius}(\theta)} - \nabla^{i} g(\theta)\|_{\max} \leq \frac{c_{i}'} {\radius^{\beta_{2} - (i + 1)}}, \\
    \sup_{\theta \in \Theta} \Exs \brackets{ \enorm{ \nabla^{i} \deconvest_{n,\radius}(\theta) - \Exs \brackets{ \nabla^{i} \deconvest_{n,\radius}(\theta)}}^2} \leq \bar{C}_{i} \frac{\radius^{(2+2\beta_{1})d + 2i}}{n}.
\end{align*}
\end{theorem}
\noindent
The proof for Theorem~\ref{theorem:deconvolution_bias_variance_ordinarysmooth_derivatives} is similar to that of Theorem~\ref{theorem:deconvolution_bias_variance_supersmooth_derivatives} when the density function $f$ is upper-supersmooth; therefore, it is omitted.

The result of part (a) of Theorem~\ref{theorem:deconvolution_bias_variance_ordinarysmooth_derivatives} suggests that the optimal choice of the radius $\radius$ satisfies $2C_{1}\radius^{\alpha} = \log n$ when $f$ is lower-ordinary smooth density function of order $\beta_{1} > 0$ and $g$ is upper-supersmoth density function of order $\alpha > 0$. Under this choice of $\radius$, the MISE of $\nabla^{r} \deconvest_{n,\radius}$ has convergence rate of the order $\bar{C} n^{-1}\parenth{\log n}^{\max \{2(r + 1 - \alpha)/ \alpha, ((2 + 2\beta_{1})d+2r)/ \alpha \}}$, which is parametric up to some logarithmic factor, where $\bar{C}$ is some universal constant. 
On the other hand, when $f$ is lower-ordinary smooth density function of order $\beta_{1} > 0$ and $g$ is upper-ordinary smooth density function of order $\beta_{2} > 1 + r$, by choosing $\radius = n^{\frac{1}{2(\beta_{2} - 1 + (1 + \beta_{1})d)}}$, the MISE rate of $\nabla^{r} \deconvest_{n,\radius}$ becomes $\bar{c} n^{-\frac{\beta_{2} - (r + 1)}{\beta_{2} - 1 + (1 + \beta_{1})d}}$ where $\bar{c}$ is some universal constant.

\section{Nonparametric mode clustering}
\label{subsec:mode_clustering}
In this section, we consider an application of Fourier (mixing) density estimators to mode clustering problem~\citep{Azzalini_2007, Chacon_2013, Chacon_2015}. We first study mode clustering via the data density in Section~\ref{sec:mode_clustering_data_density}. Then, we consider another approach to study mode clustering via a mixing density function when the data density is assumed to be a mixture; Section~\ref{sec:mode_clustering_mixing_density}.

\subsection{Mode clustering via data density}
\label{sec:mode_clustering_data_density}
We assume that $X_{1}, \ldots, X_{n}$ are i.i.d. samples from the unknown distribution $P$ admitting the density function $p_{0}$ supported on $\mathcal{X} \subseteq \mathbb{R}^{d}$. When $p_{0}$ admits a second order derivative, we say that $x$ is the local mode of $p_{0}$ if 
\begin{align*}
    \nabla p_{0}(x) = 0 \ \text{and} \ \lambda_{1}(\nabla^2 p_{0}(x)) < 0 
\end{align*}
where recall that $\lambda_{d}(\nabla^2 p_{0}(x))$ denotes the largest eigenvalue of the Hessian matrix $\nabla^2 p_{0}(x)$. We define $\mathcal{M}$ the collection of local modes of the true density function $p_{0}$ and $K = |\mathcal{M}|$ the total number of local modes of $p_{0}$. For the mode clustering problem via data density, we would like to estimate the local modes of $p_{0}$ in $\mathcal{M}$ and the number of local modes $K$. To do that, we first obtain the Fourier density estimator $\funcest_{n, \radius}$ for $p_{0}$. Then, we calculate the local modes of $\funcest_{n, \radius}$, which serve as an estimation for the local modes of $p_{0}$. Note that, in the multivariate setting, the local modes of $\funcest_{n, \radius}$ can be determined by the well-known mean-shift algorithm~\citep{Fukunaga_1975, Comaniciu_2002, Castro_2016}. Finally, the total number of total modes of $\funcest_{n, \radius}$ can be used as an estimation for the $K$.

In order to faciliate the ensuing discussion, we denote $\mathcal{M}_{n}$ the collection of local modes of the Fourier density estimator $\funcest_{n, \radius}$ and $K_{n}$ the number of local modes of $\funcest_{n, \radius}$. We use the Hausdorff metric to measure the convergence of local modes in $\mathcal{M}_{n}$ to those of $\mathcal{M}$~\citep{Chen-2016}, which is given by:
\begin{align*}
    \mathcal{H}(\mathcal{M}_{n}, \mathcal{M}) : = \max \left\{ \sup_{x \in \mathcal{M}_{n}} d(x, \mathcal{M}), \sup_{x \in \mathcal{M}} d(x, \mathcal{M}_{n})\right\}.
\end{align*}
We impose the following assumptions on the density $p_{0}$ so as to establish the consistency of $K_{n}$ to $K$ as well as the convergence rate of $\mathcal{M}_{n}$ to $\mathcal{M}$ under the Hausdorff metric:
\begin{assumption}
\label{assume:Hessian_matrix}
There exists universal constant $\lambda^{*} < 0$ such that $\lambda_{d}(\nabla^2 p_{0}(x)) \leq \ldots \leq \lambda_{1}(\nabla^2 p_{0}(x)) \leq \lambda^{*}$ for any $x \in \mathcal{M}$.
\end{assumption}
\begin{assumption}
\label{assume:local_neighborhood}
The density function $p_{0} \in \mathcal{C}^{3}(\mathcal{X})$ and $\|\nabla^{3} p_{0}(x)\| \leq C$ for some universal constant $C$ for all $x \in \mathcal{X}$. Furthermore, there exists universal constant $\modclus$ such that $\{x: \ \|\nabla p_{0}(x)\| \leq \modclus, \ \lambda_{1}(\nabla^2 p_{0}(x)) \leq \frac{\lambda_{*}}{2}\} \subset \mathcal{M} \oplus \frac{|\lambda^{*}|}{2Cd}$ where $\lambda^{*}$ is constant in Assumption~\ref{assume:Hessian_matrix}. 
\end{assumption}
Note that, Assumptions~\ref{assume:Hessian_matrix} and~\ref{assume:local_neighborhood} had been employed in~\citep{Chen-2016} to analyze mode clustering via data density based on kernel density estimator. The idea of these assumptions are as follows. Assumption~\ref{assume:Hessian_matrix} is to guarantee that the Hessian matrix $\nabla^2 p_{0}(x)$ is not degenerate at each local mode $x \in \mathcal{M}$. Assumption~\ref{assume:local_neighborhood} is to make sure that for any points that have quite similar behaviors to local modes, they should also be close to these local models. 

Given Assumptions~\ref{assume:Hessian_matrix} and~\ref{assume:local_neighborhood} on hand, we proceed to only provide the result with mode clustering when the density function $p_{0}$ is upper-supersmooth as the result when the density function $p_{0}$ is upper-ordinary smooth can be argued in the similar fashion (see our discussion after Theorem~\ref{theorem:mode_clustering_data_density}).
\begin{proposition} \label{theorem:mode_clustering_data_density}
Assume that Assumptions~\ref{assume:Hessian_matrix} and~\ref{assume:local_neighborhood} hold. Furthermore, $p_{0}$ is upper-supersmooth density function of order $\alpha$ and $\mathcal{X}$ is a bounded subset of $\mathbb{R}^{d}$. Then, for any $\delta > 0$, when $\radius \geq C$ and $n \geq c\radius^{2(d+2)} \log(\radius) \log(6/ \delta)$ where $C$ and $c$ are some universal constants, the following holds:

\noindent
(a) (Consistency of estimating the number of modes) We have
\begin{align*}
    \Prob (\widehat{K}_{n} \neq K) \leq \delta.
\end{align*}
\noindent
(b) (Convergence rates of modes estimation) There exists universal constant $c_{1}$ such that
\begin{align*}
    \Prob \parenth{ \mathcal{H}(\mathcal{M}_{n}, \mathcal{M}) \leq c_{1} \parenth{\radius^{\max\{2 - \alpha, 0\}} \exp \parenth{-C_{1} \radius^{\alpha}} + \sqrt{\frac{\radius^{d + 2} \log(2/ \delta)}{n}}}} \geq 1 - \delta,
\end{align*}
where $C_{1}$ is a constant associated with upper-supersmooth density function in Definition~\ref{def:tail_Fourier}.
\end{proposition}
\noindent
The proof of Proposition~\ref{theorem:mode_clustering_data_density} is in Appendix~\ref{subsec:proof:theorem:mode_clustering_data_density}.

A few comments with Proposition~\ref{theorem:mode_clustering_data_density} are in order. First, given the result of part (b), we can choose $\radius$ such that $C_{1}\radius^{\alpha} = \log n/ 2$. Then, the convergence rate of $\mathcal{H}(\mathcal{M}_{n}, \mathcal{M})$ becomes $\bar{C} n^{-\half}\parenth{\log(n)}^{\max \{2/ \alpha - 1, (d + 2)/ (2 \alpha) \}}$, where $\bar{C}$ is some universal constant.  That parametric convergence rate of estimating modes is faster than the rate $n^{-2/(d + 6)}$ of estimating modes from kernel density estimator~\citep{Chen-2016}.

Second, when $p_{0}$ is an upper--ordinary smooth density function of order $\beta > 3$, with the similar proof argument as that of Theorem~\ref{theorem:mode_clustering_data_density}, we can demonstrate that when $\radius$ is sufficiently large and $n \geq \bar{c} \radius^{2(d + 2) \log R \log(6/ \delta)}$ where $\bar{c}$ is some universal constant, the following hold:
%\begin{align*}
  $$   \Prob (\widehat{K}_{n} \neq K) \leq \delta, \quad\mbox{and}\quad
     \Prob \parenth{ \mathcal{H}(\mathcal{M}_{n}, \mathcal{M}) \leq \frac{c_{1}'}{\radius^{\beta - 2}} + c_{2}' \sqrt{\frac{\radius^{d + 2} \log(2/ \delta)}{n}}} \geq 1 - \delta,$$
%\end{align*}
where $c_{1}'$ and $c_{2}'$ are some universal constants. Therefore, under the upper-ordinary smoothness setting of $p_{0}$, we can choose $\radius$ such that $\radius^{\beta - 2 + (d + 2)/2} = \sqrt{n}$. Then, the convergence rate of $\mathcal{H}(\mathcal{M}_{n}, \mathcal{M})$ is at the order of $n^{-\frac{\beta - 2}{2(\beta - 2) + d + 2}}$. If we further have $\beta > 4$, that convergence of modes estimation under the upper-ordinary smooth setting of $p_{0}$ is faster than the rate $n^{-2/(d + 6)}$ from kernel density estimator~\citep{Chen-2016}.

\subsection{Mode clustering via mixing density}
\label{sec:mode_clustering_mixing_density}
In this section, we assume that the density function $p_{0}$ of $X_{1},\ldots, X_{n}$ takes the mixture form $p_{0}(x) = \int_{\Theta} f(x - \theta) \deconv(\theta) d \theta$. Here, the density function $f$ is known and only the mixing density function $\deconv$ is unknown. When $\deconv$ is the mixture of Dirac delta functions, it is well-known that we can cluster the data based on estimating the support points of these Dirac delta distributions. For general $\deconv$, we would like to take this perspective of clustering and estimate the modes of $\deconv$ so as to cluster the data.

Since the mixing density $\deconv$ is unknown, we use the Fourier deconvolution estimator $\deconvest_{n,\radius}$ in equation~\eqref{eq:decon_mixture} to estimate $\deconv$ and then use the local modes of $\deconvest_{n,\radius}$ to estimate those of $\deconv$. To ease the presentation, we denote $\mathcal{M}'$ and $\mathcal{M}_{n}'$ respectively the set of all local modes of $\deconv$ and $\deconvest_{n,\radius}$. Furthermore, we denote $K' = |\mathcal{M}'|$ and $K_{n}' = |\mathcal{M}_{n}'|$ respectively as the number of local modes of $g$ and $\deconvest_{n, \radius}$. 

Since the proof techniques are similar for different smoothness settings of $f$ and $g$, we only focus on the setting when both $f$ and $g$ are supersmooth densities. The following result establishes the consistency of $K_{n}'$ and the convergence rate of $\mathcal{H}(\mathcal{M}_{n}', \mathcal{M}')$ when $n$ goes to infinity.
\begin{proposition}
\label{theorem:mode_clustering_mixing_density}
Assume that the mixing density function $\deconv$ satisfies Assumptions~\ref{assume:Hessian_matrix} and~\ref{assume:local_neighborhood}. Furthermore, $f$ is a symmetric lower-supersmooth density function of order $\alpha_{1} > 0$ while $\deconv$ is upper-smooth density function of order $\alpha_{2} > 0$ such that $\alpha_{2} \geq \alpha_{1}$. Then, for any $\delta > 0$, when $\radius \geq C$ and $n \geq c\radius^{2(d+2) + \alpha_{1}} \exp(2C_{2}d\radius^{\alpha_{1}})\log(6/ \delta)$ where $C$ and $c$ are some universal constants and $C_{2}$ is a given constant associated with the lower-supersmoothness of $f$ in Definition~\ref{def:tail_Fourier}, the following holds:

\noindent
(a) (Consistency of estimating the number of modes) We find that
\begin{align*}
    \Prob (K_{n}' \neq K') \leq \delta.
\end{align*}
\noindent
(b) (Convergence rates of modes estimation) There exists universal constants $c_{1}$ such that
\begin{align*}
    \Prob \biggr( \mathcal{H}(\mathcal{M}_{n}', \mathcal{M}') \leq c_{1} \radius^{\max\{2 - \alpha_{2}, 0\}} \exp \parenth{-C_{1} \radius^{\alpha_{2}}} & \\
    & \hspace{-5 em} + c_{1} \sqrt{\frac{\radius^{2(d + 1) + \alpha_{1}}\exp(2C_{2}d\radius^{\alpha_{1}}) \log (2/\delta)}{n}} \biggr) \geq 1 - \delta,
\end{align*}
where $C_{1}$ is a given constant associated with the upper-supersmoothness of $g$ in Definition~\ref{def:tail_Fourier}.
\end{proposition}  
\noindent
The proof of Proposition~\ref{theorem:mode_clustering_mixing_density} is in Appendix~\ref{subsec:proof:theorem:mode_clustering_mixing_density}.

Given the result of Proposition~\ref{theorem:mode_clustering_mixing_density}, we can choose $(2C_{1} + 2C_{2}d) \radius^{\alpha_{2}} = \log n$. Then, the convergence rate of $\mathcal{H}(\mathcal{M}_{n}', \mathcal{M}')$ is at the order of $n^{-C_{1}/(2C_{1} + 2C_{2}d)}$ (up to some logarithmic factor) where $C_{1}$ and $C_{2}$ are respectively the constants associated with the upper-supersmoothness and lower-supersmoothness of $g$ and $f$.
 %%%%%%%%%%%%%%%%%%%%%%%%%%%%%%%%%%%%%%%%%%%%%%%%%%%%%%%%%%%%%%%%%%%%%%%%%%%%%%%%%%%%%%%%%%%%%%%%%%%%%%%%%%
\section{Nonparametric regression}
\label{subsec:nonparametric_regression}
In this section we consider an application of the Fourier integral theorem to the setting of nonparametric regression. We assume that $Y_{i} = m(X_{i}) + \epsilon_{i}$ for all $i \in [n]$ where $\epsilon_{1},\ldots, \epsilon_{n}$ are i.i.d. additive noises satisfying $\Exs(\epsilon_{i}) = 0$ and $\var(\epsilon_{i}) = \sigma^2$. In our model, the function $m$ is unknown and to be estimated. 
%and we would like to estimate it.
%\subsection{Random design}
We consider the random design setting, namely, $X_{1}, \ldots, X_{n} \in \mathcal{X} \subseteq \mathbb{R}^{d}$ are i.i.d. samples from some density function $p_{0}$. Furthermore, to simplify the argument later, we assume the additive noises $\epsilon_{1}, \ldots, \epsilon_{n}$ are independent of the observations $X_{1}, \ldots, X_{n}$. 

Based on the Fourier density estimator studied in Section~\ref{sec:Fourier_density}, we propose the following Fourier nonparametric regression version of Nadaraya–Watson kernel estimator, named \emph{Fourier regression estimator}, for estimating the unknown function $m$:
\begin{equation}\label{nonregfou}
    \nonpreg(x) : = \dfrac{\sum_{i = 1}^{n} Y_{i} \cdot \prod_{j = 1}^{d} \frac{\sin(R (x_{j} - X_{ij}))}{x_{j} - X_{ij}}}{\sum_{i = 1}^{n} \prod_{j = 1}^{d} \frac{\sin(R (x_{j} - X_{ij}))}{x_{j} - X_{ij}}} = \frac{\widehat{a}(x)}{\funcest_{n, \radius}(x)},
\end{equation}
where $\widehat{a}(x) = \frac{1}{\pi^{d} n} \sum_{i = 1}^{n} Y_{i} \cdot \prod_{j = 1}^{d} \frac{\sin(R (x_{j} - X_{ij}))}{x_{j} - X_{ij}}$ and $\funcest_{n, \radius}$ is the Fourier density estimator given in equation~\eqref{eq:multivariate_density_estimator}. One notable advantage of the Fourier regression estimator $\nonpreg$ is that both its denominator and numerator can automatically capture the dependence between the covariates of $X_{1}, \ldots, X_{n}$, without the need to model a covariance matrix, as it is in the standard Nadaraya--Watson Gaussian kernel~\citep{Wass06, Tsy09}. Therefore, the Fourier regression estimator is convenient to use as we only need to choose the radius $\radius$.

Another benefit of using the estimator~\eqref{nonregfou} for estimating the function $m$ is that it can have parametric MSE rate when the density function $p_{0}$ of the observations $X_{1}, \ldots, X_{n}$ is upper-supersmooth. Indeed, under this setting of $p_{0}$, we have the following upper bound regarding the MSE of $\widehat{m}(x)$.
\begin{theorem} \label{theorem:random_nonparametric}
Assume that $p_{0}$ is an upper--supersmooth density function of order $\alpha > 0$ and $\|p_{0}\|_{\infty} < \infty$. Furthermore, assume that the function $m$ is such that $\|m^2 \times p_{0}\|_{\infty} < \infty$ and
\begin{align}
    \abss{ \widehat{m \cdot p_{0}}(t)} & \leq C \cdot Q(|t_{1}|,\ldots,|t_{d}|) \exp \parenth{ -C_{1} \parenth{ \sum_{i = 1}^{d} |t_{i}|^{\alpha}} }, \label{eq:cond_m_supersmooth}
\end{align}
where $C$ is some universal constant, $C_{1}$ is given constant in Definition~\ref{def:tail_Fourier}, and $Q(|t_{1}|,\ldots,|t_{d}|)$ is some polynomial in terms of $|t_{1}|, \ldots, |t_{d}|$ with non-negative coefficient. Then, there exist universal constants $C', (C_{i}')_{i = 1}^{3}$ such that as long as $R \geq C'$ we have
\begin{align*}
    \Exs \brackets{(\widehat{m}(x) - m(x))^2} \leq \frac{C_{1}' \radius^{\max \{2 \deg(Q) + 2 - 2 \alpha, 0\}} \exp( - 2 C_{1} \radius^{\alpha}) + C_{2}' \frac{(m(x) + C_{3}') \radius^{d}}{n}}{p_{0}^2(x) J(\radius)},
\end{align*}
where $J(R) = 1 - R^{\max \{2 - 2 \alpha, 0\}} \exp \parenth{-2 C_{1} \radius^{\alpha} } + \frac{\radius^{d} \log (n \radius)}{n}/p_{0}^2(x)$.
\end{theorem}
\noindent
The proof of Theorem~\ref{theorem:random_nonparametric} is in Section~\ref{subsec:proof:theorem:random_nonparametric}.

We have a few remarks with Theorem~\ref{theorem:random_nonparametric}. First, the assumptions with the unknown function $m$ in Theorem~\ref{theorem:random_nonparametric} is quite mild. It is satisfied when $p_{0}$ is a multivariate Gaussian distribution and $m$ is a polynomial function or polynomial trigonometric function. Second, by choosing the radius $\radius$ such that $2 C_{1} \radius^{\alpha} = \log n$, the rate of the MSE of $\nonpreg(x)$ becomes $$\Exs \brackets{(\widehat{m}(x) - m(x))^2} \leq \frac{\bar{C} (m(x) + \bar{C}_{1})}{p_{0}^2(x)} \cdot \frac{(\log n)^{\max\{\frac{2 \deg(Q) + 2 - 2\alpha}{\alpha}, \frac{d}{\alpha} \}}}{n}$$ where $\bar{C}$ and $\bar{C}_{1}$ are some universal constants. Therefore, we have parametric rate of MSE of $\nonpreg(x)$ for each $x \in \mathcal{X}$ when $p_{0}$ is an upper--supersmooth density function and $m$ satisfies the assumptions in Theorem~\ref{theorem:random_nonparametric}. This rate is also faster than the well-known MSE rate $n^{-1/(4 + d)}$ of Nadaraya-Watson regression kernel when both $p_{0}$ and $m$ have bounded second order derivatives~\citep{Wass06, Tsy09}. 

Based on the result of Theorem~\ref{theorem:random_nonparametric}, our next result provides the point-wise confidence interval for $m(x)$ based on the Fourier regression estimator $\nonpreg(x)$.
\begin{proposition}
\label{prop:CI_nonparametric_regression}
Assume that the assumptions of Theorem~\ref{theorem:random_nonparametric} hold and $\mathcal{X}$ is a bounded subset of $\mathbb{R}^{d}$. Then, for each $x \in \mathcal{X}$, as $\radius^{\alpha} = C \log n$ where $C$ is some universal constant and $n \to \infty$, we have
\begin{align*}
    \sqrt{\frac{n}{\radius^{d}}} \parenth{\nonpreg(x) - m(x)} \overset{d}{\to} \mathcal{N} \parenth{0, \frac{\sigma^2}{p_{0}(x) \pi^{d}}}.
\end{align*}
\end{proposition}
\noindent
The proof of Proposition~\ref{prop:CI_nonparametric_regression} is in Appendix~\ref{subsec:proof:prop:CI_nonparametric_regression}.

Based on the result of Proposition~\ref{prop:CI_nonparametric_regression}, for any $\tau \in (0, 1)$ we can construct the $1 - \tau$ point-wise confidence interval for $m(x)$ as follows:
\begin{align*}
    \widehat{m}(x) \pm z_{1 - \tau/2} \sqrt{\frac{\sigma^2 \radius^{d}}{n \pi^{d} p_{0}(x)}},
\end{align*}
where $z_{1 - \tau/2}$ stands for critical value of standard Gaussian distribution at the tail area $\tau/2$. Since the noise variance $\sigma^2$ and the value of $p_{0}(x)$ are unknown, we utilize the plug-in estimators for these terms. For $p_{0}(x)$, we can use $\abss{\funcest_{n, \radius}(x)}$ as plug-in estimator. Note that, we do not use $\max \{\funcest_{n, \radius}(x), 0\}$ as a plug-in estimator for $p_{0}(x)$ in this case since the inverse of this estimator will be infinity as long as $\funcest_{n, \radius}(x) < 0$. For $\sigma^2$, the common plug-in estimator is as follows~\citep{Peter_1990, Wass06}:
\begin{align*}
    \widehat{\sigma}^2 = \frac{\parenth{\sum_{i = 1}^{n} Y_{i} - \nonpreg(X_{i})}^2}{n - 2 \trace(L) + \trace(L^{\top}L)},
\end{align*}
where the matrix $L \in \mathbb{R}^{n \times n}$ satisfies $$L_{ij} = \frac{\prod_{u = 1}^{d} \frac{\sinfunc(\radius(X_{iu} - X_{ju}))}{X_{iu} - X_{ju}}}{\sum_{k = 1}^{n} \prod_{u = 1}^{d} \frac{\sinfunc(\radius(X_{iu} - X_{ku}))}{X_{iu} - X_{ku}}}.$$ Given these plug-in estimators, the $1 - \tau$ point-wise confidence interval for $m(x)$ becomes
\begin{align}
    \text{NPCI}_{1 - \tau}(x) = \widehat{m}(x) \pm z_{1 - \tau/2} \sqrt{\frac{\widehat{\sigma}^2 \radius^{d}}{n \pi^{d} \abss{\funcest_{n, \radius}(x)}}}, \label{eq:nonparametric_regression_CI}
\end{align}
where $\radius^{\alpha} = \mathcal{O}(\log n)$. In the random design setting, constructing the confidence band for the function $m$ based on the Fourier regression estimator is complicated due to the involvement of the Fourier density estimator $\funcest_{n, \radius}(x)$ in the denominator of $\nonpreg(x)$. We leave the development of confidence band of $m$ for the future work.
%%%%%%%%%%%%%%%%%%%%%%%%%%%%%%%%%%%%%%%%%%%%%%%%%%%%%%%%%%%%%%%%%%%%%%%%%%%%%%%%%%%%%%
%%%%%%%%%%%%%%%%%%%%%%%%%%

\section{Nonparametric modal regression}
\label{sec:modal_regression}
In this section, we consider an extension of local mode estimation to the regression setting. It is different from the traditional conditional mean nonparametric regression being considered in Section~\ref{subsec:nonparametric_regression}. In particular, assume that $Y \in \mathcal{Y} \subseteq \mathbb{R}$ is the response variable while $X \in \mathcal{X} \subseteq \mathbb{R}^{d}$ is the predictor variable. In nonparametric modal regression, we would like to study the conditional local mode at $X = x$, which is given by:
\begin{align*}
    \mathcal{M}(x) : = \left\{y: \frac{\partial{p_{0}}}{\partial{y}}(x, y) = 0, \ \frac{\partial^2{p_{0}}}{\partial{y}^2}(x, y) < 0 \right\},
\end{align*}
where $p_{0}(x, y)$ is the joint density between $X$ and $Y$. Since $p_{0}$ is unknown, we utilize the Fourier density estimator to estimate it, which admits the following form:
\begin{align}
    \funcestmode_{n, \radius}(x, y) = \frac{1}{n \pi^{d}} \sum_{ i = 1}^{n} \parenth{ \prod_{j = 1}^{d} \frac{\sinfunc(\radius(x_{j} - X_{ij}))}{x_{j} - X_{ij}}} \cdot \frac{\sinfunc(\radius(y - Y_{i}))}{y - Y_{i}}. \label{eq:model_regression_density_estimator}
\end{align}
Note that, even though $Y_{i}$ and $X_{i}$ are not independent, their dependence is captured via the Fourier integral theorem; therefore, the estimator~\eqref{eq:model_regression_density_estimator} is comfortable to use as we only need to choose the radius $\radius$. The corresponding conditional local mode at $X = x$ based on the estimator $\funcest_{n, \radius}$ is given by:
\begin{align}
    \mathcal{M}_{n}(x) : = \left\{y: \frac{\partial{\funcestmode_{n, \radius}}}{\partial{y}}(x, y) = 0, \ \frac{\partial^2{\funcestmode_{n, \radius}}}{\partial{y}^2}(x, y) < 0 \right\}.
\end{align}
Similar to the mode clustering setting, we would like to establish the convergence rates of local modes in $\mathcal{M}_{n}(x)$ to those in $\mathcal{M}(x)$ based on the Hausdorff metric for all $x \in \mathcal{X}$. To facilitate the later discussion, we denote the modal manifold collection as follows:
\begin{align*}
    \mathcal{S} = \left\{(x, y): \ x \in \mathcal{X}, \ y \in \mathcal{M}(x) \right\}
\end{align*}
We impose the following assumption with $\mathcal{S}$, which had been employed in the previous work~\citep{Yen_Chen_2016}: 
\begin{assumption}
\label{assume:manifold_regression}
The modal manifold collection $\mathcal{S} = \cup_{i = 1}^{K} S_{i}$ where the modal manifold $S_{i} = \{(x, m_{i}(x)): \ x \in A_{i}\}$ for some modal function $m_{i}$ and open set $A_{i}$.
\end{assumption}
The Assumption~\ref{assume:manifold_regression} is to guarantee that the number of local modes of $p(x, y)$ for each $x \in \mathcal{X}$ is finite. Furthermore, under this assumption, we can rewrite $\mathcal{M}(x)$ as follows:
\begin{align*}
    \mathcal{M}(x) = \{m_{1}(x), \ldots, m_{K}(x)\}.
\end{align*}
When the true density $p_{0}$ is second order differentiable, the modal functions $m_{i}$ are also differentiable and the set of local modes $\mathcal{M}(x)$ is smooth under Hausdorff metric. To guarantee that the decomposition of the modal manifold collection $\mathcal{S}$ in Assumption~\ref{assume:manifold_regression} is unique, we need the following non-degenerate assumption regarding the curvature around the critical points, i.e., those when $\frac{\partial{p_{0}}}{\partial{y}} (x, y) = 0$:
\begin{assumption}
\label{assume:Hessian_manifold_regression}
For any $(x, y) \in \mathcal{X} \times \mathcal{Y}$ such that $\frac{\partial{p_{0}}}{\partial{y}} (x, y) = 0$, we have $|\frac{\partial^2{p_{0}}}{\partial{y}^2} (x, y)| \geq \curvature$ where $\curvature > 0$ is some universal constant.
\end{assumption}

\noindent
Given Assumptions~\ref{assume:manifold_regression} and~\ref{assume:Hessian_manifold_regression} at hand, we have the following result regarding the uniform convergence rate of $\mathcal{M}_{n}(x)$ to $\mathcal{M}(x)$ under the Hausdorff distance:
\begin{proposition}
\label{theorem:mode_regression_super_smooth}
Assume that Assumptions~\ref{assume:manifold_regression} and~\ref{assume:Hessian_manifold_regression} hold. Furthermore, $p_{0} \in \mathcal{C}^3(\mathcal{X} \times \mathcal{Y})$ where $\mathcal{X}$ and $\mathcal{Y}$ are bounded subsets of $\mathbb{R}^{d}$ and $\mathbb{R}$ respectively. Then, the following holds:

\noindent
(a) When $p_{0}$ is an upper-supersmooth density function of order $\alpha > 0$, there exists universal constant $C$ such that
\begin{align*}
    \Prob \parenth{ \sup_{x \in \mathcal{X}} \mathcal{H}(\mathcal{M}_{n}(x), \mathcal{M}(x)) \leq C \brackets{\radius^{\max \{2 - \alpha, 0\}} \exp ( - C_{1} \radius^{\alpha}) + \sqrt{\frac{\radius^{d + 3} \log \radius \log(2/ \delta)}{n}}}} \geq 1 - \delta. 
\end{align*}
Here, $C_{1}$ is a given constant associated with upper--supersmooth density function in Definition~\ref{def:tail_Fourier}.

\noindent
(b) When $p_{0}$ is an upper-ordinary smooth density function of order $\beta > 3$, there exists universal constant $c$ such that
\begin{align*}
    \Prob \parenth{ \sup_{x \in \mathcal{X}} \mathcal{H}(\mathcal{M}_{n}(x), \mathcal{M}(x)) \leq c \brackets{\radius^{2 - \beta} + \sqrt{\frac{\radius^{d + 3} \log \radius \log(2/ \delta)}{n}}}} \geq 1 - \delta. 
\end{align*}
\end{proposition}
\noindent
The proof of Proposition~\ref{theorem:mode_regression_super_smooth} is in Appendix~\ref{subsec:proof:theorem:mode_regression_super_smooth}.

The result of part (a) of Proposition~\ref{theorem:mode_regression_super_smooth} indicates that by choosing the radius $\radius$ such that $C_{1}\radius^{\alpha} = \log n/ 2$ where $C_{1}$ is given in part (a), we have $$\sup_{x \in \mathcal{X}} \mathcal{H}(\mathcal{M}_{n}(x), \mathcal{M}(x)) = \mathcal{O}_{P} \parenth{\frac{(\log n)^{\max \{\frac{2}{\alpha} - 1, \frac{d + 3}{2\alpha}\}}}{\sqrt{n}}}.$$ Therefore, we can estimate the local modes of $\mathcal{M}(x)$ with parametric rate when the joint density function $p_{0}$ of $(X, Y)$ is supersmooth. That parametric rate is also faster than the rate $n^{-2/(d+7)}$ from kernel density estimator in~\citep{Yen_Chen_2016}. On the other hand, when $p_{0}$ is upper-ordinary smooth density function, by choosing the radius $\radius$ such that $\radius = n^{1/(2\beta + d - 1)}$, the result of part (b) shows that the rate of $\sup_{x \in \mathcal{X}} \mathcal{H}(\mathcal{M}_{n}(x), \mathcal{M}(x))$ is at the order of $\sqrt{\log n}\,n^{-(\beta - 2)/(2\beta + d - 1)}$. It is also faster than the rate $n^{-2/(d+7)}$ from kernel density estimator in~\citep{Yen_Chen_2016}.

Furthermore, since the results of Proposition~\ref{theorem:mode_regression_super_smooth} hold for all $x \in \mathcal{X}$, the conclusions in part (a) and (b) still hold for $\int_{x \in \mathcal{X}} \mathcal{H}(\mathcal{M}_{n}(x), \mathcal{M}(x))$, i.e., the MISE of $\mathcal{H}(\mathcal{M}_{n}(x), \mathcal{M}(x))$. Finally, we also can construct the confidence interval and band for $\mathcal{H}(\mathcal{M}_{n}(x), \mathcal{M}(x))$ based on the previous argument with confidence interval and band in Section~\ref{sec:uniform_confidence_Fourier_density}.
%%%%%%%%%%%%%%%%%%%%%%%%%%%%%%%%%%%%%%%%%%%%%%%%%%%%%%%%%%%%%%%%%%%%%%%%%%%%%%%%%%%%%%%%%%%%%%%%%%%%%%%%%%%%%%
\section{Dependent data}
\label{sec:dependent_data}
In this section, we discuss an application of Fourier integral theorem to estimate the Markov transition probability when the data $X_{1}, \ldots, X_{n} \in \mathcal{X} \subseteq \mathbb{R}^{d}$ are a Markov sequence with stationary density function $p_{0}$ and transition probability distribution $f(\cdot\mid\cdot)$. This relies specifically on the Fourier integral theorem and the Monte Carlo estimate and the ergodic theorem. A unique combination involving the Fourier kernel. 

For the density function $p_{0}$, we can use the Fourier density estimator $\funcest_{n, \radius}$ in equation~\eqref{eq:multivariate_density_estimator}. Since we can write $f(y \mid x) = p(x, y)/ p_{0}(x)$ where $p(\cdot,\cdot)$ is the joint stationary density of $(X_{i}, X_{i + 1})$, we can also use the Fourier density estimator to estimate the joint stationary density $p$. An estimate of the transition probability distribution based on the Fourier integral theorem is
\begin{align}
    \transitionmat_{n, \radius}(y\mid x) : = \frac{1}{\pi^{d}} \frac{\sum_{i=1}^{n - 1} \prod_{j=1}^d\frac{\sin(R(x-X_{ij}))}{x-X_{ij}} \cdot \frac{\sin(R(y-X_{(i+1)j}))}{y-X_{(i+1)j}}} {\sum_{i=1}^n \prod_{j=1}^d\frac{\sin(R(x-X_{ij}))}{x-X_{ij}}}. \label{eq:transition_prob_estimator}
\end{align}
We refer the estimator $\transitionmat_{n, \radius}$ to as \emph{Fourier transition estimator}. To study the MSE of the Fourier transition estimator $\transitionmat_{n, \radius}(x)$ for each $x \in \mathcal{X}$, we impose a mixing condition on the transition probability function of the Markov sequence $(X_{1}, \ldots, X_{n})$. In particular, we define the following transition probability operator
%\begin{align*}
   $ (\mathcal{T}h)(x) : = \int h(y) f(y\mid x)dy$,
%\end{align*}
for any bounded function $h:\mathcal{X} \to \mathbb{R}$. Then, we denote the $\mathbb{L}_{2}$ norm of the operator $\mathcal{T}$ as follows: 
\begin{align*}
    |\mathcal{T}|_{2} = \sup_{h \neq 0} \frac{\|\mathcal{T}h - \Exs \brackets{h(X)}\|_{2}}{\|h - \Exs \brackets{h(X)}\|_{2}},
\end{align*}
where the expectations are taken with respect to $X \sim p_{0}$ and $\|h\|_{2}^2 = \int (h(x))^{2} p_{0}(x) dx$. It is clear that $|\mathcal{T}^{j}|_{2} \leq 1$ for all $j \in \mathbb{N}$. We impose the following assumption on the transition probability operator $\mathcal{T}$ so as to guarantee  geometric ergodicity~\citep{Yakowitz_1985, Rosenblatt_2011}:
\begin{assumption}
\label{assume:strong_mixing_transition_prob} 
There exist $\tau \in \mathbb{N}$ and $\eta \in (0, 1)$ such that the transition probability operator $\mathcal{T}$ satisfies 
%\begin{align*}
   $ |\mathcal{T}^{\tau}|_{2} \leq \eta$.
%\end{align*}
\end{assumption}

\noindent
As an example, and as pointed out in \cite{Rosenblatt_2011}, Assumption~\ref{assume:strong_mixing_transition_prob} is satisfied when the stationary density function 
%$p_{0}(x) = \frac{1}{(2\pi)^{d/2}} \prod_{j = 1}^{d} \exp( - x_{j}^2/2)$, i.e., $p_{0}$ 
is a standard multivariate Gaussian distribution and the transition probability density is
\begin{align}
f(y\mid x) = \frac{1}{(2\pi)^{d/2}} \prod_{j = 1}^{d} \frac{1}{\sqrt{(1 - \eta_{j}^2)}} \exp (-(y_{j} - \eta_{j} x_{j})^2/(2(1 - \eta_{j}^2))), \label{eq:example_transition_prob}
\end{align}
for some $\eta_{1}, \ldots, \eta_{d} \in (0, 1)$. Then, we can verify that $|\mathcal{T}|_{2} \leq \prod_{j = 1}^{d} \eta_{j}^2$.

For the simplicity of the presentation of the results, we only focus on studying the MSE of $\transitionmat_{n, \radius}(x)$ when both the stationary density function $p_{0}$ and the stationary joint density function $p$ are upper--supersmooth.
\begin{theorem}
\label{theorem:bias_variance_transition_prob_distribution}
Assume that the stationary density and joint density functions $p_{0}$ and $p$ are respectively upper--supersmooth density functions of order $\alpha_{1} > 0$ and $\alpha_{2} > 0$, such that $\max\{ \|p_{0}\|_{\infty}, \|p\|_{\infty}\} < \infty$. Furthermore, the transition probability operator $\mathcal{T}$ satisfies Assumption~\ref{assume:strong_mixing_transition_prob}. Then, for each $x, y \in \mathcal{X}$, there exist universal constants $C_{1}, C_{2}, c_{1}, c_{2}$ such that as long as $\radius \geq C$ for some universal constant $C$, we have 
\begin{align*}
    \Exs \brackets{(\transitionmat(y\mid x) - f(y\mid x))^2} \leq \frac{C (p_{0}^2(x) + p^2(x, y))}{p_{0}^4(x) \bar{J}(R)} \parenth{\radius^{\max \{2(1 - \bar{\alpha}), 0 \}} \exp \parenth{-C_{1} \radius^{\bar{\alpha}}} + \frac{\radius^{2d}}{n}},
\end{align*}
where $\bar{\alpha} = \min \{\alpha_{1}, \alpha_{2}\}$ and $\bar{J}(R) = 1 - c R^{\max \{2 - 2 \alpha_{1}, 0\}} \exp \parenth{-c_{1} \radius^{\alpha_{1}} } + \frac{\radius^{d} \log (n \radius)}{n}/p_{0}^2(x)$.
\end{theorem}
\noindent
The proof of Theorem~\ref{theorem:bias_variance_transition_prob_distribution} is in Section~\ref{subsec:proof:theorem:bias_variance_transition_prob_distribution}.

A few comments with Theorem~\ref{theorem:bias_variance_transition_prob_distribution} are in order. First, the assumptions of Theorem~\ref{theorem:bias_variance_transition_prob_distribution} are satisfied when $p_{0}$ is standard multivariate Gaussian distribution and the transition probability distribution $f(.|.)$ takes the form~\eqref{eq:multivariate_regression_example}. Under this example, both the stationary density and joint density functions $p_{0}$ and $p$ are upper--supersmooth of second order. Second, the result of Theorem~\ref{theorem:bias_variance_transition_prob_distribution} indicates that we can choose the radius $\radius$ such that $\radius^{\bar{\alpha}} = \mathcal{O}(\log n)$. Then, given that choice of $\radius$, the MSE rate of the Fourier transition estimator is at the order $(\log n)^{\max \{2(1 - \bar{\alpha}), 2d \}}/ n$. It is faster than the MSE rate $n^{-1/(2d + 4)}$ of kernel density estimator for estimating transition probability density function from Markov sequence data~\citep{Yakowitz_1985}. Finally, since the Fourier transition estimator $\transitionmat_{n, \radius}$ is constructed based on Fourier integral theorem, it already preserves the dependence structure of the Markov sequence data. It is different from the standard kernel density estimator where the choice of covariance matrix is non-trivial to choose.

We note in passing that the idea of Fourier integral theorem can also be adapted to the nonparametric regression for Markov sequence in the similar fashion as when the data are independent in Section~\ref{subsec:nonparametric_regression}. We leave a detailed development of this direction for the future work.
%%%%%%%%%%%%%%%%%%%%%%%%%%%%%%%%%%%%%%%%%%%%%%%%%%%%%%%%%%%%%%%%%%%%%%%%%%%%%%%%%%%%%%%%%%%%%%%%%%%%%%%%%%%%
\section{Illustrations}
\label{sec:illustrations}
In this section, we provide experimental results illustrating the performance of Fourier estimators developed in the previous sections. In the first one we highlight the difference between using the Gaussian kernel and the Fourier kernel. This is in the multivariate setting and in many instances, such as~\citep{Yen_Chen_2016}, even if there is a dependence between variables, a product of independent Gaussian kernels is used. On the other hand, a consequence of the special Fourier kernel and its connection with the Fourier intergral theorem, a product of independent Fourier kernels work and are adequate even when modeling dependent variables. 

The next two examples involve multidimensional regression models. To report the good estimation properties using the Fourier integral we present a curve on the surface of the regression function. We also consider estimation of a mixing density, specifically the gradient of the density which would allow us to search for the modes, opening up the possibility of modal regression. A further example indeed is concerned with modal regression. We conclude the section with dependent data, specifically Markov sequence data.

\subsection{Example 1.}
First we make a comparison between the Fourier regression estimator and the multivariate Gaussian estimator based on a diagonal covariance matrix. With the sample size $n=1000$, we generate the data from the model with 
$(X_{i1})$ as independent standard normal and $X_{i2}=X_{i1}+0.1\times Z_i$, where the $(Z_i)$ are also independent standard normal. Then
$$Y_i=X_{i1}^2-3X_{i2}+\epsilon_i,\quad\epsilon_i\sim\mbox{standard normal}.$$
We then compare the Fourier kernel  estimator $\widehat{m}_R(x)$ in equation~\eqref{multreg} when $R=9$ with the Gaussian kernel regression estimator
$$\widehat{m}_h(x)=\frac{\sum_{i=1}^n Y_i\,K_h(x_1-X_{i1})K_h(x_{2}-X_{i2})}
{\sum_{i=1}^n K_h(x_{1}-X_{i1})K_h(x_2-X_{i2})},$$
with $K_h(u)=h^{-1}\exp(-u^2/(2h^2))$. We use the literature recommended choice of $h=n^{-1/(4+d)}=n^{-1/6}$. The issue is that the denominator is attempting to estimate the joint density of $(x_1,x_2)$ from the sample and, without a covariance matrix modeling the dependence, $\widehat{m}_h$ will struggle to provide a decent estimator~\citep{WJones93, Wand_kernel}.

In this simple illustration we compare the estimators evaluated at $x=(1,2)$; the true value being $-5$.  
We repeated the experiments 1000 times and hence for each estimator we have 1000 sample estimates for this true value. The histogram representation of the two sets of samples are presented in Fig.~\ref{fig1}. As can be seen, the samples from the Fourier kernel are centered about 5; while those from the Gaussian kernel are clearly wrong.

\begin{center}
\begin{figure}[!htbp]
\begin{center}
\includegraphics[width=14cm,height=10cm]{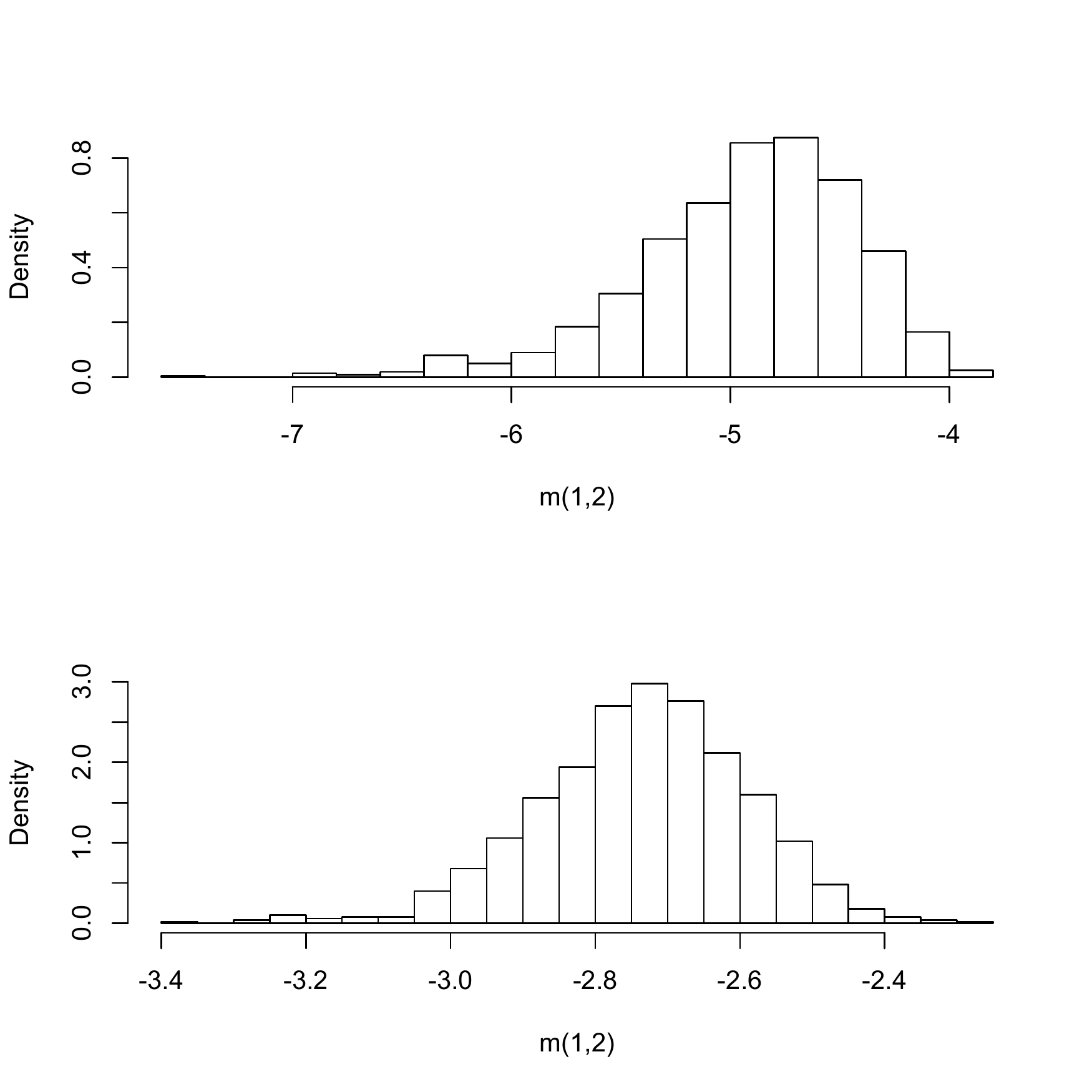}
\caption{Top: Histogram of $\widehat{m}_R(1,2)$ samples, Bottom: Histogram of $\hat{m}_h(1,2)$.}
\label{fig1}
\end{center}
\end{figure}
\end{center} 

To highlight the point about the dependence between $X_1$ and $X_2$; without any, so we can generate them as two independent standard normals, the Gaussian kernel estimator performs much better.

\subsection{Example 2.} In this example we take the dimension $d=4$ and generate the data from
\begin{align}
y_i=\sum_{j=1}^d a_j\,x_{ij}+0.01\epsilon_i, \label{eq:multivariate_regression_example}
\end{align}
and take $n=10^6$. Here the $(x_{ij})$ are taken as independent standard normal and $a_j=j/4$. We then estimate a particular curve for $-0.4<t<0.4$ with
$$x_1=\sqrt{t+2},\quad x_2=t,\quad x_3=\sin(25(t+2)/\pi),\quad x_4=\exp((t+2)/4).$$
So we are estimating the curve
$m(x)=m(x_1(t),x_2(t),x_3(t),x_4(t))$
and comparing with the true one.
\begin{figure}[t]
\centering
\begin{subfigure}[t]{0.48\textwidth}
\includegraphics[width=1\textwidth]{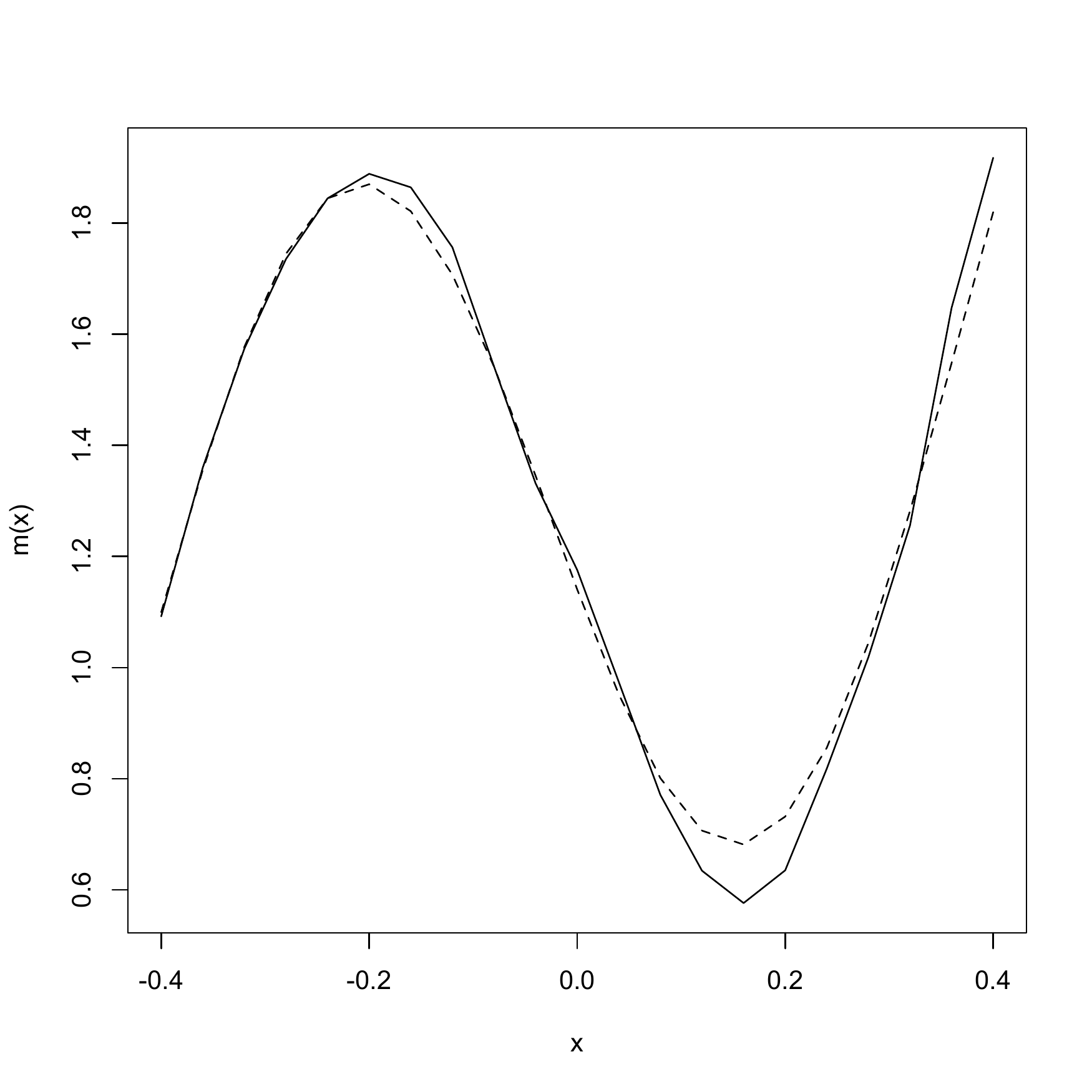}
\caption{}
\end{subfigure}
\begin{subfigure}[t]{0.48\textwidth}
\includegraphics[width=1\textwidth]{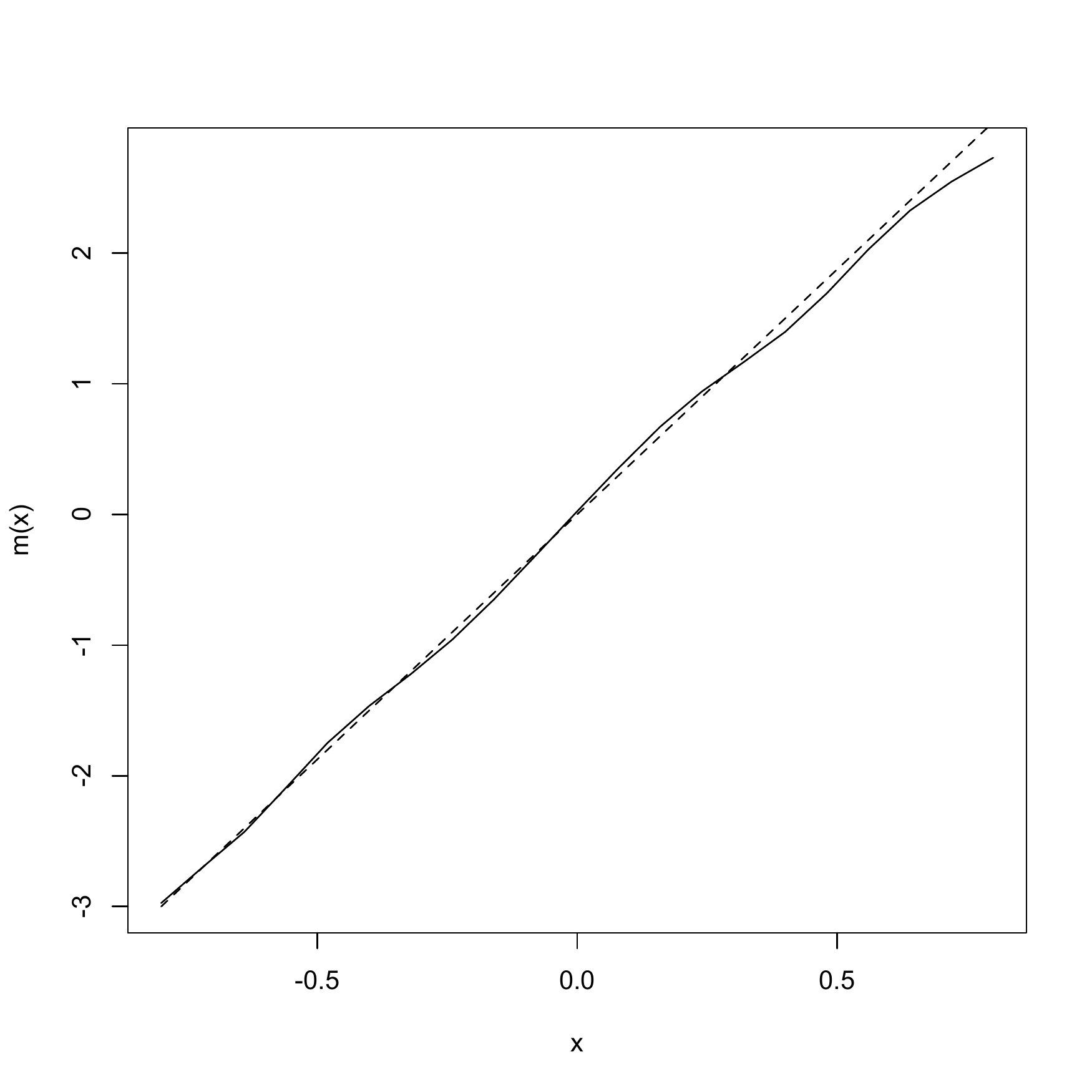}
\caption{}
\end{subfigure}
\caption{
Simulations with the Fourier regression estimator~\eqref{nonregfou} for nonparametric regression model~\eqref{eq:multivariate_regression_example} when $d \in \{4, 5\}$. In both figures, the estimated and true regression functions are respectively represented in bold and dashed lines. (a) $d = 4$; (b) $d = 5$.
}
\label{fig:nonparametric_regression}
\end{figure}

%\begin{center}
%\begin{figure}[!htbp]
%\begin{center}
%\includegraphics[width=14cm,height=10cm]{Fig3}
%\caption{Estimated (bold line) and true curve (dashed line)}
%\label{fig3}
%\end{center}
%\end{figure}
%\end{center} 

The Fourier regression estimator is provided by equation~\eqref{nonregfou} with $R=7$. The Fig.~\ref{fig:nonparametric_regression}(a) presents the estimated curve (bold line) alongside the true curve (dashed line). 

\subsection{Example 3.} Here we present a similar example to Example 2 except now we extend the dimension to 5, take $n=100,000$. All other aspects are the same as in Example 2, though now we estimate the line curve $m(x)$ with $x=(x_1,x_2,x_3,x_4,x_5)$ and
$x_1=x_2=x_3=x_4=x_5=t$,
with $-0.6<t<0.6$.

%\begin{center}
%\begin{figure}[!htbp]
%\begin{center}
%\includegraphics[width=14cm,height=10cm]{Fig2}
%\caption{Estimated (bold line) and true curve (dashed line)}
%\label{fig2}
%\end{center}
%\end{figure}
%\end{center} 

Again, the Fourier regression estimator is provided in equation~\eqref{nonregfou} with $R=5$. The Fig.~\ref{fig:nonparametric_regression}(b) presents the estimated curve (bold line) alongside the true curve (dashed line).

\subsection{Example 4.} In this example we are investigating the problem of estimating mixing density with a normal kernel. The data model is given by
$$p(x)=\int f(x-\theta)\,g(\theta)\,d\theta$$
where $f(x-\theta)$ is a normal kernel with a fixed variance (the standard deviation $h$ is set at $h=0.1$) and location $\theta$. We focus on obtaining the derivative of $g$; i.e., $g'(\theta)$ for the purposes of obtaining the modes of $g$. So specifically identifying the $\theta$ values (in increasing order the odd values) for which $g'(\theta)=0$. 
The density estimator we use is a modification to the Fourier deconvolution estimator~\eqref{eq:decon_mixture};
$$\widehat{g}_{n,R}(\theta)=\frac{R}{n\pi} \sum_{i=1}^n e^{u_i^2 h^2/2}\,\cos(u_i(\theta-x_i))$$
where the $(x_i)$ are the observed sample from $p(x)$, and the $(u_i)$ are independent samples from the uniform distribution on $(0,R)$, with $R=5$. Hence, straightforwardly
we get
$$\widehat{g}_{n,R}'(\theta)=\frac{-R}{n\pi} \sum_{i=1}^nu_i\, e^{u_i^2 h^2/2}\,\sin(u_i(\theta-x_i)).$$
We present an illustration in Fig.~\ref{fig:mode_estimation}(a), where we compare with the true $g'(\theta)$ which is 
$$g(\theta)=0.6\,N(\theta\mid -2,0.6^2)+0.4\,N(\theta\mid 2,0.6^2).$$
As indicated in Fig.~\ref{fig:mode_estimation}(a), $\widehat{g}_{n,R}'(\theta)$ gives a good estimate of $g'(\theta)$.
\begin{figure}[t]
\centering
\begin{subfigure}[t]{0.48\textwidth}
\includegraphics[width=1\textwidth]{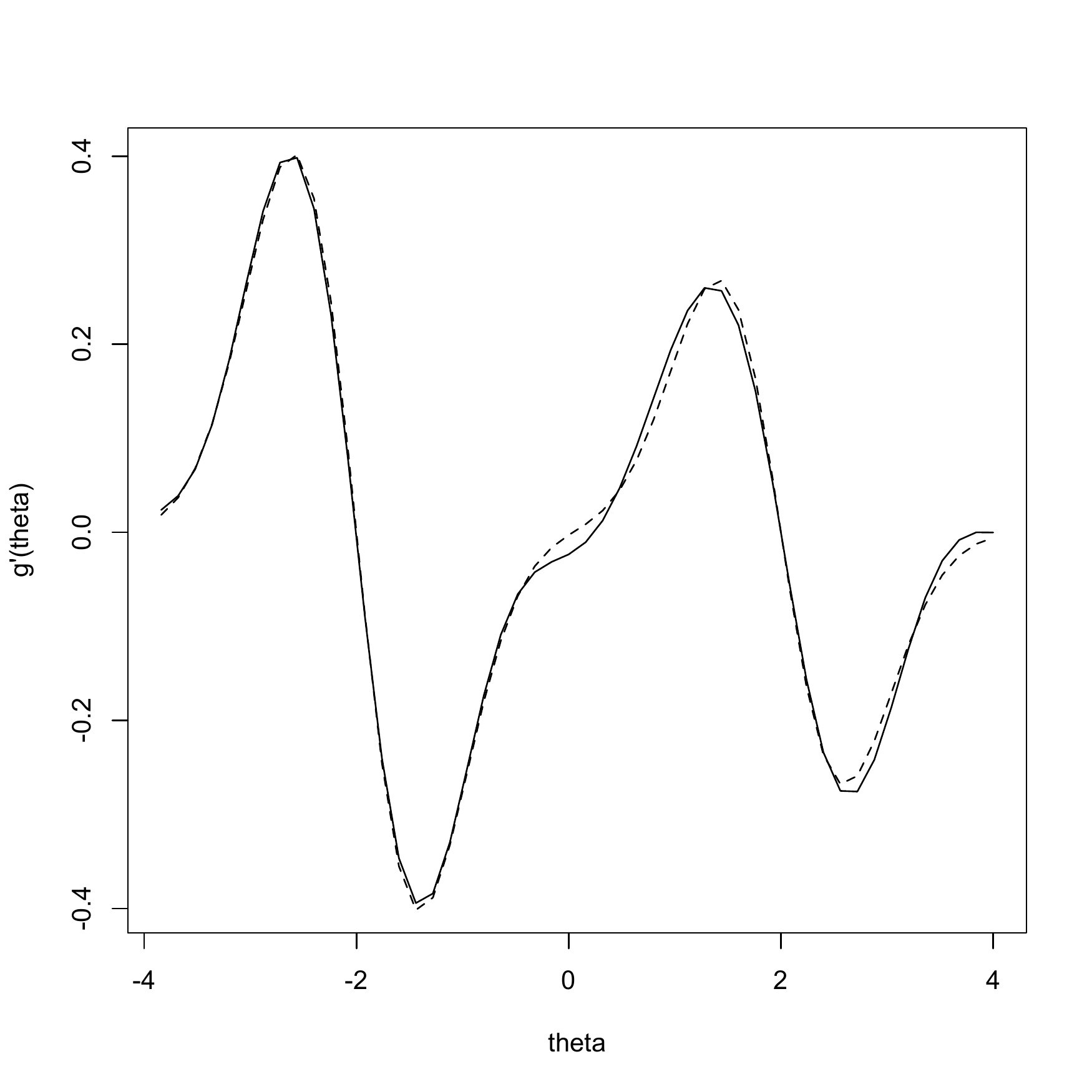}
\caption{}
\end{subfigure}
\begin{subfigure}[t]{0.48\textwidth}
\includegraphics[width=1\textwidth]{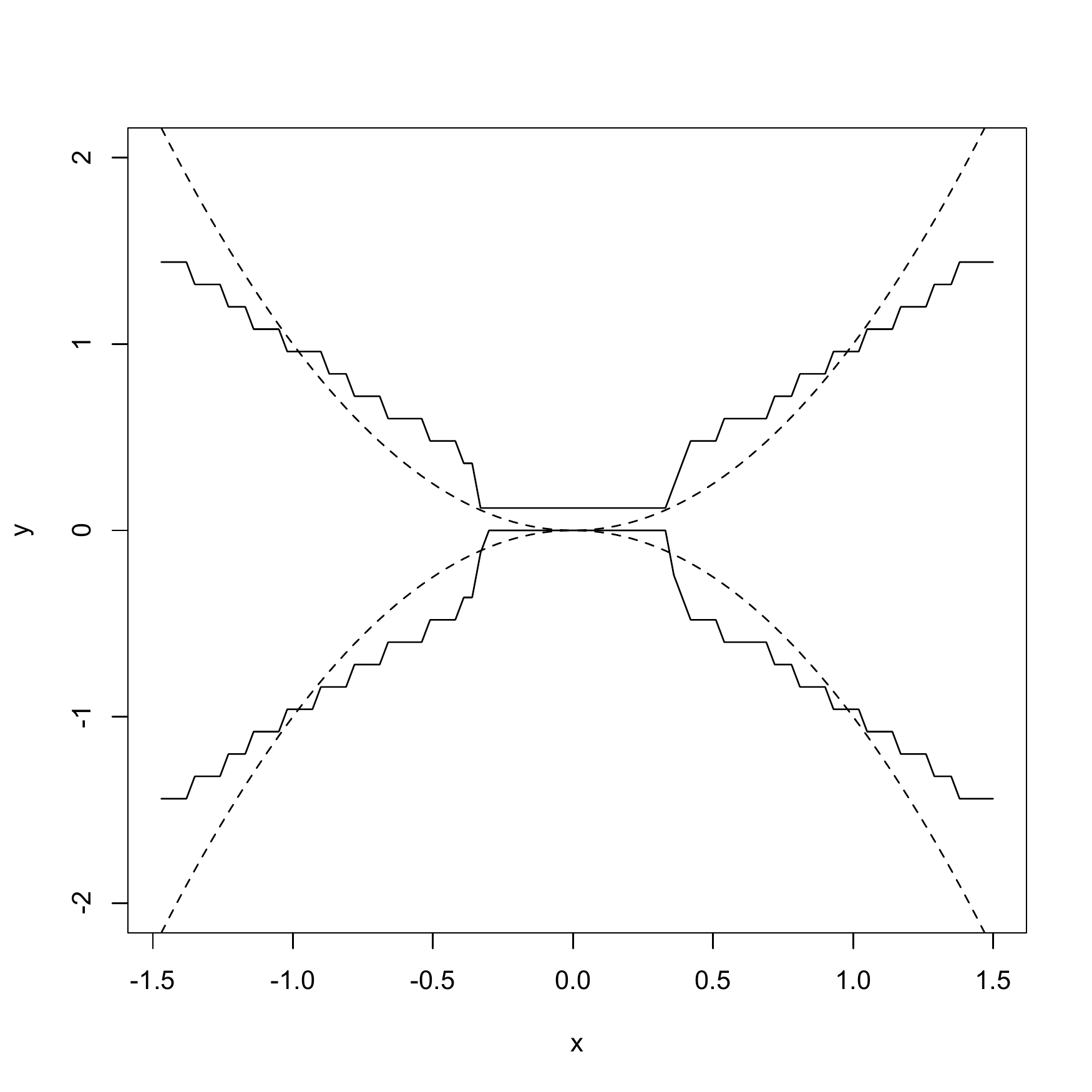}
\caption{}
\end{subfigure}
\caption{
Simulations with Fourier mode estimators. (a) We consider estimating modes of mixing density. The estimated first order derivative of mixing density $\widehat{g}_{n,R}'(\theta)$ is in bold line while the first order derivative of true mixing density $g'(\theta)$ is in dashed line. (b) We illustrate mode estimation from nonparametric modal regression problem. The true modes are represented in dashed lines while the estimated modes are in bold line.
}
\label{fig:mode_estimation}
\end{figure}
%\begin{center}
%\begin{figure}[!htbp]
%\begin{center}
%\includegraphics[width=14cm,height=10cm]{Fig4}
%\caption{Estimated $\widehat{g}_{n,R}'(\theta)$, bold line, and true %$g'(\theta)$, dashed line.}
%\label{fig4}
%\end{center}
%\end{figure}
%\end{center} 
%\begin{center}
%\begin{figure}[!htbp]
%\begin{center}
%\includegraphics[width=14cm,height=10cm]{Fig5}
%\caption{True modes, dashed lines, and estimated modes, bold line. }
%\label{fig5}
%\end{center}
%\end{figure}
%\end{center} 
\subsection{Example 5.} In this example we look at nonparametric modal regression; see for example~\citep{Sager82} and~\citep{Yen_Chen_2016}. For a regression model with conditional density $p(y\mid x)$, the idea is to find the modes given values of $x$. Of course, there may be more than a single mode for some $x$, which indeed separates modal regression from other types, such as mean regression, which yield a single answer. The possibly multiple modes can provide necessary information concerning $p(y\mid x)$.

In the example we take $p(y\mid x)$ as a bivariate normal density with modes at $-x^2$ and $+x^2$, and both with standard deviation 0.6, and with equal probability of $1/2$ assigned to each component. The estimate of the modes over a range of $x$ values is provided in Fig.~\ref{fig:mode_estimation}(b). In this example, the sample size was $n=10,000$, the data $(x_i)_{i = 1}^{n}$ we sampled uniformly from the interval $(-2,2)$, and the value of $R$ was 7. 

\begin{figure}[t]
\centering
\begin{subfigure}[t]{0.48\textwidth}
\includegraphics[width=1\textwidth]{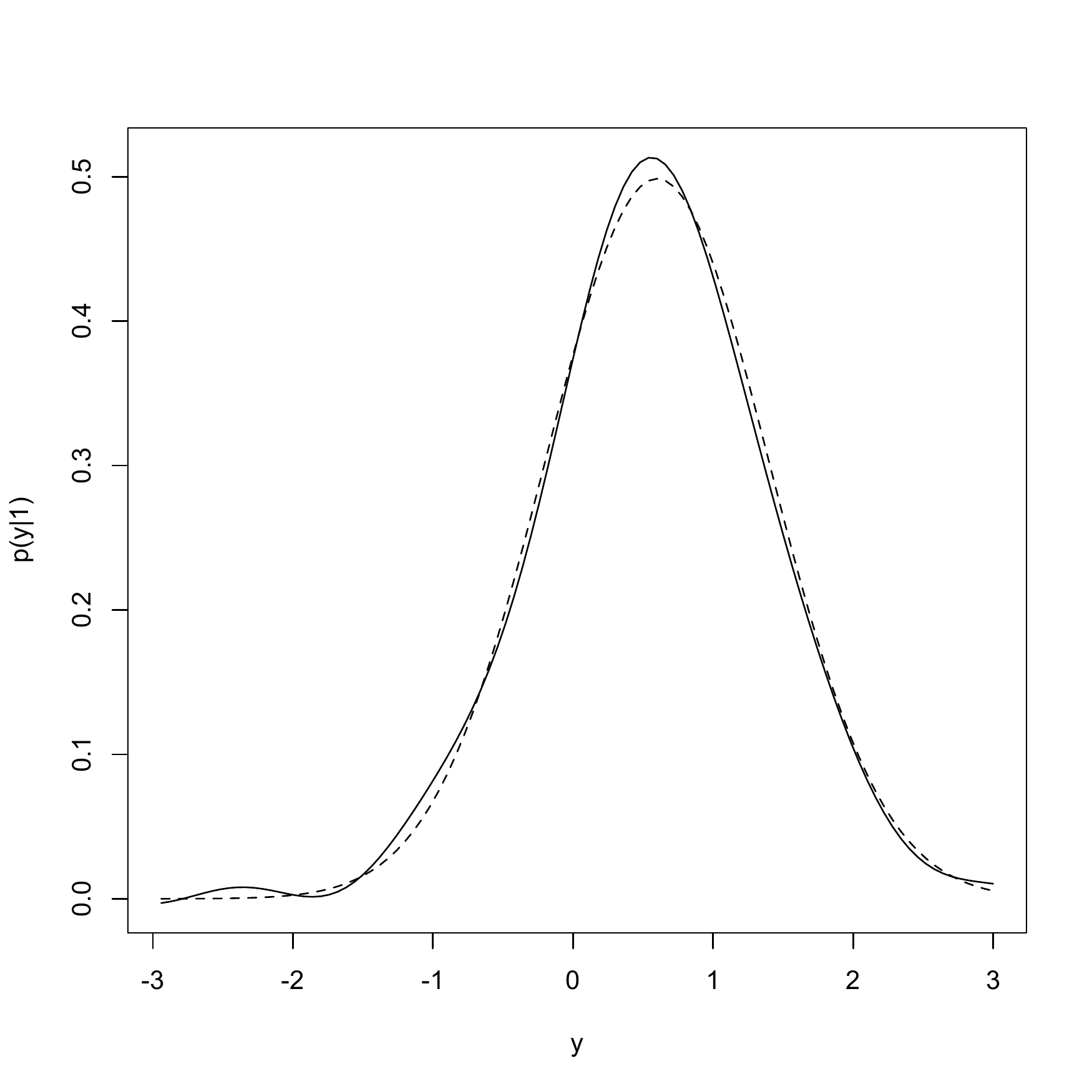}
\caption{}
\end{subfigure}
\begin{subfigure}[t]{0.48\textwidth}
\includegraphics[width=1\textwidth]{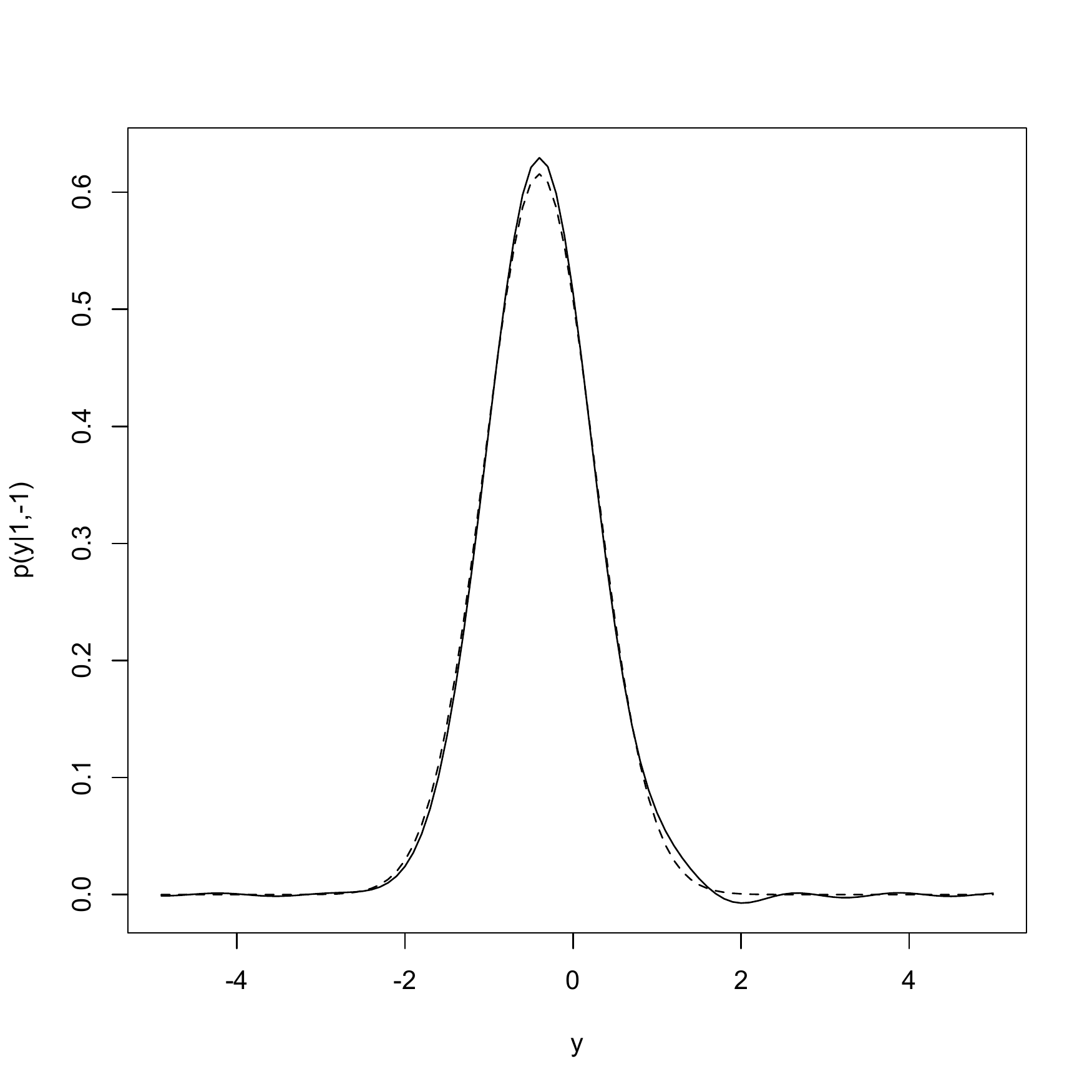}
\caption{}
\end{subfigure}
\caption{
Simulations with the Fourier transition estimator~\eqref{eq:transition_prob_estimator} for Markov sequences. In both figures, the estimated and true transition probabilities are respectively represented in bold and dashed lines. (a) One dimensional Gaussian Markov process; (b) Two dimensional Markov process~\eqref{eq:two_dimensional_process}.
}
\label{fig:transition_estimator}
\end{figure}

%\begin{center}
%\begin{figure}[!htbp]
%\begin{center}
%\includegraphics[width=14cm,height=10cm]{Fig6}
%\caption{Estimated (bold line) and true (dashed line) transition density in one--dimensional case}
%\label{fig6}
%\end{center}
%\end{figure}
%\end{center} 

%\begin{center}
%\begin{figure}[!htbp]
%\begin{center}
%\includegraphics[width=14cm,height=10cm]{Fig7}
%\caption{Estimated (bold line) and true (dashed line) transition density in two--dimensional case}
%\label{fig7}
%\end{center}
%\end{figure}
%\end{center} 

\subsection{Example 6.} Here we consider estimation of transition densities associated with a Markov sequence via the Fourier transition estimator~\eqref{eq:transition_prob_estimator}. The first case is a classic Gaussian Markov process
$$X_{n+1}=\rho X_n+\sqrt{1-\rho^2}Z_n,$$
where the $(Z_n)$ are independent standard normal random variables. The stationary density $p_{0}$ is well known to be the standard normal distribution. Starting with $X_0=\half$, we generated 10000 samples with $\rho=0.6$.

The true transition density $f(y\mid x)$ and its Fourier transition estimator are shown in Fig.~\ref{fig:transition_estimator}(a) with $x=1$.

The second case is a two--dimensional process $(X_{n1},X_{n2})$ given by:
\begin{align}
X_{n+1\,1} & =\rho X_{n\,1}+\sqrt{1-\rho^2}Z_{n\,1}, \nonumber \\
X_{n+1\,2} & =\rho_1\,X_{n\,1}+\rho_2X_{n\,2}+\sqrt{1-\rho_1^2-\rho_2^2}Z_{n\,2}, \label{eq:two_dimensional_process}
\end{align}
where the $(Z_{n\,1},Z_{n,2})$ are two independent sequences of standard normal random variables. In our simulation, we took $X_{0\,1}=0.5$ and $X_{0\,2}=0.2$ and $\rho=0.6$, $\rho_1=0.3$, and $\rho_2=0.7$, and $n=100000$.
The estimated transition density $f(y\mid x_1,x_2)$, also given by
(\ref{eq:transition_prob_estimator}), is shown in Fig.~\ref{fig:transition_estimator}(b) with $x_1=1$ and $x_2=-1$.

\subsection{Example 7.} 
In this subsection we use Fourier kernels on a real data set. The data set can be found in the R package \emph{fBasics} and consists of $n=9311$ data points of daily records of the NYSE Composite Index. A plot of the data is given in Fig.~\ref{fig10}.
\begin{figure}[!htbp]
\begin{center}
\includegraphics[width=14cm,height=6cm]{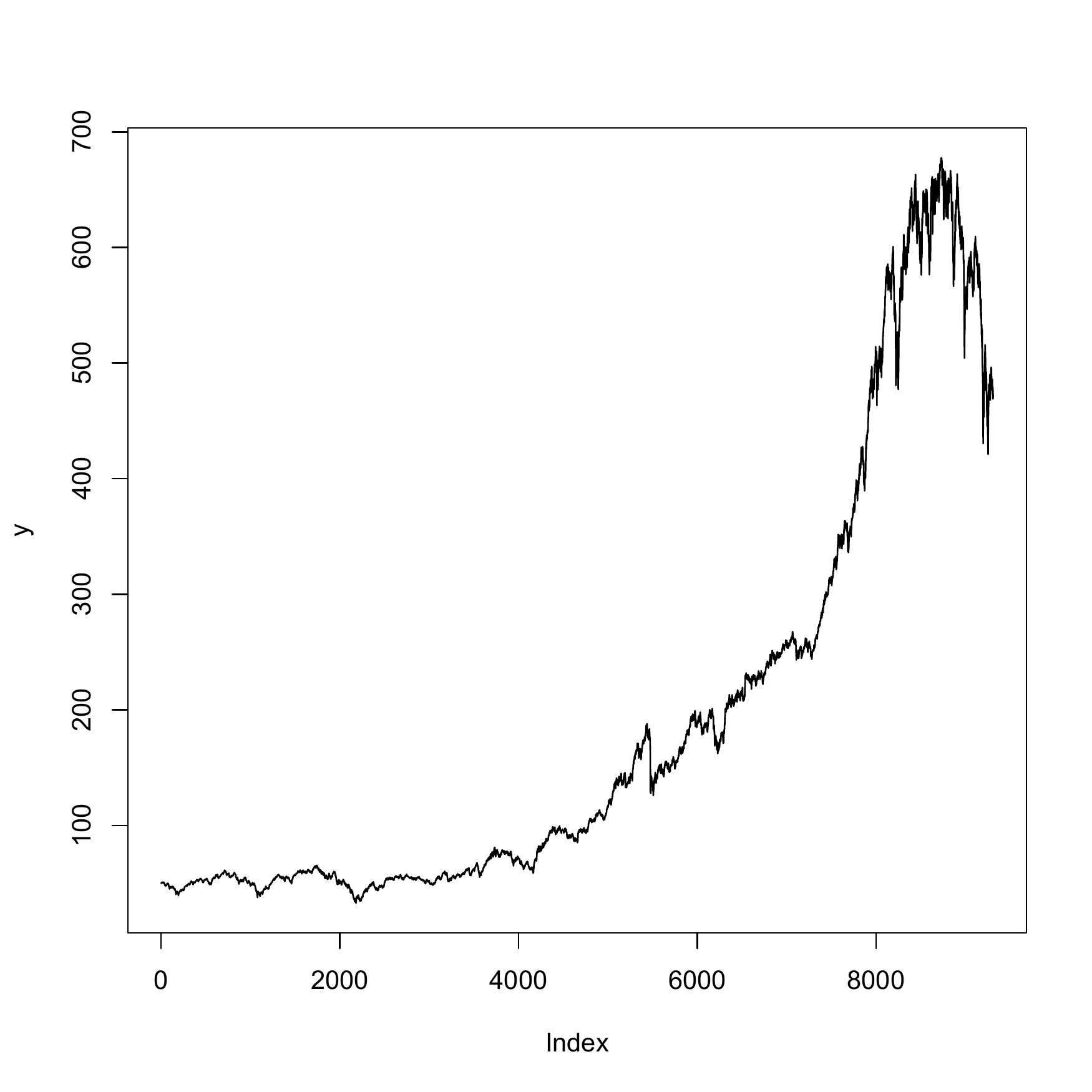}
\caption{Raw data of 9311 daily records of NYSE Composite Index.}
\label{fig10}
\end{center}
\end{figure}

We analyse the transformed data
$z_i=10\log(y_{i+1}/y_i)$, where $(y_i)$ are the raw data. This gives us a sample size of $n=9310$. First, we model the data $(z_i)$ using the Fourier kernels with the value of $\radius = 50$. The density estimator alongside a histogram of the $(z)$ samples is given in Fig.~\ref{fig:real_data}(a). 
\begin{figure}[t]
\centering
\begin{subfigure}[t]{0.48\textwidth}
\includegraphics[width=1\textwidth]{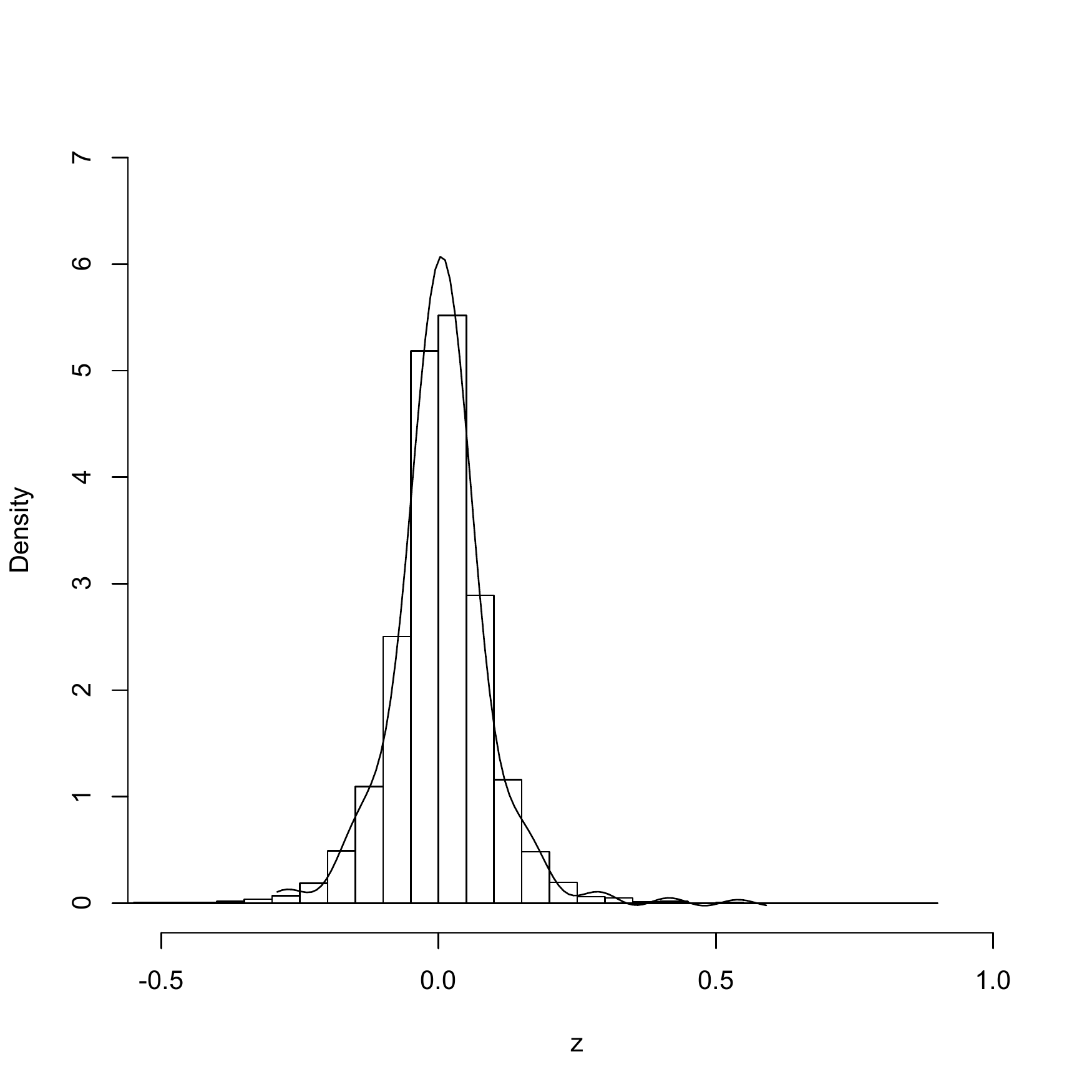}
\caption{}
\end{subfigure}
\begin{subfigure}[t]{0.48\textwidth}
\includegraphics[width=1\textwidth]{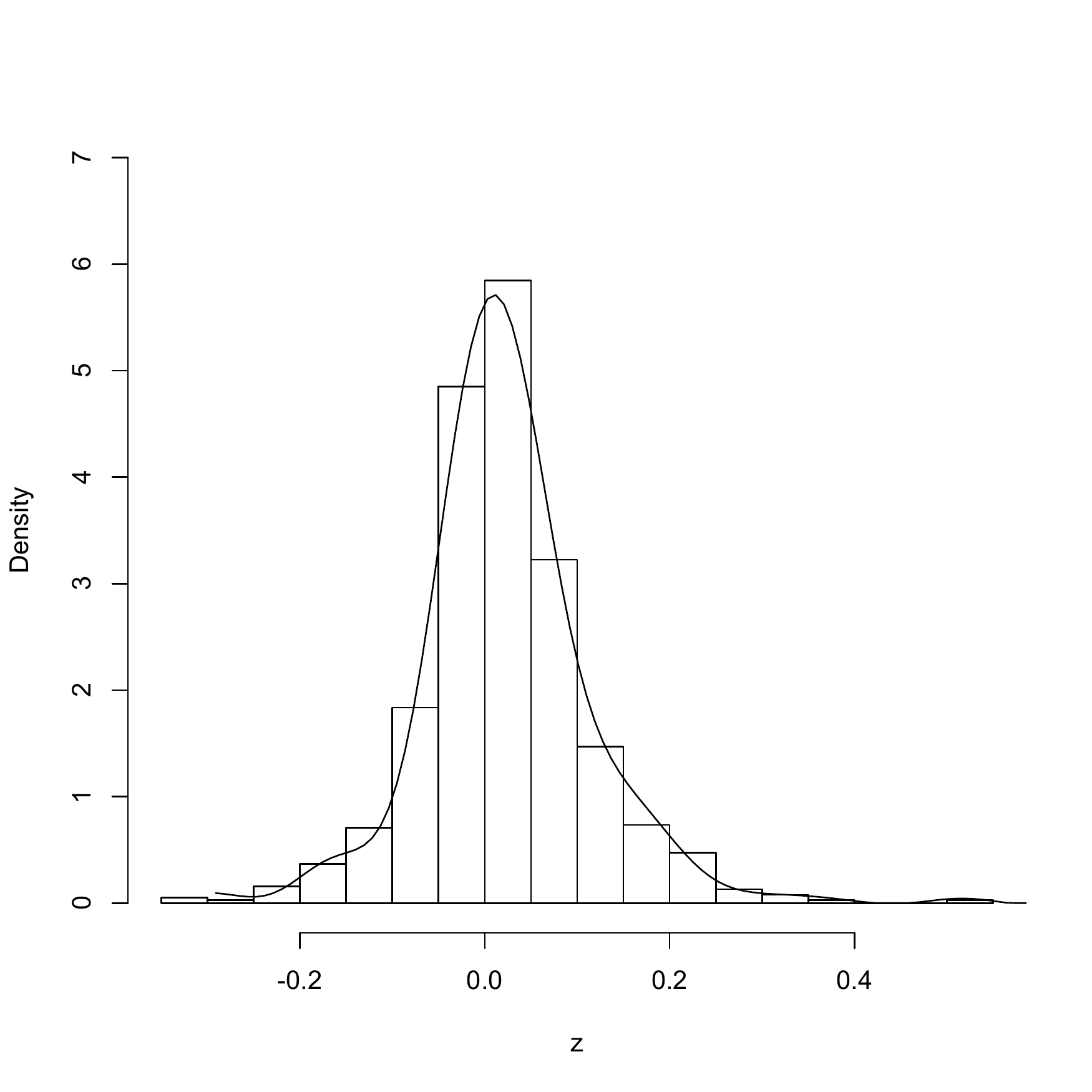}
\caption{}
\end{subfigure}
\caption{
Simulations with the Fourier density estimator~\eqref{eq:Fourier_integral} and Fourier transition estimator~\eqref{eq:transition_prob_estimator} for the NYSE Composite Index dataset. (a) Transformed data $(z_i)$ as histogram with density estimator using the Fourier kernel; (b) Histogram of conditional samples with conditional density estimator using the Fourier kernel.
}
\label{fig:real_data}
\end{figure}

%\begin{center}
%\begin{figure}[!htbp]
%\begin{center}
%\includegraphics[width=14cm,height=6cm]{Fig8}
%\caption{Transformed data $(z_i)$ as histogram with density %estimator using Fourier kernel}
%\label{fig8}
%\end{center}
%\end{figure}
%\end{center} 

%\begin{center}
%\begin{figure}[!htbp]
%\begin{center}
%\includegraphics[width=14cm,height=6cm]{Fig9}
%\caption{Histogram of conditional samples with conditional %density estimator using Fourier kernel}
%\label{fig9}
%\end{center}
%\end{figure}
%\end{center}
We than estimated the conditional density conditioning on the value of $0.15$. We obtained an approximate sample estimate of this by constructing the histogram of samples which have the immediately previous sample being an absolute value of  no more than a distance of 0.05 from $0.15$. The histogram sample along with our conditional density estimator is given in Fig.~\ref{fig:real_data}(b). The reason why there is little shift in the conditional density
from the marginal density is due to the low autocorrelation from the $(z_i)$ data. The data has a lag--1 autocorrelation of 0.1 and is negligible for lag--2. 

%\subsection{Example 4.} 

\section{Proofs}
\label{sec:proof}
In this section, we provide the proofs of the main results in the paper. The values of universal constants (e.g., $C$, $c'$ etc.) can change from line-to-line.
\subsection{Proof of Theorem~\ref{theorem:bias_variance_Fourier_density}}
\label{subsec:proof:theorem:bias_variance_Fourier_density}
Given the upper--supersmoothness or upper--ordinary smoothness of the density function $p_{0}$, its Fourier transform $\widehat{p}_{0}$ is integrable. Therefore, the Fourier inversion transform and integral theorem in equations~\eqref{eq:Fourier_integral} and~\eqref{eq:Fourier_inverstion} hold. An application of Fourier integral theorem leads to
\begin{align}
    \abss{ \Exs \brackets{\funcest_{n, \radius}(x)} - p_{0}(x)} & = \abss{\frac{1}{(2\pi)^{d}} \int_{\mathbb{R}^{d} \backslash [-\radius, \radius]^{d}} \int_{\mathbb{R}^{d}} \cos(s^{\top}(x - t)) p_{0}(t) ds dt} \nonumber \\
    & = \abss{\frac{1}{(2\pi)^{d}} \int_{\mathbb{R}^{d} \backslash [-\radius, \radius]^{d}} \brackets{ \cos(s^{\top} x) \text{Re}(\widehat{p}_{0}(s)) - \sin(s^{\top} x) \text{Im}(\widehat{p}_{0}(s))} ds} \nonumber \\
    & \leq \frac{1}{(2 \pi)^{d}} \int_{\mathbb{R}^{d} \backslash [-\radius, \radius]^{d}} \brackets{ \abss{\cos(s x)} \abss{\text{Re}(\widehat{p}_{0}(s))} + \abss{\sin(s x)} \abss{\text{Im}(\widehat{p}_{0}(s))}} ds \nonumber \\
    & \leq \frac{\sqrt{2}}{(2\pi)^{d}} \int_{\mathbb{R}^{d} \backslash [-\radius, \radius]^{d}} |\widehat{p}_{0}(s)| ds \leq \frac{\sqrt{2}}{(2\pi)^{d}} \sum_{i = 1}^{d} \int_{A_{i}} |\widehat{p}_{0}(s)| ds, \label{eq:upper_bound_bias}
\end{align}
where $\text{Re}(\widehat{p}_{0})$, $\text{Im}(\widehat{p}_{0})$ respectively denote the real and imaginary part of the Fourier transform $\widehat{p}_{0}$ and $A_{i} = \{x \in \mathbb{R}^{d}: |x_{i}| \geq R \}$ for all $i \in [d]$. Here, the second inequality is due to Cauchy-Schwarz inequality.

(a) When $p_{0}$ is upper--supersmooth density function of order $\alpha > 0$, we have
\begin{align*}
    \int_{A_{i}} |\widehat{p}_{0}(s)| ds  & \leq C \int_{A_{i}} \exp \parenth{ - C_{1} \parenth{ \sum_{i = 1}^{d} |s_{i}|^{\alpha}} } ds \\
    & = C \parenth{ \int_{-\infty}^{\infty} \exp(-C_{1} |t|^{\alpha})dt}^{d - 1} \cdot \int_{|t| \geq \radius} \exp(-C_{1} |t|^{\alpha})dt \\
    & = \frac{C \alpha^{d - 1}}{\parenth{2 C_{1}\Gamma(1/ \alpha)}^{d - 1}} \cdot \int_{|t| \geq \radius} \exp(-C_{1} |t|^{\alpha})dt,
\end{align*}
where $C$ and $C_{1}$ are universal constants from Definition~\ref{def:tail_Fourier} with upper-supersmooth density. If $\alpha \geq 1$, then $\int_{R}^{\infty} \exp \parenth{-C_{1} t^{\alpha}} dt \leq \int_{\radius}^{\infty} t^{\alpha - 1} \exp \parenth{-C_{1} t^{\alpha}} dt = \exp(-C_{1} \radius^{\alpha})/ (C_{1} \alpha)$. If $\alpha \in (0, 1)$, then we have
\begin{align*}
    \int_{\radius}^{\infty} \exp(-C_{1} t^{\alpha}) dt & = \int_{\radius}^{\infty} t^{1 - \alpha} t^{\alpha - 1} \exp(-C_{1} t^{\alpha}) dt \\
    & = \frac{\radius^{1 - \alpha} \exp \parenth{-C_{1}\radius^{\alpha}}}{C_{1} \alpha} + \frac{1 - \alpha}{C_{1} \alpha} \int_{\radius}^{\infty} t^{-\alpha} \exp(-C_{1} t^{\alpha}) dt \\
    & \leq \frac{\radius^{1 - \alpha} \exp \parenth{-C_{1}\radius^{\alpha}}}{C_{1} \alpha} + \frac{1 - \alpha}{C_{1} \alpha \radius^{\alpha}} \int_{\radius}^{\infty} \exp(-C_{1} t^{\alpha}) dt,
\end{align*}
where the first equality is due to the integration by part. By choosing $\radius$ such that $\radius^{\alpha} \geq \frac{2(1 - \alpha)}{C_{1} \alpha}$, the above inequality leads to
\begin{align*}
    \int_{\radius}^{\infty} \exp(-C_{1} t^{\alpha}) dt \leq \frac{2 \radius^{1 - \alpha} \exp \parenth{-C_{1}\radius^{\alpha}}}{C_{1} \alpha}.
\end{align*}
Putting the above results together, we obtain that
\begin{align*}
    \int_{|t| \geq \radius} \exp(-C_{1} |t|^{\alpha})dt \leq \frac{4 \radius^{\max \{1 - \alpha, 0\}}}{C_{1} \alpha} \exp(-C_{1}\radius^{\alpha}).
\end{align*}
Therefore, for each $i \in [d]$, we have the following upper bound:
\begin{align}
    \int_{A_{i}} |\widehat{p}_{0}(s)| ds \leq \frac{C \alpha^{d - 2} \radius^{\max \{1 - \alpha, 0\}}}{2^{d - 3} C_{1}^{d}(\Gamma(1/ \alpha))^{d - 1}} \exp(-C_{1} \radius^{\alpha}). \label{eq:upper_bound_supersmooth}
\end{align}
Combining the results from equations~\eqref{eq:upper_bound_bias} and~\eqref{eq:upper_bound_supersmooth}, we obtain that
\begin{align*}
    \abss{ \Exs \brackets{\funcest_{n, \radius}(x)} - p_{0}(x)} \leq \frac{\sqrt{2}C d \cdot \alpha^{d - 2} \radius^{\max \{1 - \alpha, 0\}}}{\pi^{d} 2^{2d - 3} C_{1}^{d}(\Gamma(1/ \alpha))^{d - 1}} \exp(-C_{1}\radius^{\alpha}).
\end{align*}
Therefore, we reach the conclusion with the upper bound of the bias of $\funcest_{n, \radius}(x)$ under the upper-supersmooth setting of the density function $p_{0}$.

Moving to the variance of $\funcest_{n, \radius}(x)$, we have
\begin{align*}
    \var \brackets{\funcest_{n, \radius}(x)} & = \frac{1}{n \pi^{2d}} \var \brackets{\prod_{i = 1}^{d} \frac{\sinfunc(\radius(x_{i} - X_{\cdot i}))}{x_{i} - X_{\cdot i}}} \leq \frac{1}{n \pi^2} \Exs \brackets{\prod_{i = 1}^{d} \frac{\sinfunc^2(\radius(x_{i} - X_{\cdot i}))}{(x_{i} - X_{\cdot i})^2}} \\
    & \leq \frac{\|p_{0}\|_{\infty}}{n \pi^{2d}} \parenth{\int_{-\infty}^{\infty} \frac{\sinfunc^2(\radius(x - t))}{(x - t)^2} dt}^{d} = \frac{\radius^{d} \|p_{0}\|_{\infty}}{n \pi^{d}},
\end{align*}
where the variance and the expectation are taken with respect to $X = (X_{\cdot 1}, \ldots, X_{\cdot d}) \sim p_{0}$.
As a consequence, we reach the conclusion of part (a).

(b) For part (b), the variance analysis is similar to that of part (a); therefore, it is omitted. For the bias of $\funcest_{n, \radius}(x)$, since the density function is upper--ordinary smooth of order $\beta$, for each $i \in [d]$ we obtain
\begin{align*}
    \int_{A_{i}} |\widehat{p}_{0}(s)| ds \leq c \int_{A_{i}} \prod_{j = 1}^{d}\frac{1}{(1 + |s_{j}|^{\beta})} ds = c \parenth{ \int_{-\infty}^{\infty} \frac{1}{1 + |t|^{\beta}} d t}^{d - 1} \cdot \int_{|t| \geq \radius} \frac{1}{1 + |t|^{\beta}} d t.
\end{align*}
Since $\beta > 1$, $I_{\beta} = \int_{-\infty}^{\infty} \frac{1}{1 + |t|^{\beta}} d t < \infty$. Furthermore, we obtain that
\begin{align*}
    \int_{|t| \geq \radius} \frac{1}{1 + |t|^{\beta}} d t \leq 2 \int_{\radius}^{\infty} \frac{1}{t^{\beta}} ds = \frac{2}{\beta - 1} \radius^{-\beta + 1}.
\end{align*}
Therefore, we have
\begin{align}
    \int_{A_{i}} |\widehat{p}_{0}(s)| ds \leq \frac{2 c I_{\beta}^{d - 1}}{\beta - 1} \radius^{1 - \beta}. \label{eq:upper_bound_ordinarysmooth}
\end{align}
Combining the results from equations~\eqref{eq:upper_bound_bias} and~\eqref{eq:upper_bound_ordinarysmooth}, we reach the conclusion with the bias of upper--ordinary smooth density $p_{0}$. 

%%%%%%%%%%%%%%%%%%%%%%%%%%%%%%%%%%%%%%%%%%%%%%%%%%%%%%%%%%%%%%%%%%%%%%%%%%%%%%%%%%%%%%%%%%%%%%%%%%%%%%%%%%%%%%%%%%
\subsection{Proof of Theorem~\ref{theorem:bias_variance_Fourier_density_derivatives}}
\label{subsec:proof:theorem:bias_variance_Fourier_density_derivatives}
We first compute $\Exs \brackets{\nabla^{i} \funcest_{n, \radius}(x)} - \nabla^{i} p_{{0}}(x)$ when $i \in \{1, \ldots, r\}$. Since $p_{0} \in \mathcal{C}^r(\mathcal{X})$, we have 
\begin{align*}
    \widehat{\partial^{\gamma} p_{0}}(s) = (i s)^{\gamma} \widehat{p}_{0}(s),
\end{align*}
for any $\gamma = (\gamma_{1},\ldots, \gamma_{d}) \in \mathbb{N}^{d}$ such that $|\gamma| \leq r$. Here, $\widehat{\partial^{\gamma} p_{0}}$ denotes the Fourier transform of the partial derivative $\frac{\partial^{\gamma} p_{0}}{\partial x^{\gamma}} (x)$. Given the upper--supersmoothness or lower--ordinary smoothness assumptions of $p_{0}$, it is clear that $\widehat{\partial^{\gamma} p_{0}}$ is integrable for all $\gamma = (\gamma_{1},\ldots, \gamma_{d})$ such that $|\gamma| \leq r$. Therefore, the Fourier inversion theorem is applicable to all the partial derivatives up to $r$-th order of $p_{0}$. It means that we have the following equations:
\begin{align*}
    \nabla^{i} p_{0}(x) = \frac{1}{(2\pi)^{d}} \int_{\mathbb{R}^{d}} \int_{\mathbb{R}^{d}} \nabla^{i} p_{0}(t) \cosfunc(s^{\top} (x - t)) dt d s
\end{align*}
for $i \in \{1, \ldots, r \}$. By means of integration by part, the above equations can be rewritten as follows:
\begin{align*}
    \nabla^{i} p_{0}(x) & = \frac{1}{(2\pi)^{d}} \int_{\mathbb{R}^{d}} \int_{\mathbb{R}^{d}} p_{0}(t) \nabla_{x}^{i} \cosfunc(s^{\top} (x - t)) dt d s.
\end{align*}
Therefore, we obtain that
\begin{align*}
    \parenth{ \nabla^{i} p_{0}(x)}_{u_{1}u_{2}\ldots u_{i}} = - \frac{1}{(2\pi)^{d}} \int_{\mathbb{R}^{d}} \int_{\mathbb{R}^{d}} s_{u_{1}}\ldots s_{u_{i}} \cdot \sinfunc(s^{\top} (x - t))  p_{0}(t) dt d s, \quad \text{if} \ i = 4 l + 1  \nonumber \\
    \parenth{ \nabla^{i} p_{0}(x)}_{u_{1}u_{2}\ldots u_{i}} = - \frac{1}{(2\pi)^{d}} \int_{\mathbb{R}^{d}} \int_{\mathbb{R}^{d}} s_{u_{1}}\ldots s_{u_{i}} \cdot \cosfunc(s^{\top} (x - t))  p_{0}(t) dt d s, \quad \text{if} \ i = 4 l + 2  \nonumber \\
    \parenth{ \nabla^{i} p_{0}(x)}_{u_{1}u_{2}\ldots u_{i}} =  \frac{1}{(2\pi)^{d}} \int_{\mathbb{R}^{d}} \int_{\mathbb{R}^{d}} s_{u_{1}}\ldots s_{u_{i}} \cdot \sinfunc(s^{\top} (x - t))  p_{0}(t) dt d s, \quad \text{if} \ i = 4 l + 3  \nonumber \\
    \parenth{ \nabla^{i} p_{0}(x)}_{u_{1}u_{2}\ldots u_{i}} =  \frac{1}{(2\pi)^{d}} \int_{\mathbb{R}^{d}} \int_{\mathbb{R}^{d}} s_{u_{1}}\ldots s_{u_{i}} \cdot \cosfunc(s^{\top} (x - t))  p_{0}(t) dt d s, \quad \text{if} \ i = 4 l + 4  \nonumber
\end{align*}
for any $1 \leq u_{1}, \ldots, u_{i} \leq d$.
%\begin{align}
%   \nabla p_{0}(x) & = - \frac{1}{(2\pi)^{d}} \int_{\mathbb{R}^{d}} \int_{\mathbb{R}^{d}} s \cdot \sinfunc(s^{\top} (x - t))  p_{0}(t) dt d s, \nonumber \\
%    \nabla^2 p_{0}(x) & = - \frac{1}{(2\pi)^{d}} \int_{\mathbb{R}^{d}} \int_{\mathbb{R}^{d}} s s^{\top} \cdot \cosfunc(s^{\top} (x - t))  p_{0}(t) dt d s. \label{eq:Fourier_inverse_high_order}
%\end{align}
Based on the above equations, when $i = 4l + 1$ for any $1 \leq u_{1}, \ldots, u_{i} \leq d$ we find that
%\begin{align}
%    \abss{\parenth{ \Exs \brackets{\nabla \funcest_{n, \radius}(x)}}_{i} - \parenth{ \nabla p_{{0}}(x)}_{i}} & = \left| \frac{1}{(2\pi)^{d}} \int_{\mathbb{R}^{d} \backslash [-\radius, \radius]^{d}} \int_{\mathbb{R}^{d}} s_{i} \cdot \sinfunc(s^{\top} (x - t))  p_{0}(t) dt d s \right| \nonumber \\
%    & = \left| \frac{1}{(2\pi)^{d}} \int_{\mathbb{R}^{d} \backslash [-\radius, \radius]^{d}} s_{i} \cdot \parenth{ \sinfunc(s^{\top} x) \text{Re}(\widehat{p}_{0}(s)) - \cosfunc(s^{\top} x) \text{Im}(\widehat{p}_{0}(s)) } d s \right| \nonumber \\
%    & \leq \frac{\sqrt{2}}{(2\pi)^{d}} \int_{\mathbb{R}^{d} \backslash [-\radius, \radius]^{d}} \abss{s_{i}} \abss{\widehat{p}_{0}(s)} ds \nonumber \\
%    & \leq \frac{\sqrt{2}}{(2\pi)^{d}} \sum_{j = 1}^{d} \int_{A_{j}} |s_{i}| |\widehat{p}_{0}(s)| ds, \label{eq:bound_first_order_derivative}
%\end{align} Similarly, for any $u, v \in [d]$ we obtain that
\begin{align}
    & \hspace{-3 em} \abss{\parenth{ \Exs \brackets{\nabla^i \funcest_{n, \radius}(x)}}_{u_{1}\ldots u_{i}} - \parenth{ \nabla^i p_{{0}}(x)}_{u_{1}\ldots u_{i}}} \nonumber \\
    & = \left| \frac{1}{(2\pi)^{d}} \int_{\mathbb{R}^{d} \backslash [-\radius, \radius]^{d}} \int_{\mathbb{R}^{d}} s_{u_{1}}\ldots s_{u_{i}} \cdot \sinfunc(s^{\top} (x - t))  p_{0}(t) dt d s \right| \nonumber \\
    & = \left| \frac{1}{(2\pi)^{d}} \int_{\mathbb{R}^{d} \backslash [-\radius, \radius]^{d}} s_{u_{1}}\ldots s_{u_{i}} \cdot \parenth{ \sinfunc(s^{\top} x) \text{Re}(\widehat{p}_{0}(s)) - \cosfunc(s^{\top} x) \text{Im}(\widehat{p}_{0}(s)) } d s \right| \nonumber \\
    & \leq \frac{1}{(2\pi)^{d}} \int_{\mathbb{R}^{d} \backslash [-\radius, \radius]^{d}} \abss{s_{u_{1}} \ldots s_{u_{i}}} \abss{\widehat{p}_{0}(s)} ds \nonumber \\
    & \leq \frac{\sqrt{2}}{(2\pi)^{d}} \sum_{j = 1}^{d} \int_{A_{j}} |s_{u_{1}}\ldots s_{u_{i}}| |\widehat{p}_{0}(s)| ds, \label{eq:bound_higher_order_derivative}
\end{align}
where $A_{j} = \{x \in \mathbb{R}^{d}: |x_{j}| \geq R \}$ for all $j \in [d]$. With similar argument, we can check that the bound~\eqref{eq:bound_higher_order_derivative} also holds for other settings of $i$, i.e., when $i \in \{4l+2, 4l+3, 4l+4\}$. Therefore, the bound~\eqref{eq:bound_higher_order_derivative} holds for all $i \leq r$. Now,
given the bound in equation~\eqref{eq:bound_higher_order_derivative}, we are ready to upper bound the mean-squared bias and variance of the higher order derivatives of $\funcest_{n, R}$.

\noindent
(a) Since $p_{0}$ is upper--supersmooth density function of order $\alpha > 0$, for any $1 \leq u_{1}, \ldots, u_{i} \leq d$, we obtain the following bounds:
\begin{align*}
    \int_{A_{j}} |s_{u_{1}} \ldots s_{u_{i}}| |\widehat{p}_{0}(s)| ds \leq C \int_{A_{j}} |s_{u_{1}} \ldots s_{u_{i}}| \exp \parenth{ - C_{1} \parenth{ \sum_{j = 1}^{d} |s_{j}|^{\alpha}} } ds.
\end{align*}
where $C$ and $C_{1}$ are universal constants from the Definition~\ref{def:tail_Fourier} with upper--supersmooth density. For any given $1 \leq u_{1}, \ldots, u_{i} \leq d$, we denote $B_{l} = \{v: u_{v} = l\}$ for any $l \in [d]$. Then, we have
\begin{align}
    \int_{A_{j}} |s_{u_{1}} \ldots s_{u_{i}}| \exp \parenth{ - C_{1} \parenth{ \sum_{j = 1}^{d} |s_{j}|^{\alpha}} } ds = \int_{A_{j}} \prod_{l = 1}^{d} |s_{l}|^{|B_{l}|} \exp \parenth{ - C_{1} \parenth{ \sum_{l = 1}^{d} |s_{j}|^{\alpha}} } ds. \label{eq:bound_higher_order_derivative_1}
\end{align}
We now bound $\int_{A_{j}} |s_{j}|^{|B_{j}|} \exp \parenth{ - C_{1} |s_{j}|^{\alpha} } d s_{j}$. When $\alpha > |B_{j}| + 1$, $|s_{j}|^{|B_{j}|} \leq |s_{j}|^{\alpha - 1}$ for all $|s_{j}| \geq R \geq 1$. Therefore, we obtain that following bound:
\begin{align*}
    \int_{A_{j}} |s_{j}|^{|B_{j}|} \exp \parenth{ - C_{1} |s_{j}|^{\alpha} } d s_{j} \leq  \int_{|s_{j}| \geq \radius} |s_{j}|^{\alpha - 1} \exp \parenth{ - C_{1} |s_{j}|^{\alpha} } d s_{j} = \frac{2 \exp(-C_{1} \radius^{\alpha})}{C_{1} \alpha}.
\end{align*}
When $\alpha \in (0, |B_{j}| + 1]$, we find that
\begin{align*}
    \int_{|s_{j}| \geq \radius} |s_{j}|^{|B_{j}|} \exp(-C_{1} |s_{j}|^{\alpha})ds_{j} & = 2 \int_{s_{j} \geq R} s_{j}^{|B_{j}| + 1 - \alpha} s_{j}^{\alpha - 1} \exp(-C_{1} s_{j}^{\alpha})ds_{j} \\
     = \frac{2 \radius^{|B_{j}| + 1 - \alpha} \exp \parenth{- C_{1} \radius^{\alpha}}}{C_{1} \alpha} +  &\frac{2(|B_{j}| + 1 - \alpha)}{C_{1}\alpha} \int_{t \geq \radius} s_{j}^{|B_{j}| - \alpha} \exp(-C_{1}s_{j}^{\alpha})d s_{j} \\
     \leq \frac{2 \radius^{|B_{j}| + 1 - \alpha} \exp \parenth{- C_{1} \radius^{\alpha}}}{C_{1} \alpha} +  &\frac{2(|B_{j}| + 1 - \alpha)}{C_{1}\alpha \radius^{\alpha}} \int_{s_{j} \geq \radius} s_{j}^{|B_{j}|} \exp(-C_{1}s_{j}^{\alpha})d s_{j},
\end{align*}
where the equality in the above display is due to the integration by part. By choosing $\radius$ such that $\radius^{\alpha} \geq \frac{2(|B_{j}| + 1 - \alpha)}{C_{1} \alpha}$, the above inequality leads to
\begin{align*}
    \int_{|s_{j}| \geq \radius} |s_{j}|^{|B_{j}|} \exp(-C_{1} |s_{j}|^{\alpha})d s_{j} \leq \frac{4 \radius^{|B_{j}| + 1 - \alpha} \exp \parenth{- C_{1} \radius^{\alpha}}}{C_{1} \alpha}. %\label{eq:first_order_supersmooth01}
\end{align*}
Collecting the above results, we obtain
\begin{align}
    \int_{A_{j}} |s_{j}|^{|B_{j}|} \exp \parenth{ - C_{1} |s_{j}|^{\alpha} } d s_{j} & \leq \frac{4}{C_{1}\alpha} \cdot \radius^{\max \{|B_{j}| + 1 - \alpha, 0\}} \exp \parenth{- C_{1} \radius^{\alpha}} \nonumber \\
    & \leq \frac{4}{C_{1}\alpha} \cdot \radius^{\max \{i + 1 - \alpha, 0\}} \exp \parenth{- C_{1} \radius^{\alpha}}, \label{eq:bound_higher_order_derivative_2}
\end{align}
where the second inequality is due to the fact that $|B_{j}| \leq i$. For any $\alpha > 0$ and $l \in \mathbb{N}$, we denote $I(\alpha, l) = \int_{\mathbb{R}} |t|^{l} \exp ( -C_{1} |t|^{\alpha}) dt$. It is clear that $I(\alpha, l) < \infty$. Plugging the result in equation~\eqref{eq:bound_higher_order_derivative_2} into the equation~\eqref{eq:bound_higher_order_derivative_1}, we find that
\begin{align*}
    \int_{A_{j}} |s_{u_{1}} \ldots s_{u_{i}}| \exp \parenth{ - C_{1} \parenth{ \sum_{j = 1}^{d} |s_{j}|^{\alpha}} } ds \leq \frac{4}{C_{1}\alpha} \parenth{\prod_{l \neq j} I(\alpha, |B_{l}|)} \cdot \radius^{\max \{i + 1 - \alpha, 0\}} \exp \parenth{- C_{1} \radius^{\alpha}},
\end{align*}
for $j \in [d]$ and $1 \leq u_{1}, \ldots, u_{i} \leq d$. Combining that bound and the bound in equation~\eqref{eq:bound_higher_order_derivative}, we arrive at the following inequality:
\begin{align*}
    \abss{\parenth{ \Exs \brackets{\nabla^i \funcest_{n, \radius}(x)}}_{u_{1}\ldots u_{i}} - \parenth{ \nabla^2 p_{{0}}(x)}_{u_{1}\ldots u_{i}}} \leq \frac{\sqrt{2} d}{2^{d - 2} \pi^{d} C_{1}\alpha} \parenth{\prod_{l \neq j} I(\alpha, |B_{l}|)} \cdot \radius^{\max \{i + 1 - \alpha, 0\}} e^{ \parenth{- C_{1} \radius^{\alpha}}}.
\end{align*}
Hence, we obtain that
\begin{align*}
    & \hspace{-4 em} \| \Exs \brackets{\nabla^{i} \funcest_{n, \radius}(x)} - \nabla^{i} p_{{0}}(x)\|_{\max} \\
    & \leq \frac{\sqrt{2} d}{2^{d - 2} \pi^{d} C_{1}\alpha} \max_{\substack{|B_{1}|, \ldots, |B_{d}|; \\ |B_{1}| + \ldots + |B_{d}| = i}} \parenth{\prod_{l \neq j} I(\alpha, |B_{l}|)} \cdot \radius^{\max \{i + 1 - \alpha, 0\}} e^{\parenth{- C_{1} \radius^{\alpha}}}.
\end{align*}
As a consequence, we obtain a conclusion with the upper bound of bias of $\nabla^{i} \funcest_{n, \radius}(x)$. 

Moving to the variance of $\nabla^{i} \funcest_{n, \radius}(x)$, direct algebra lead to
\begin{align*}
    & \hspace{- 3 em} \Exs \brackets{ \enorm{ \nabla^{i}\funcest_{n, \radius}(x) - \Exs \brackets{\nabla^{i}\funcest_{n, \radius}(x)}}^2} \\
    & = \sum_{1 \leq u_{1}, \ldots, u_{i} \leq d} \Exs \brackets{\parenth{ \parenth{\nabla^i \funcest_{n, \radius}(x)}_{u_{1}\ldots u_{i}} - \parenth{ \Exs \brackets{\nabla^i \funcest_{n, \radius}(x)}}_{u_{1}\ldots u_{i}}}^2} \\
    & \leq \sum_{1 \leq u_{1}, \ldots, u_{i} \leq d} \frac{1}{(2\pi)^{d} n} \Exs \brackets{ \parenth{\int_{[-\radius, \radius]^{d}} \parenth{\nabla_{x}^{i} \cosfunc(s^{\top} (x - X))}_{u_{1}\ldots u_{i}} }^2},
\end{align*}
where the outer expectation is taken with respect to $X \sim p_{0}$. To simplify the presentation, we denote $h(y, s) = \frac{\sin(\radius(y - s))}{y - s}$ for all $y, s \in \mathbb{R}$. Recall that, $B_{l} = \{v: u_{v} = l\}$ for any $l \in [d]$ and for any given $1 \leq u_{1}, \ldots, u_{i} \leq d$. Then, we can check that
\begin{align*}
    \Exs \brackets{ \parenth{\int_{[-\radius, \radius]^{d}} \parenth{\nabla_{x}^{i} \cosfunc(s^{\top} (x - X))}_{u_{1}\ldots u_{i}} }^2} & = \Exs \brackets{\parenth{ \prod_{j = 1}^{d} \frac{\partial^{|B_{j}|}}{\partial{x_{j}^{|B_{j}|}}} h(x_{j}, X_{.j})}^2} \\
    & \leq \|p_{0}\|_{\infty} \prod_{j = 1}^{d} \int_{\mathbb{R}} \parenth{\frac{\partial^{|B_{j}|}}{\partial{x_{j}^{|B_{j}|}}} h(x_{j}, t)}^2 d t,
\end{align*}
where we denote $X = (X_{.1}, \ldots, X_{.d})$.  Direct calculation shows that $\int_{\mathbb{R}} \parenth{\frac{\partial^{l}}{\partial{y^{l}}} h(y, t)}^2 d t = c_{l} \radius^{2 l + 1}$ for any $l \geq 0$ and $y \in \mathbb{R}$ where $c_{l}$ are some universal constants. Collecting these results, we obtain
\begin{align*}
    \Exs \brackets{ \parenth{\int_{[-\radius, \radius]^{d}} \parenth{\nabla_{x}^{i} \cosfunc(s^{\top} (x - X))}_{u_{1}\ldots u_{i}} }^2} & \leq \|p_{0}\|_{\infty} \prod_{j = 1}^{d} c_{|B_{j}|} \radius^{2 |B_{j}|  + 1} \\
    & = \parenth{ \|p_{0}\|_{\infty} \prod_{j = 1}^{d} c_{|B_{j}|}} \radius^{2 i + d},
\end{align*}
where the final equality is due to $\sum_{j = 1}^{d} |B_{j}| = i$. Putting all the results together, we finally have
\begin{align*}
    \Exs \brackets{ \enorm{ \nabla^{i}\funcest_{n, \radius}(x) - \Exs \brackets{\nabla^{i}\funcest_{n, \radius}(x)}}^2} \leq \bar{C}_{i} \frac{R^{2i+d}}{n},
\end{align*}
where $\bar{C}_{i}$ is some universal constant and $\|p_{0}\|_{\infty}$. As a consequence, we reach the conclusion of part (a) of the theorem.

\noindent
(b) The analysis of variance in the ordinary smooth setting is similar to that of variance in the supersmooth setting in part (a); therefore, it is omitted. Our proof with part (b) will only focus on bounding the bias. In particular, since $p_{0}$ is ordinary smooth density function of order $\beta$, for any $1 \leq u_{1}, \ldots, u_{i} \leq d$ we obtain that
\begin{align*}
    \int_{A_{j}} |s_{u_{1}} \ldots s_{u_{i}}| |\widehat{p}_{0}(s)| ds & \leq c \int_{A_{j}} |s_{u_{1}} \ldots s_{u_{i}}| \prod_{l = 1}^{d} \frac{1}{1 + |s_{l}|^{\beta}}  ds = c \int_{A_{j}} \prod_{l = 1}^{d} \frac{|s_{l}|^{|B_{l}|}}{1 + |s_{l}|^{\beta}}  ds \\
    & \leq c \parenth{ \prod_{l \neq j} \parenth{ \int_{\mathbb{R}} \frac{|s_{l}|^{|B_{l}|}}{1 + |s_{l}|^{\beta}} d s_{l}}} \cdot \int_{|s_{j}| \geq \radius} \frac{|s_{j}|^{|B_{j}|}}{1 + |s_{j}|^{\beta}} d s_{j}.
\end{align*}
Here, $c$ in the above bounds is the universal constant associated with the ordinary smooth density function $p_{0}$ from Definition~\ref{def:tail_Fourier}. Since $|B_{l}| \leq r < \beta - 1$, we have $\int_{\mathbb{R}} \frac{|s_{l}|^{|B_{l}|}}{1 + |s_{l}|^{\beta}} d s_{l} < \infty$ for all $l \in \{1, \ldots, d\}$. Furthermore, we find that
\begin{align*}
    \int_{|s_{j}| \geq \radius} \frac{|s_{j}|^{|B_{j}|}}{1 + |s_{j}|^{\beta}} d s_{j} \leq 2 \int_{s_{j} \geq \radius} \frac{1}{s_{j}^{\beta - |B_{j}|}} d s_{j} = \frac{2 \radius^{ - \beta + |B_{j}| + 1}}{\beta - |B_{j}| - 1} \leq \frac{2 \radius^{- \beta + i + 1}}{\beta - |B_{j}| - 1},
\end{align*}
where the final inequality is due to $|B_{j}| \leq i$. Collecting the above results, we arrive at the following bound:
\begin{align}
    \int_{A_{j}} |s_{u_{1}} \ldots s_{u_{i}}| |\widehat{p}_{0}(s)| ds \leq \frac{2 c}{\beta - |B_{j}| - 1} \parenth{ \prod_{l \neq j} \parenth{ \int_{\mathbb{R}} \frac{|s_{l}|^{|B_{l}|}}{1 + |s_{l}|^{\beta}} d s_{l}}} \cdot \radius^{- \beta + i + 1}. \label{eq:bound_higher_order_derivative_ordinary}
\end{align}
Plugging the result from equation~\eqref{eq:bound_higher_order_derivative_ordinary} into the bound in equation~\eqref{eq:bound_higher_order_derivative}, we obtain the conclusion with the upper bound of bias in part (b). 
%%%%%%%%%%%%%%%%%%%%%%%%%%%%%%%%%%%%%%%%%%%%%%%%%%%%%%%%%%%%%%%%%%%%%%%%%%%%%%%%%%%%%%%%%%%%%%%%%%%%%%%%%%%%%%%%%%%%%%%
\subsection{Proof of Theorem~\ref{theorem:uniform_bound_derivatives}}
\label{subsec:proof:theorem:uniform_bound_derivatives}
By triangle inequality, we find that
\begin{align*}
    \sup_{x \in \mathcal{X}} \left\|\nabla^{i} \funcest_{n, \radius}(x) - \nabla^{i} p_{0}(x)\right\|_{\max} \leq \sup_{x \in \mathcal{X}} \left\|\nabla^{i} \funcest_{n, \radius}(x) - \Exs \brackets{ \nabla^{i} \funcest_{n, \radius}(x)} \right\|_{\max} & \\
    & \hspace{-7 em} + \sup_{x \in \mathcal{X}} \left\|\Exs \brackets{ \nabla^{i} \funcest_{n, \radius}(x)} - \nabla^{i} p_{0}(x)\right\|_{\max}.
\end{align*}
We first establish the uniform concentration bound for 
$$\sup_{x \in \mathcal{X}} \abss{\parenth{ \nabla^{i} \funcest_{n, \radius}(x)}_{u_{1}\ldots u_{i}} - \parenth{ \Exs \brackets{ \nabla^{i} \funcest_{n, \radius}(x)}}_{u_{1}\ldots u_{i}} }$$ for any $1 \leq u_{1}, \ldots, u_{i} \leq d$. To simplify the notation, we denote $h(y, s) = \frac{\sin(\radius(y - s))}{y - s}$ for all $y, s \in \mathbb{R}$. Then, we can rewrite $\parenth{ \nabla^{i} \funcest_{n, \radius}(x)}_{u_{1}\ldots u_{i}}$ as follows:
\begin{align*}
    \parenth{ \nabla^{i} \funcest_{n, \radius}(x)}_{u_{1}\ldots u_{i}} = \frac{1}{n (2\pi)^{d}} \sum_{j = 1}^{n} \prod_{l = 1}^{d} \frac{\partial^{|B_{l}|}}{\partial{x_{l}^{|B_{l}|}}} h(x_{l}, X_{jl}),
\end{align*}
where $B_{l} = \{v: u_{v} = l\}$ for any $l \in [d]$ and for any given $1 \leq u_{1}, \ldots, u_{i} \leq d$. We denote $Y_{j} = \frac{1}{\pi^{d}} \prod_{l = 1}^{d} \frac{\partial^{|B_{l}|}}{\partial{x_{l}^{|B_{l}|}}} h(x_{l}, X_{jl})$ for all $j \in [n]$. Then, since $\abss{\frac{\partial^{l}}{\partial^{l}} h(y,s)} \leq \radius^{l + 1}$ for all $l \geq 0$, we have $|Y_{j}| \leq \radius^{\sum_{l = 1}^{d} |B_{l}| + d} = \radius^{i + d}$ for all $j \in [n]$. Furthermore, we have $\Exs(\abss{Y_{j}}) \leq \radius^{i}$ for all $j \in [n]$. Given these results, an application of Bernstein's inequality leads to
\begin{align*}
    \Prob \parenth{ \sup_{x \in \mathcal{X}} \abss{\parenth{ \nabla^{i} \funcest_{n, \radius}(x)}_{u_{1}\ldots u_{i}} - \parenth{ \Exs \brackets{ \nabla^{i} \funcest_{n, \radius}(x)}}_{u_{1}\ldots u_{i}} } > t} & \\
    & \hspace{- 5 em} \leq 4 \mathcal{N}_{[]} \parenth{ t/8, \mathcal{F}', \mathbb{L}_{1}(P)} \exp \parenth{ - \frac{96 n t^2}{76 \radius^{2i + d}}},
\end{align*}
where $\mathcal{F}' = \{f_{x}: \mathbb{R}^{d} \to \mathbb{R}: f_{x}(t) = \prod_{l = 1}^{d} \frac{\partial^{|B_{l}|}}{\partial{x_{l}^{|B_{l}|}}} h(x_{l}, t_{l}) \ \text{for all} \ x \in \mathbb{X}, t \in \mathbb{R}^{d} \}$. Direct algebra shows that for any $x_{1}, x_{2} \in \mathcal{X}$, $|f_{x_{1}}(t) - f_{x_{2}}(t)| \leq d \radius^{i + d + 1} \enorm{x_{1} - x_{2}}$ for all $t \in \mathbb{R}^{d}$. As $\mathcal{X}$ is a bounded subset of $\mathbb{R}^{d}$, combining the above results leads to
\begin{align*}
    \Prob \parenth{ \sup_{x \in \mathcal{X}} \abss{\parenth{ \nabla^{i} \funcest_{n, \radius}(x)}_{u_{1}\ldots u_{i}} - \parenth{ \Exs \brackets{ \nabla^{i} \funcest_{n, \radius}(x)}}_{u_{1}\ldots u_{i}} } > t} & \\ 
    & \hspace{- 7 em} \leq \parenth{ \frac{ 4 d \sqrt{d} \cdot \text{Diam}(\mathcal{X}) \radius^{d + i + 1}}{t}}^{d} \exp \parenth{ - \frac{96 n t^2}{76 \radius^{2i + d}}}.
\end{align*}
Given the above result, an application of union bound shows that
\begin{align*}
    & \hspace{-4 em} \Prob \biggr(\sup_{x \in \mathcal{X}} \left\|\nabla^{i} \funcest_{n, \radius}(x) - \nabla^{i} p_{0}(x)\right\|_{\max} > t) \\
    & \leq \sum_{1 \leq u_{1}, \ldots, u_{i} \leq d} \Prob \parenth{ \sup_{x \in \mathcal{X}} \abss{\parenth{ \nabla^{i} \funcest_{n, \radius}(x)}_{u_{1}\ldots u_{i}} - \parenth{ \Exs \brackets{ \nabla^{i} \funcest_{n, \radius}(x)}}_{u_{1}\ldots u_{i}} } > t} \\
    & \leq \frac{4^{d} d^{3d/2 + i} (\text{Diam}(\mathcal{X}))^{d} \radius^{d(d + i + 1)}}{t^{d}} \exp \parenth{ - \frac{96 n t^2}{76 \radius^{2i + d}}},
\end{align*}
From the above concentration bound, by choosing 
$$t = \bar{C} \sqrt{\frac{\radius^{d + 2i} \parenth{\log(2/ \delta) + d (d + i + 1) \log \radius + d (\log d + \text{Diam}(\mathcal{X})}}{n}}$$ where $\bar{C}$ is some universal constant, we obtain $\Prob \biggr(\sup_{x \in \mathcal{X}} \left\|\nabla^{i} \funcest_{n, \radius}(x) - \nabla^{i} p_{0}(x)\right\|_{\max} > t) \leq \delta$. Combining this result with the upper bounds of $\sup_{x \in \mathcal{X}} \left\|\Exs \brackets{ \nabla^{i} \funcest_{n, \radius}(x)} - \nabla^{i} p_{0}(x)\right\|_{\max}$ from Theorem~\ref{theorem:bias_variance_Fourier_density_derivatives}, we reach the conclusion of the theorem.
%%%%%%%%%%%%%%%%%%%%%%%%%%%%%%%%%%%%%%%%%%%%%%%%%%%%%%%%%%%%%%%%%%%%%%%%%%%%%%%%%%%%%%%%%%%%%%%%%%%%%%%%%%%%%%%%%%%%%%%%%%%%%%%%%%
\subsection{Proof of Theorem~\ref{theorem:deconvolution_bias_variance_supersmooth_derivatives}}
\label{subsec:proof:theorem:deconvolution_bias_variance_supersmooth_derivatives}
Since $g \in \mathcal{C}^{r}(\Theta)$, we have $p_{0} \in \mathcal{C}^{r}(\mathcal{X})$. From the Fourier inverse theorem, we have
\begin{align*}
    \frac{\partial^{\gamma} g}{\partial{\theta^{\gamma}}}(\theta) = \frac{1}{(2 \pi)^{d}} \int_{\mathbb{R}^{d}} \widehat{\partial^{\gamma} g}(s) \exp \parenth{i \theta^{\top} s} ds, 
\end{align*}
for any $\gamma = (\gamma_{1}, \ldots, \gamma_{d}) \in \mathbb{N}^{d}$ such that $|\gamma| \leq r$. Since $\widehat{\partial^{\gamma} g}(s) = (i s)^{\gamma} \widehat{g}(s)$, the above identity becomes
\begin{align*}
    \frac{\partial^{\gamma} g}{\partial{\theta^{\gamma}}} (\theta) = \frac{1}{(2 \pi)^{d}} \int_{\mathbb{R}^{d}} (i s)^{\gamma} \widehat{g}(s) \exp \parenth{i \theta^{\top} s} ds & = \frac{1}{(2 \pi)^{d}} \int_{\mathbb{R}^{d}} (i s)^{\gamma} \frac{\widehat{p_{0}}(s)}{\widehat{f}(s)} \exp \parenth{i \theta^{\top} s} ds \\
    & = \frac{1}{(2 \pi)^{d}} \int_{\mathbb{R}^{d}} \frac{\widehat{\partial^{\gamma} p_{0}}(s)}{\widehat{f}(s)} \exp \parenth{i \theta^{\top} s} ds \\
    & = \frac{1}{(2 \pi)^{d}} \int_{\mathbb{R}^{d}} \int_{\mathbb{R}^{d}} \frac{\partial^{\gamma} p_{0}}{\partial{t^{\gamma}}}(t) \cdot \frac{\cosfunc(s^{\top}(\theta - t))}{\widehat{f}(s)} d t ds,
\end{align*}
where the final inequality is because $\widehat{f}$ is an even function. An application of integration by parts leads to 
\begin{align*}
    \int_{\mathbb{R}^{d}} \frac{\partial^{\gamma} p_{0}}{\partial{t^{\gamma}}}(t) \cdot \frac{\cosfunc(s^{\top}(\theta - t))}{\widehat{f}(s)} d t = \int_{\mathbb{R}^{d}} p_{0}(t) \frac{\frac{\partial^{\gamma}}{\partial{\theta^{\gamma}}} \cosfunc(s^{\top}(\theta - t))}{\widehat{f}(s)} d t.
\end{align*}
Therefore, for any $i \in \{1, \ldots, r\}$ we have
\begin{align*}
    \nabla^{i} g(\theta) = \frac{1}{(2 \pi)^{d}} \int_{\mathbb{R}^{d}} \int_{\mathbb{R}^{d}} p_{0}(t) \frac{\nabla^{i}_{\theta} \cosfunc(s^{\top}(\theta - t))}{\widehat{f}(s)} d t ds.
\end{align*}
For any $1 \leq u_{1}, \ldots, u_{i} \leq d$ and $i = 4l + 1$ for some $l \geq 0$, simple algebra leads to
\begin{align*}
    (\nabla^{i} g(\theta))_{u_{1}\ldots u_{i}} = - \frac{1}{(2\pi)^{d}} \int_{\mathbb{R}^{d}} \int_{\mathbb{R}^{d}} s_{u_{1}}\ldots s_{u_{i}} \cdot \frac{\sinfunc(s^{\top} (x - t))  p_{0}(t)}{\widehat{f}(s)} dt d s.
\end{align*}
Therefore, we obtain that
\begin{align}
    & \hspace{- 7 em} \abss{\parenth{ \Exs \brackets{ \nabla^{i} \deconvest_{n,\radius}(\theta)}}_{u_{1}\ldots u_{i}} - (\nabla^{i} g(\theta))_{u_{1}\ldots u_{i}}} \nonumber \\
    & = \abss{\frac{1}{(2 \pi)^{d}} \int_{\mathbb{R}^{d} \backslash [-\radius, \radius]^{d}} \int_{\mathbb{R}^{d}} s_{u_{1}}\ldots s_{u_{i}} \cdot \frac{\sinfunc(s^{\top} (x - t))  p_{0}(t)}{\widehat{f}(s)} dt d s} \nonumber \\
    & \leq \frac{1}{(2 \pi)^{d}} \int_{\mathbb{R}^{d} \backslash [-\radius, \radius]^{d}} \abss{ s_{u_{1}} \ldots s_{u_{i}}}  \abss{\widehat{g}(s)} d s \nonumber \\
    & \leq \frac{\sqrt{2}}{(2\pi)^{d}} \sum_{j = 1}^{d} \int_{A_{j}} |s_{u_{1}}\ldots s_{u_{i}}| \abss{\widehat{g}(s)} ds, \label{eq:bound_higher_order_derivative_convolution}
\end{align}
where $A_{j} = \{x \in \mathbb{R}^{d}: |x_{j}| \geq R \}$ for all $j \in [d]$. We can check the inequality~\eqref{eq:bound_higher_order_derivative_convolution} also holds for $i \in \{4l + 2, 4l+3, 4l + 4\}$. Therefore, this inequality holds for all $i \leq r$. From here, based on the proof of Theorem~\ref{theorem:bias_variance_Fourier_density_derivatives} with upper-supersmooth density function, for each $j \in [d]$, when $\radius \geq C'$ where $C'$ is some universal constant we have
\begin{align*}
    \sum_{j = 1}^{d} \int_{A_{j}} |s_{u_{1}}\ldots s_{u_{i}}| \abss{\widehat{g}(s)} ds \leq C \radius^{\max \{i + 1 - \alpha_{2}, 0\}} \exp \parenth{- C_{1} \radius^{\alpha_{2}}},
\end{align*}
where $C$ is some universal constant and $C_{1}$ is the given constant in Definition~\ref{def:tail_Fourier}. Plugging the above bound into the equation~\eqref{eq:bound_higher_order_derivative_convolution}, we obtain
\begin{align*}
    \abss{\parenth{ \Exs \brackets{ \nabla^{i} \deconvest_{n,\radius}(\theta)}}_{u_{1}\ldots u_{i}} - (\nabla^{i} g(\theta))_{u_{1}\ldots u_{i}}} \leq \frac{\sqrt{2} C d}{(2\pi)^{d}} \radius^{\max \{i + 1 - \alpha_{2}, 0\}} \exp \parenth{- C_{1} \radius^{\alpha_{2}}}.
\end{align*}
Therefore, we have
\begin{align*}
    \| \Exs \brackets{\nabla^{i} \deconvest_{n, \radius}(\theta)} - \nabla^{i} g(\theta)\|_{\max} \leq \bar{C} \radius^{\max \{i + 1 - \alpha_{2}, 0\}} \exp \parenth{- C_{1} \radius^{\alpha_{2}}},
\end{align*}
where $\bar{C}$ is some universal constant depending on $d$. As a consequence, we obtain the conclusion of the theorem with the bias of the Fourier deconvolution estimator $\nabla^{i} \deconvest_{n, \radius}$. 

Moving to the variance of $\nabla^{i} \deconvest_{n, \radius}$, for each $1 \leq u_{1}, \ldots, u_{i} \leq d$ we have
$$
    \Exs \brackets{ \parenth{ \parenth{\Exs \brackets{ \nabla^{i} \deconvest_{n,\radius}(\theta)}}_{u_{1}\ldots u_{i}} - (\nabla^{i} \deconvest_{n,\radius}(\theta))_{u_{1}\ldots u_{i}}}^2}
$$
upper bounded by
\begin{align*}
& \frac{1}{(2 \pi)^{d} n} \Exs \brackets{\parenth{\int_{[-\radius, \radius]^{d}} \frac{\parenth{\nabla^{i}_{\theta} \cosfunc(s^{\top}(\theta - X))}_{u_{1}\ldots u_{i}}}{\widehat{f}(s)} d s  }^2} \\
    & \leq \frac{\|g\|_{\infty} \radius^{2(i + d)}}{(2 \pi)^{d} n \min_{s \in [-\radius,\radius]^{d}} \widehat{f}^2(s)}  \leq \frac{\|g\|_{\infty} \radius^{2(i + d)} \exp(2C_{2}d\radius^{\alpha_{1}})}{(2 \pi)^{d} n},
\end{align*}
where $C_{2}$ is a given constant in Definition~\ref{def:tail_Fourier}. Hence, we obtain that
\begin{align*}
    \Exs \brackets{ \enorm{ \nabla^{i} \deconvest_{n,\radius}(\theta) - \Exs \brackets{ \nabla^{i} \deconvest_{n,\radius}(\theta)}}^2} & = \sum_{1 \leq u_{1}, \ldots, u_{i} \leq d} \Exs \brackets{ \parenth{ \parenth{\Exs \brackets{ \nabla^{i} \deconvest_{n,\radius}(\theta)}}_{u_{1}\ldots u_{i}} - (\nabla^{i} \deconvest_{n,\radius}(\theta))_{u_{1}\ldots u_{i}}}^2} \\
    & \leq C' \radius^{2(i + d)} \exp(2C_{2}d\radius^{\alpha_{1}}),
\end{align*}
where $C'$ is some universal constant depending on $\|g\|_{\infty}$ and dimension $d$. As a consequence, we obtain the conclusion of the theorem with the variance of the derivatives of $\deconvest_{n, \radius}$.
%%%%%%%%%%%%%%%%%%%%%%%%%%%%%%%%%%%%%%%%%%%%%%%%%%%%%%%%%%%%%%%%%%%%%%%%%%%%%%%%%%%%%%%
%%%%%%%%%%%%%%%%%%%%%%%%%%%%%%%%%%%%%
\subsection{Proof of Theorem~\ref{theorem:random_nonparametric}}
\label{subsec:proof:theorem:random_nonparametric}
In this proof, we first bound the bias of $\nonpreg(x)$. Then, we establish an upper bound the variance of $\nonpreg(x)$ for each $x \in \mathcal{X}$. 
\paragraph{Upper bound on the bias of $\nonpreg(x)$:} From the definition of $\nonpreg(x)$, simple algebra leads to
\begin{align}
    \nonpreg(x) - m(x) = \frac{\widehat{a}(x) - m(x) \funcest_{n, \radius}(x)}{p_{0}(x)} + \frac{(\nonpreg(x) - m(x)) (p_{0}(x) - \funcest_{n,\radius}(x))}{p_{0}(x)}. \label{eq:nonparametric_regression_equation}
\end{align}
Therefore, we obtain that
\begin{align}
    & \hspace{- 2 em} \parenth{ \Exs \brackets{\nonpreg(x)} - m(x)}^2 \nonumber \\
    & \leq 2 \frac{\parenth{\Exs \brackets{\widehat{a}(x) - m(x) \funcest_{n, \radius}(x)}}^2}{p_{0}^2(x)} + 2 \frac{\parenth{\Exs \brackets{(\nonpreg(x) - m(x)) (p_{0}(x) - \funcest_{n,\radius}(x))}}^2}{p_{0}^2(x)} \nonumber \\
    & \leq 2 \frac{\parenth{\Exs \brackets{\widehat{a}(x) - m(x) \funcest_{n, \radius}(x)}}^2}{p_{0}^2(x)} + 2 \frac{\Exs \brackets{(\nonpreg(x) - m(x))^2} \Exs \brackets{(p_{0}(x) - \funcest_{n,\radius}(x))^2}}{p_{0}^2(x)}, \label{eq:nonparametric_regression_equation_first}
\end{align}
where the first inequality is due to Cauchy-Schwarz inequality and the second inequality is due to the standard inequality $\Exs^2(XY) \leq \Exs(X^2)\Exs(Y^2)$. Since $p_{0}$ is upper-supersmooth density function of order $\alpha > 0$, from the result of Theorem~\ref{theorem:bias_variance_Fourier_density}, we have
\begin{align*}
    \Exs \brackets{(p_{0}(x) - \funcest_{n,\radius}(x))^2} \leq C^2 \radius^{\max \{2 - 2\alpha, 0\}} \exp \parenth{-2 C_{1} \radius^{\alpha} } + \frac{\|p_{0}\|_{\infty}}{\pi^{d}} \cdot \frac{\radius^{d}}{n},
\end{align*}
where $C_{1}$ is the given constant in Definition~\ref{def:tail_Fourier} and $C$ is some universal constant. 

Now, we proceed to bound $\abss{\Exs \brackets{\widehat{a}(x) - m(x) \funcest_{n, \radius}(x)}}$. Direct calculation shows that
\begin{align*}
    \Exs \brackets{\widehat{a}(x) - m(x) \funcest_{n, \radius}(x)} & = \frac{1}{(2\pi)^{d}} \biggr( \int_{\mathbb{R}^{d}} \int_{[-\radius, \radius]^{d}} \cos(s^{\top}(x - t)) m(t) p_{0}(t) d s d t \\
    & - \int_{\mathbb{R}^{d}} \int_{[-\radius, \radius]^{d}} \cos(s^{\top}(x - t)) m(x) p_{0}(t) ds dt \biggr).
\end{align*}
From the Fourier integral theorem, we obtain
\begin{align*}
m(x) p_{0}(x) - \frac{1}{(2\pi)^{d}} \int_{\mathbb{R}^{d}} \int_{[-\radius, \radius]^{d}} \cos(s^{\top}(x - t)) m(t) p_{0}(t) d s d t & \\
& \hspace{-8 em} = \frac{1}{(2\pi)^{d}} \int_{\mathbb{R}^{d}} \int_{\mathbb{R}^{d} \backslash [-\radius, \radius]^{d}} \cos(s^{\top}(x - t)) m(t) p_{0}(t) d s d t, \\
m(x) p_{0}(x) - \frac{1}{(2\pi)^{d}} \int_{\mathbb{R}^{d}} \int_{[-\radius, \radius]^{d}} \cos(s^{\top}(x - t)) m(x) p_{0}(t) dt \\
& \hspace{- 8 em} = \frac{1}{(2\pi)^{d}} \int_{\mathbb{R}^{d}} \int_{\mathbb{R}^{d} \backslash [-\radius, \radius]^{d}} \cos(s^{\top}(x - t)) m(x) p_{0}(t) ds dt.
\end{align*}
Collecting the above equations, we arrive at the following result:
\begin{align}
    \abss{ \Exs \brackets{\widehat{a}(x) - m(x) \funcest_{n, \radius}(x)}} & \leq \frac{1}{(2\pi)^{d}} \biggr(\abss{ \int_{\mathbb{R}^{d}} \int_{\mathbb{R}^{d} \backslash [-\radius, \radius]^{d}} \cos(s^{\top}(x - t)) m(t) p_{0}(t) d s d t} \nonumber \\
    & + \abss{\int_{\mathbb{R}^{d}} \int_{\mathbb{R}^{d} \backslash [-\radius, \radius]^{d}} \cos(s^{\top}(x - t)) m(x) p_{0}(t) ds dt} \biggr) \nonumber \\
    & \leq \frac{1}{(2\pi)^{d}} \int_{\mathbb{R}^{d} \backslash [-\radius, \radius]^{d}} ( \abss{\widehat{p_{0}}(s)} + \abss{\widehat{m \cdot p_{0}}(s)} ) ds. \label{eq:nonparametric_regression_equation_half_second}
\end{align}
Since $p_{0}$ is upper-smooth density function of order $\alpha$, from the proof of Theorem~\ref{theorem:bias_variance_Fourier_density}, we have
\begin{align}
    \int_{\mathbb{R}^{d} \backslash [-\radius, \radius]^{d}} \abss{\widehat{p_{0}}(s)} \leq \bar{C} \radius^{\max \{1 - \alpha, 0\}} \exp( -C_{1}\radius^{\alpha}), \label{eq:nonparametric_regression_equation_second}
\end{align}
where $\bar{C}$ is some universal constant depending on $d$. Furthermore, we find that
\begin{align*}
\int_{\mathbb{R}^{d} \backslash [-\radius, \radius]^{d}} \abss{\widehat{m \cdot p_{0}}(s)} ds & \leq \sum_{j = 1}^{d} \int_{A_{j}} \abss{\widehat{m \cdot p_{0}}(s)} ds \\
& \leq \sum_{j = 1}^{d} \int_{A_{j}} C \cdot Q(|s_{1}|,\ldots,|s_{d}|) \exp \parenth{ -C_{1} \parenth{ \sum_{i = 1}^{d} |s_{i}|^{\alpha}} } ds
,
\end{align*}
where $A_{j} = \{x: |x_{j}| \geq \radius\}$. For any $0 \leq \tau_{1}, \tau_{2}, \ldots, \tau_{d} \leq r$ where $r \geq 1$, we have
\begin{align*}
    \int_{A_{j}} \prod_{i = 1}^{d} |s_{i}|^{\tau_{i}} \exp \parenth{ -C_{1} \parenth{ \sum_{i = 1}^{d} |s_{i}|^{\alpha}} } ds = \parenth{ \prod_{i \neq j} I(\alpha, \tau_{i})} \int_{A_{j}} |s_{j}|^{\tau_{j}} \exp ( -C_{1} |s_{j}|^{\alpha}) ds_{j},
\end{align*}
where $I(\alpha, \tau) = \int_{\mathbb{R}} |t|^{\tau} \exp ( -C_{1} |t|^{\alpha}) dt$ for all $\tau \geq 0$. Based on the proof argument of equation~\ref{eq:bound_higher_order_derivative_2} in Theorem~\ref{theorem:bias_variance_Fourier_density_derivatives}, we obtain
\begin{align*}
    \int_{A_{j}} |s_{j}|^{\tau_{j}} \exp ( -C_{1} |s_{j}|^{\alpha}) ds_{j} \leq C'\radius^{\max \{\tau_{j} + 1 - \alpha, 0\}} \exp( - C_{1} \radius^{\alpha}),
\end{align*}
where $C'$ is some universal constant. Putting the above results together leads to the following bound: 
\begin{align}
    \int_{\mathbb{R}^{d} \backslash [-\radius, \radius]^{d}} \abss{\widehat{m \cdot p_{0}}(s)} ds \leq C' \radius^{\max \{\deg(Q) + 1 - \alpha, 0\}} \exp( - C_{1} \radius^{\alpha}). \label{eq:nonparametric_regression_equation_third}
\end{align}
Combining the results from equations~\eqref{eq:nonparametric_regression_equation_half_second},~\eqref{eq:nonparametric_regression_equation_second}, and~\eqref{eq:nonparametric_regression_equation_third}, we have 
\begin{align}
    \abss{ \Exs \brackets{\widehat{a}(x) - m(x) \funcest_{n, \radius}(x)}} \leq C'' \radius^{\max \{\deg(Q) + 1 - \alpha, 0\}} \exp( - C_{1} \radius^{\alpha}), \label{eq:bias_bound_nonparametric_regression}
\end{align}
where $C''$ is some universal constant. Plugging the result from equation~\eqref{eq:bias_bound_nonparametric_regression} into equation~\eqref{eq:nonparametric_regression_equation_first} leads to
\begin{align}
    \parenth{ \Exs \brackets{\nonpreg(x)} - m(x)}^2 & \leq \frac{C}{p_{0}^2(x)} \biggr(\radius^{\max \{2 \deg(Q) + 2 - 2 \alpha, 0\}} \exp( - 2 C_{1} \radius^{\alpha}) \nonumber \\ 
    & \hspace{- 3 em} + \Exs \brackets{(\nonpreg(x) - m(x))^2} \parenth{ \radius^{\max \{2 - 2\alpha, 0\}} \exp \parenth{-2 C_{1} \radius^{\alpha} } + \frac{\|p_{0}\|_{\infty}}{\pi^{d}} \cdot \frac{\radius^{d}}{n}}\biggr), \label{eq:bias_bound_nonparametric_regression}
\end{align}
where $C$ is some universal constant.
\paragraph{Upper bound on the variance of $\nonpreg(x)$:} Moving to the variance of $\widehat{m}(x)$, by taking variance both sides of the equation~\eqref{eq:nonparametric_regression_equation}, we find that
\begin{align}
    \var(\widehat{m}(x)) & = \var \parenth{ \frac{\widehat{a}(x) - m(x) \funcest_{n, \radius}(x)}{p_{0}(x)} + \frac{(\nonpreg(x) - m(x)) (p_{0}(x) - \funcest_{n,\radius}(x))}{p_{0}(x)}} \nonumber \\
    & \hspace{- 4 em} \leq \frac{2}{p_{0}^2(x)} \parenth{ \underbrace{\Exs \brackets{ \parenth{\widehat{a}(x) - m(x) \funcest_{n, \radius}(x)}^2}}_{T_{1}} + \underbrace{\Exs \brackets{(\nonpreg(x) - m(x))^2 (p_{0}(x) - \funcest_{n,\radius}(x))^2}}_{T_{2}}}. \label{eq:bound_var_nonparametric_regression}
\end{align}
First, we upper bound $T_{2}$. Denote $A$ the event such that $$\abss{\funcest_{n, \radius}(x) - p_{0}(x)} \leq C \parenth{ R^{\max \{1 - \alpha, 0\}} \exp \parenth{-C_{1} \radius^{\alpha} } + \sqrt{\frac{\radius^{d} \log(2/ \delta)}{n}}}$$ where $C$ is some sufficiently large constant. Then, from the result of Proposition~\ref{theorem:high_prob_Fourier_density}, we have $\Prob(A) \geq 1 - \delta$. Therefore, we obtain the following bound with $T_{2}$:
\begin{align*}
    T_{2} & = \Exs \brackets{(\nonpreg(x) - m(x))^2 (p_{0}(x) - \funcest_{n,\radius}(x))^2|A} \Prob(A) \\
    & \hspace{ 17 em} + \Exs \brackets{(\nonpreg(x) - m(x))^2 (p_{0}(x) - \funcest_{n,\radius}(x))^2| A^{c}} \Prob(A^{c}) \\
    & \leq 2 C \Exs \brackets{(\nonpreg(x) - m(x))^2} \parenth{ R^{\max \{2 - 2 \alpha, 0\}} \exp \parenth{-2 C_{1} \radius^{\alpha} } + \frac{\radius^{d} \log(2/ \delta)}{n} + \delta \parenth{p_{0}^2(x) + \radius^{2d}}},
\end{align*}
where the final inequality is due to the fact that $\Prob(A^{c}) \leq \delta$ and $(p_{0}(x) - \funcest_{n,\radius}(x))^2 \leq 2 ( p_{0}^2(x) + \funcest_{n,\radius}^2(x)) \leq 2 \parenth{(p_{0}^2(x) + \radius^{2d}}$. By choosing $\delta$ such that $\delta = \frac{\radius^{d}}{n(p_{0}^2(x) + \radius^{2d})}$, we obtain that
\begin{align}
    T_{2} \leq C' \Exs \brackets{(\nonpreg(x) - m(x))^2} \parenth{ R^{\max \{2 - 2 \alpha, 0\}} \exp \parenth{-2 C_{1} \radius^{\alpha} } + \frac{\radius^{d} \log (n \radius)}{n}}, \label{eq:bound_T2}
\end{align}
for some universal constant $C'$ when $\radius$ is sufficiently large.

For the upper bound of $T_{1}$, using the condition that $Y_{i} = m(X_{i}) + \epsilon_{i}$ for all $i \in [n]$, we have
\begin{align*}
    T_{1} & \leq 2 \Exs \brackets{\parenth{\frac{1}{n \pi^{d}} \sum_{i = 1}^{n} (m(X_{i}) - m(x)) \prod_{j = 1}^{d} \frac{\sinfunc(\radius(x_{j} - X_{ij}))}{x_{j} - X_{ij}}}^2 } \\
    & \hspace{ 8 em} + 2 \Exs \brackets{ \parenth{\frac{1}{n \pi^{d}} \sum_{i = 1}^{n} \epsilon_{i} \prod_{j = 1}^{d} \frac{\sinfunc(\radius(x_{j} - X_{ij}))}{x_{j} - X_{ij}}}^2 } = 2( S_{1} + S_{2}).
\end{align*}
Since $\Exs \brackets{ \parenth{\frac{1}{n} \sum_{i = 1}^{n} Y_{i} }^2} \leq \frac{1}{n} \Exs \brackets{Y_{1}^2} + \Exs^2 \brackets{Y_{1}}$ for any $Y_{1},\ldots, Y_{n}$ that are i.i.d., we find that
\begin{align*}
    S_{1} \leq \frac{1}{n \pi^{2d}} \Exs \brackets{(m(X) - m(x))^2 \prod_{j = 1}^{d} \frac{\sinfunc^2(\radius(x_{j} - X_{.j}))}{(x_{j} - X_{.j})^2}} & \\
    & \hspace{- 6 em} + \frac{1}{\pi^{2d}} \Exs^2 \brackets{(m(X) - m(x)) \prod_{j = 1}^{d} \frac{\sinfunc(\radius(x_{j} - X_{.j}))}{(x_{j} - X_{.j})}},
\end{align*}
where we denote $X = (X_{.1}, \ldots, X_{.d})$. From the result in equation~\eqref{eq:bias_bound_nonparametric_regression}, we have 
\begin{align*}
    \Exs^2 \brackets{(m(X) - m(x)) \prod_{j = 1}^{d} \frac{\sinfunc(\radius(x_{j} - X_{.j}))}{(x_{j} - X_{.j})}} \leq C'' \radius^{2 \max \{\deg(Q) + 1 - \alpha, 0\}} \exp( - 2 C_{1} \radius^{\alpha}),
\end{align*}
where $C''$ is some universal constant. Furthermore, based on Cauchy-Schwarz inequality and the assumptions of the theorem, we obtain the following bound
\begin{align*}
    \frac{1}{n \pi^{2d}} \Exs \brackets{(m(X) - m(x))^2 \prod_{j = 1}^{d} \frac{\sinfunc^2(\radius(x_{j} - X_{.j}))}{(x_{j} - X_{.j})^2}} \leq \frac{2 (\|m^2 \times p_{0}\|_{\infty} + m^2(x)) \radius^{d}}{n \pi^{2d}}.
\end{align*}
Putting the above results together, we find that
\begin{align*}
    S_{1} \leq \frac{2 (\|m^2 \times p_{0}\|_{\infty} + m^2(x)) \radius^{d}}{n \pi^{2d}} + \frac{C'' \radius^{2 \max \{\deg(Q) + 1 - \alpha, 0\}} \exp( - 2 C_{1} \radius^{\alpha})}{\pi^{2d}}.
\end{align*}
Similarly, since $\Exs(\epsilon_{i}) = 0$ and $\var(\epsilon_{i}) = \sigma^2$ for all $i \in [n]$, we have
\begin{align*}
    S_{2} = \frac{\sigma^2}{n \pi^{2d}} \Exs \brackets{\prod_{j = 1}^{d} \frac{\sinfunc^2(\radius(x_{j} - X_{.j}))}{(x_{j} - X_{.j})^2}} \leq \frac{\sigma^2 \|p_{0}\|_{\infty} \radius^{d}}{n \pi^{2d}}.
\end{align*}
%\subsection{Fixed design}
Collecting the above results, we find that
\begin{align}
    T_{1} \leq \frac{\parenth{ 4 (\|m^2 \times p_{0}\|_{\infty} + m^2(x)) + 2 \sigma^2 \|p_{0}\|_{\infty}} \radius^{d}}{n \pi^{2d}} & \nonumber \\
    & \hspace{-5 em} + \frac{C'' \radius^{2 \max \{\deg(Q) + 1 - \alpha, 0\}} \exp( - 2 C_{1} \radius^{\alpha})}{\pi^{2d}}. \label{eq:bound_T1}
\end{align}
Plugging the results from equations~\eqref{eq:bound_T2} and~\eqref{eq:bound_T1} into equation~\eqref{eq:bound_var_nonparametric_regression}, when $\radius \geq C'$ where $C'$ is some universal constant, we have 
\begin{align}
    \var(\nonpreg(x)) \leq \frac{C_{1}'}{p_{0}^2(x)} \Exs \brackets{(\nonpreg(x) - m(x))^2} \parenth{ R^{\max \{2 - 2 \alpha, 0\}} \exp \parenth{-2 C_{1} \radius^{\alpha} } + \frac{\radius^{d} \log (n \radius)}{n}} & \nonumber \\
    & \hspace{- 10 em} + \frac{C_{2}'}{p_{0}^2(x)} \frac{(m(x) + C_{3}') \radius^{d}}{n}, \label{eq:var_bound_nonparametric_regression}
\end{align}
where $C_{1}', C_{2}', C_{3}'$ are some universal constants. Combining the results with bias and variance in equations~\eqref{eq:bias_bound_nonparametric_regression} and~\eqref{eq:var_bound_nonparametric_regression}, we obtain the conclusion of the theorem.
%%%%%%%%%%%%%%%%%%%%%%%%%%%%%%%%%%%%%%%%%%%%%%%%%%%%%%%%%%%%%%%%%%%%%%%%%%%%%%%%%%%%%%%%%%%%%%%%%%%%%%%%%%%%%%%%%%%%%%%
\subsection{Proof of Theorem~\ref{theorem:bias_variance_transition_prob_distribution}}
\label{subsec:proof:theorem:bias_variance_transition_prob_distribution}
The proof of Theorem~\ref{theorem:bias_variance_transition_prob_distribution} shares a similar strategy with the proof of Theorem~\ref{theorem:random_nonparametric}. We first need the following lemmas regarding the MSE and concentration of the Fourier density estimator $\funcest_{n, \radius}$ when $(X_{1}, \ldots, X_{n})$ are a Markov sequence. 
\begin{lemma}
\label{lemma:bias_variance_fourier_markov_sequence}
Assume that $p_{0}$ is an upper--smooth density function of order $\alpha_{1} > 0$ such that $\|p_{0}\|_{\infty} < \infty$ and the transition probability operator $\mathcal{T}$ satisfies Assumption~\ref{assume:strong_mixing_transition_prob}. Then, there exist universal constants $C'$ and $C''$ such that as long as $\radius \geq C$ for some universal constant $C$ and for each $x \in \mathcal{X}$, we find that
\begin{align*}
    \Exs \brackets{(\funcest_{n, \radius}(x) - p_{0}(x))^2} \leq C' \radius^{\max \{2(1 - \alpha_{1}), 0 \}} \exp \parenth{-2C_{1} \radius^{\alpha_{1}}} + \frac{C'' \radius^{d}}{n},
\end{align*}
where $C_{1}$ is the associated constant with upper--supersmooth density function in Definition~\ref{def:tail_Fourier}.
\end{lemma}
\begin{proof}
The proof for the bias of $\funcest_{n, \radius}(x)$ is similar to the case when $(X_{1}, \ldots, X_{n})$ are independent. Therefore, from the proof of Theorem~\ref{theorem:bias_variance_Fourier_density}, for $\radius \geq C$ where $C$ is some universal constant, we have
\begin{align*}
    \abss{ \Exs \brackets{\funcest_{n, \radius}(x)} - p_{0}(x)} \leq C' \radius^{\max \{1 - \alpha_{1}, 0 \}} \exp \parenth{-C_{1} \radius^{\alpha_{1}}}.
\end{align*}
Here, $C_{1}$ is the associated constant with supersmooth density function in Definition~\ref{def:tail_Fourier}. Now, we proceed to bound the variance of $\funcest_{n, \radius}(x)$ where we utilize the assumption on the transition probability operator $\mathcal{T}$. Direct calculations yield
\begin{align*}
    \var\left(\funcest_{n, \radius}(x)\right) = \frac{1}{n} \var(Y_{1}) + \frac{2}{n^2} \sum_{i = 1}^{n - 1} (n - i) \cov(Y_{1}, Y_{i + 1}),
\end{align*}
where $Y_{i} = \frac{1}{\pi^{d}}\prod_{j = 1}^{n} \sin(\radius(x_{j} - X_{ij}))/(x_{j} - X_{ij})$. Since $\|p_{0}\|_{\infty} < \infty$, from the proof of Theorem~\ref{theorem:bias_variance_Fourier_density}, we have $\var(Y_{1}) \leq C' \radius^{d}$ where $C'$ is some universal constant. Furthermore, if we define $g(y) = \frac{1}{\pi^{d}}\prod_{j = 1}^{n} \sin(\radius(x_{j} - y_{j}))/(x_{j} - y_{j}) - \Exs \brackets{Y_{1}}$ for all $y \in \mathcal{X}$, then we find that
\begin{align*}
    \abss{ \cov(Y_{1}, Y_{i + 1})} = \abss{ \Exs \brackets{g(X_{1})(\mathcal{T}^{i}g)(X_{1})}} & \leq \sqrt{\Exs \brackets{g^2(X_{1})} \Exs \brackets{(\mathcal{T}^{i}g)^2(X_{1})}} \\
    & \leq \eta^{[i/\tau]} \Exs \brackets{g^2(X_{1})} \leq C' \eta^{[i/\tau]} \radius^{d},
\end{align*}
where $[x]$ denotes the greatest integer number that is less than or equal to $x$. Putting these results together, we have the following bound:
\begin{align*}
    \var(\funcest_{n, \radius}(x)) \leq \frac{C' \radius^{d}}{n} + \frac{2 C' \tau \parenth{\sum_{i = 0}^{[n/\tau]} \eta^{i}} \radius^{d}}{n^2} \leq \frac{C'' \radius^{d}}{n}.
\end{align*}
Combining all the previous results, we obtain the conclusion of Lemma~\ref{lemma:bias_variance_fourier_markov_sequence}.
\end{proof}
Our next lemma establishes the point-wise concentration bound of $\funcest_{n, \radius}(x)$ around its expectation for each $x \in \mathcal{X}$.
\begin{lemma}
\label{lemma:concentration_density_Markov_sequence}
Assume that $(X_{1}, \ldots, X_{n})$ are a Markov sequence with stationary density function $p_{0}$ and transition probability distribution $f(\cdot\mid\cdot)$. Then, for any $\delta \in (0,1)$, there exists universal constant $\bar{C}$ such that
\begin{align*}
    \Prob \parenth{\abss{\funcest_{n, \radius}(x) - \Exs \brackets{\funcest_{n, \radius}(x)}} \geq \bar{C} \sqrt{\frac{\radius^{d} \log(2/ \delta)}{n}}} \leq \delta
\end{align*}
\end{lemma}
\begin{proof}
The proof of Lemma~\ref{lemma:concentration_density_Markov_sequence} relies on Bernstein inequality for weakly dependent variable~\citep{Delyon_2009}. Define 
$$Y_{i} = \frac{1}{n \pi^{d}}\prod_{j = 1}^{n} \frac{\sin(\radius(x_{j} - X_{ij}))}{x_{j} - X_{ij}} - \Exs \brackets{\frac{1}{n \pi^{d}}\prod_{j = 1}^{n} \frac{\sin(\radius(x_{j} - X_{.j}))}{x_{j} - X_{.j}}}$$ 
for any $1 \leq i \leq n$ where the outer expectation is taken with respect to $X = (X_{.1},\ldots, X_{.d}) \sim p_{0}$. Furthermore, we denote $\mathcal{F}_{i} = \sigma(Y_{1}, \ldots, Y_{i})$ as the sigma-algebra generated by $Y_{1}, \ldots, Y_{i}$ for all $i \in [n]$. It is clear that $|Y_{i}| \leq C \radius^{d}/n$ for all $i \in [n]$ where $C$ is some universal constant. Additionally, for any $i \in [n]$ and $j < i$, we have $\abss{ \Exs \brackets{Y_{i} | \mathcal{F}_{j}}} \leq C'/n$ for some constant $C'$. Similarly, for each $i \in [n]$, we can check that $\abss{ \Exs \brackets{Y_{i}^2 | Y_{i - 1},\ldots, Y_{1}}} \leq C'' \radius^{d}/n^2$ for some universal constant $C''$. Therefore, based on the result of Theorem 4 in~\citep{Delyon_2009}, we have
\begin{align*}
    \Prob \parenth{\abss{\frac{1}{n} \sum_{i = 1}^{n} Y_{i} - \Exs \brackets{Y_{1}}} \geq t} \leq 2 \exp \parenth{- \frac{n t^2}{c (\radius^{d} + \radius^{d} t)}},
\end{align*}
where $c$ is some universal constant. By choosing $t = c_{1} \sqrt{\radius^{d} \log(2/ \delta)/n}$ for some universal constant $c_{1}$, we obtain the conclusion of Lemma~\ref{lemma:concentration_density_Markov_sequence}.
\end{proof}
Equipped with the results of Lemmas~\ref{lemma:bias_variance_fourier_markov_sequence} and~\ref{lemma:concentration_density_Markov_sequence}, we are ready to prove Theorem~\ref{theorem:bias_variance_transition_prob_distribution}. To ease the ensuing discussion, we define 
\begin{align*}
    \widehat{b}_{n, \radius}(x,y) = \frac{1}{n \pi^{2d}} \sum_{i=1}^{n - 1} \prod_{j=1}^d\frac{\sin(R(x-X_{ij}))}{x-X_{ij}} \cdot \frac{\sin(R(x-X_{(i+1)j}))}{x-X_{(i+1)j}}.
\end{align*}
Direct algebra leads to
\begin{align*}
    \transitionmat_{n, \radius}(y\mid x) - f(y\mid x) = \frac{\widehat{b}_{n,\radius}(x,y) p_{0}(x) - p(x,y) \funcest_{n, \radius}(x)}{p_{0}^2(x)} + \frac{ (\transitionmat_{n, \radius}(y|x) - f(y|x))(p_{0}(x) - \funcest_{n, \radius}(x))}{p_{0}(x)}.
\end{align*}
\paragraph{Bias of $\transitionmat_{n, \radius}(y\mid x)$:} 
An application of the Cauchy--Schwarz inequality leads to
\begin{align*}
    \parenth{\Exs \brackets{\transitionmat_{n, \radius}(y\mid x)} - f(y\mid x) }^2 & \leq 2 \frac{\parenth{ \Exs \brackets{\widehat{b}_{n,\radius}(x,y) p_{0}(x) - p(x,y) \funcest_{n, \radius}(x)}}^2}{p_{0}^4(x)} \\
    & \hspace{5 em} + 2 \frac{\Exs \brackets{(\transitionmat_{n, \radius}(y|x) - f(y|x))^2} \Exs \brackets{(p_{0}(x) - \funcest_{n, \radius}(x))^2}}{p_{0}^2(x)} \\
    & = \frac{A_{1}}{p_{0}^4(x)} + \frac{A_{2}}{p_{0}^2(x)}.
\end{align*}
For $A_{1}$, we find that
\begin{align*}
    \parenth{ \Exs \brackets{\widehat{b}_{n,\radius}(x,y) p_{0}(x) - p(x,y) \funcest_{n, \radius}(x)}}^2 & \leq 2 p_{0}^2(x) \parenth{\Exs \brackets{\widehat{b}_{n,\radius}(x,y)} - p(x,y)}^2 \\
    & \hspace{ 5 em} + 2 p^2(x,y) \parenth{\Exs \brackets{\funcest_{n, \radius}(x)} - p_{0}(x)}^2.
\end{align*}
Since both the density functions $p_{0}$ and $p$ are upper--smooth of order $\alpha_{1}$ and $\alpha_{2}$, we have the following bounds:
\begin{align*}
    \parenth{\Exs \brackets{\funcest_{n, \radius}(x)} - p_{0}(x)}^2 & \leq C \radius^{\max \{2(1 - \alpha_{1}), 0 \}} \exp( - 2 C_{1} \radius^{\alpha_{1}}), \\
    \parenth{\Exs \brackets{\widehat{b}_{n,\radius}(x,y)} - p(x,y)}^2 & \leq C' \radius^{\max \{2(1 - \alpha_{2}), 0 \}} \exp( - 2 C_{2} \radius^{\alpha_{2}}),
\end{align*}
where $C, C'$ are some universal constants while $C_{1}, C_{2}$ are constants associated with upper-smooth density functions (cf. Definition~\ref{def:tail_Fourier}). Putting these results together, we have
\begin{align*}
    A_{1} \leq c (p_{0}^2(x) + p^2(x, y)) \radius^{\max \{2(1 - \bar{\alpha}), 0\}} \exp(-c_{1}\radius^{\bar{\alpha}}),
\end{align*}
where $c$ and $c_{1}$ are some universal constants. 
For the term $A_{2}$, the result of Lemma~\ref{lemma:bias_variance_fourier_markov_sequence} leads to
\begin{align*}
    A_{2} \leq \parenth{ c_{1}' \radius^{\max \{2(1 - \alpha_{1}), 0\}} \exp(-c_{2}'\radius^{\alpha_{1}}) + \frac{c_{3}' \radius^{d}}{n}} \Exs \brackets{(\transitionmat_{n, \radius}(y\mid x) - f(y\mid x))^2}.
\end{align*}
Collecting all the above results, we obtain
\begin{align}
    \parenth{\Exs \brackets{\transitionmat_{n, \radius}(y\mid x)} - f(y\mid x) }^2 & \leq \frac{c (p_{0}^2(x) + p^2(x, y)) \radius^{\max \{2(1 - \bar{\alpha}), 0\}} \exp(-c_{1}\radius^{\bar{\alpha}})}{p_{0}^4(x)} \label{eq:bias_transition_prob_bound} \\
    & + \frac{\parenth{ c_{1}' \radius^{\max \{2(1 - \alpha_{1}), 0\}} \exp(-c_{2}'\radius^{\alpha_{1}}) + \frac{c_{3}' \radius^{d}}{n}} \Exs \brackets{(\transitionmat_{n, \radius}(y|x) - f(y|x))^2}}{p_{0}^2(x)}. \nonumber
\end{align}
\paragraph{Variance of $\transitionmat_{n, \radius}(y\mid x)$:} Similar to the proof of Theorem~\ref{theorem:random_nonparametric}, we have
\begin{align*}
    \var(\transitionmat_{n, \radius}(y|x)) & \leq 2 \frac{\Exs \brackets{\parenth{\widehat{b}_{n,\radius}(x,y) p_{0}(x) - p(x,y) \funcest_{n, \radius}(x)}^2}}{p_{0}^4(x)} \\
    & \hspace{4 em} + 2 \frac{\Exs \brackets{(\transitionmat_{n, \radius}(y\mid x) - f(y\mid x))^2(p_{0}(x) - \funcest_{n, \radius}(x))^2}}{p_{0}^2(x)} = \frac{B_{1}}{p_{0}^4(x)} + \frac{B_{2}}{p_{0}^2(x)}.
\end{align*}
Using the similar proof argument to bound the variance of Fourier regression estimator in the proof of Theorem~\ref{theorem:random_nonparametric} and the result of Lemma~\ref{lemma:concentration_density_Markov_sequence}, we have
\begin{align*}
    B_{2} \leq C \cdot \Exs \brackets{(\transitionmat_{n, \radius}(y\mid x) - f(y\mid x))^2} \parenth{ R^{\max \{2 - 2 \alpha_{1}, 0\}} \exp \parenth{-C_{1} \radius^{\alpha_{1}} } + \frac{\radius^{d} \log (n \radius)}{n}},
\end{align*}
where $C$ and $C_{1}$ are some universal constants. For the term $B_{1}$, we find that
\begin{align*}
    \Exs \brackets{\parenth{\widehat{b}_{n,\radius}(x,y) p_{0}(x) - p(x,y) \funcest_{n, \radius}(x)}^2} & \leq 2 p_{0}^2(x) \Exs \brackets{\parenth{\widehat{b}_{n,\radius}(x,y) - p(x, y)}^2 } \\
    & + 2 p^2(x, y) \Exs \brackets{\parenth{\funcest_{n, \radius}(x) - p_{0}(x)}^2}.
\end{align*}
With the similar proof technique as that of Theorem~\ref{theorem:bias_variance_Fourier_density}, since $p$ is upper-supersmooth density function of order $\alpha_{2}$, when $\radius$ is sufficiently large we have
\begin{align*}
    \parenth{\Exs \brackets{\widehat{b}_{n,\radius}(x,y)} - p(x, y)}^2 \leq c \radius^{\max \{2(1 - \alpha_{2}), 0 \}} \exp ( - c_{1} \radius^{\alpha_{2}}),
\end{align*}
where $c$ and $c_{1}$ are some universal constants. For the variance of $\widehat{b}_{n,\radius}(x,y)$, since the transition probability distributions of the Markov sequences $((X_{1}, X_{2}), \ldots, (X_{n - 1}, X_{n}))$ and 
$(X_{1},\ldots, X_{n})$ are similar, using the proof argument of Lemma~\ref{lemma:bias_variance_fourier_markov_sequence}, we have
%\begin{align*}
   $ \var(\widehat{b}_{n,\radius}(x,y)) \leq c_{2}\radius^{2d}/n$.
%\end{align*}
Putting all the above results together, we have
\begin{align}
    \var(\transitionmat_{n, \radius}(y\mid x)) & \leq \frac{c (p_{0}^2(x) + p^2(x, y)) \radius^{2d}}{n p_{0}^4(x)} \label{eq:variance_transition_prob_bound} \\
    & + \frac{C \parenth{\radius^{\max \{2(1 - \alpha_{1}), 0\}} \exp(-C_{1}\radius^{\alpha_{1}}) + \frac{\radius^{d} \log (n \radius)}{n}} \Exs \brackets{(\transitionmat_{n, \radius}(y\mid x) - f(y\mid x))^2}}{p_{0}^2(x)}. \nonumber
\end{align}
Combining the bounds of bias and variance of $\transitionmat_{n, \radius}(x)$ in equations~\eqref{eq:bias_transition_prob_bound} and~\eqref{eq:variance_transition_prob_bound}, we obtain the conclusion of Theorem~\ref{theorem:bias_variance_transition_prob_distribution}.
%%%%%%%%%%%%%%%%%%%%%%%%%%%%%%%%%%%%%%%%%%%%%%%%%%%%%%%%%%%%%%%%%%%%%%%%%%%%%%%%%%%%%%%%%%%%%%%%%%%%%%%%%%%%%%%%%%%%%%%
\section{Discussion}
\label{sec:discussion}
The key to the paper is the Fourier integral theorem. It suggests a natural Monte Carlo estimator for certain types of function and also explains why product of independent Fourier kernels is sufficient for multidimensional function estimation. This is not a property of any other kernel. We have covered estimating multivariate (mixing) density functions, nonparametric regression, and mode hunting, as well as modeling time series data. We show that when the  function is sufficiently smooth, the convergence rates of the proposed multivariate smoothing estimators are faster than those of standard kernel estimators. Finally, we note in passing that to account for the possible negativity of the estimators using the Fourier kernel, such as the Fourier density estimator or the Fourier transition probability estimator, we can take the maximum of these estimators and 0 or simply the absolute value of these estimators. Then, the new estimators are always non-negative and can be used as plug-in estimators for the true density in constructing the confidence intervals (see Section~\ref{sec:uniform_confidence_Fourier_density} and Section~\ref{subsec:nonparametric_regression}).

We now discuss a few questions that arise naturally from our work. First, the results in the paper are established under the assumptions of ``clean'' data. In practice, data are commonly contaminated; therefore, it is important to develop robust versions of the proposed estimators under contamination assumptions. Second, the idea of using the Fourier integral theorem for estimating the density function is potentially useful for developing efficient sampling. Finally, while we have considered an application of Fourier integral theorem to estimate the transition probability density for a Markov sequence, investigating the application of this theorem in other settings of dependent data, such as dynamic system, is also of interest.
%%%%%%%%%%%%%%%%%%%%%%%%%%%%%%%%%%%%%%%%%%%%%%%%%%%%%%%%%%%%%%%%%%%%%%%%%%%%%%%%%%%%%%%%%%%%%%%%%%%%%%%%%%%%%%%%%%%%%%%
\appendix
\section{Proofs of remaining results}
\label{sec:proof_remaining_result}
In this Appendix, we collect the proofs of remaining results in the paper. The values of universal constants (e.g., $C$, $c'$ etc.) can change from line-to-line.
%%%%%%%%%%%%%%%%%%%%%%%%%%%%%%%%%%%%%%%%%%%%%%%%%%%%%%%%%%%%%%%%%%%%%%%%%%%%%%%%%%%%%%%%%%%%%%%%%%%%%%%%%%%%%%%%%%%%%%%
\subsection{Proof of Proposition~\ref{theorem:mode_clustering_data_density}}
\label{subsec:proof:theorem:mode_clustering_data_density}
The proof of Theorem~\ref{theorem:mode_clustering_data_density} adapts some of the proof argument of Theorem 1 in~\citep{Chen_2016} to the Fourier density estimator. 

(a) Under Assumptions~\ref{assume:Hessian_matrix} and~\ref{assume:local_neighborhood}, based on the proof of Theorem 1 in~\citep{Chen_2016}, when $\sup_{x \in \mathcal{X}} |\funcest_{n, \radius}(x) - p_{0}(x)| \leq \frac{(\lambda^{*})^3}{16 d^2 C^2}$, for each local mode $x_{j}$ in $\mathcal{M}$, there exists a local mode $\widehat{x}_{j}$ in $\mathcal{M}_{n}$ such that $\widehat{x}_{j} \in x_{j} \oplus \frac{\lambda^{*}}{2Cd}$ where $C$ is universal constant given in Assumption~\ref{assume:local_neighborhood}. Furthermore, if $\sup_{x \in \mathcal{X}} \|\nabla \funcest_{n, \radius}(x) - \nabla p_{0}(x)\|_{\max} \leq \modclus$ and $\sup_{x \in \mathcal{X}} \|\nabla^2 \funcest_{n, \radius}(x) - \nabla^2 p_{0}(x)\|_{\max} \leq \frac{|\lambda^{*}|}{4d}$, then we have $\mathcal{M}_{n} \subset \mathcal{M} \oplus \frac{|\lambda^{*}|}{2Cd}$. Therefore, if we have the following conditions
\begin{align}
    & \sup_{x \in \mathcal{X}} |\funcest_{n, \radius}(x) - p_{0}(x)| \leq \frac{(\lambda^{*})^3}{16 d^2 C^2}, \quad \sup_{x \in \mathcal{X}} \|\nabla \funcest_{n, \radius}(x) - \nabla p_{0}(x)\|_{\max} \leq \modclus, \nonumber \\
    & \sup_{x \in \mathcal{X}} \|\nabla^2 \funcest_{n, \radius}(x) - \nabla^2 p_{0}(x)\|_{\max} \leq \frac{|\lambda^{*}|}{4d}, \label{eq:consistency_condition}
\end{align}
the number of estimated local modes $\widehat{K}_{n}$ and the true number of local modes $K$ are identical. The above bounds suggest that
\begin{align*}
    \Prob (\widehat{K}_{n} \neq K) & \leq \Prob \parenth{ \sup_{x \in \mathcal{X}} |\funcest_{n, \radius}(x) - p_{0}(x)| > \frac{(\lambda^{*})^3}{16 d^2 C^2})} + \Prob \parenth{ \sup_{x \in \mathcal{X}} \|\nabla \funcest_{n, \radius}(x) - \nabla p_{0}(x)\|_{\max} > \modclus} \nonumber  \\
    & + \Prob \parenth{ \sup_{x \in \mathcal{X}} \|\nabla^2 \funcest_{n, \radius}(x) - \nabla^2 p_{0}(x)\|_{\max} > \frac{|\lambda^{*}|}{4d}}
\end{align*}
Denote 
$$t = \max \left\{\frac{(\lambda^{*})^3}{16 d^2 C^2}, \modclus, \frac{|\lambda^{*}|}{4d}\right\}.$$ 
Based on the uniform concentration bounds of $\funcest_{n, \radius}, \nabla \funcest_{n, \radius}, \nabla^2 \funcest_{n, \radius}$ in Theorems~\ref{theorem:uniform_bound_Fourier_density} and~\ref{theorem:uniform_bound_derivatives}, by choosing $\radius$ such that $\radius \geq C'$, $C_{1}' \radius^{\max{3 - \alpha, 0}} \exp(-C_{1} \radius^{\alpha}) \leq t/ 2$ and 
$$C_{2}' \sqrt{\frac{\radius^{(2d + 4)} \parenth{\log(2/ \delta) + d (d + 3) \log \radius + d (\log d + \text{Diam}(\mathcal{X})}}{n}} \leq t/ 2$$ where $C', C_{1}', C_{2}'$ are some universal constants depending on the constants in Theorems~\ref{theorem:uniform_bound_Fourier_density} and~\ref{theorem:uniform_bound_derivatives}, we have
\begin{align*}
    & \Prob \parenth{ \sup_{x \in \mathcal{X}} |\funcest_{n, \radius}(x) - p_{0}(x)| > \frac{(\lambda^{*})^3}{16 d^2 C^2})} \leq \delta, \quad \Prob \parenth{ \sup_{x \in \mathcal{X}} \|\nabla \funcest_{n, \radius}(x) - \nabla p_{0}(x)\|_{\max} > \modclus} \leq \delta, \\
    & \Prob \parenth{ \sup_{x \in \mathcal{X}} \|\nabla^2 \funcest_{n, \radius}(x) - \nabla^2 p_{0}(x)\|_{\max} > \frac{|\lambda^{*}|}{4d}} \leq \delta.
\end{align*}
As a consequence, we have $\Prob (\widehat{K}_{n} \neq K) \leq 3 \delta$, which leads to the conclusion of part (a).

(b) Assume that the conditions~\eqref{eq:consistency_condition} hold such that $\widehat{K}_{n} = K$. We now proceed to study the convergence rate of $\mathcal{M}_{n}$ to $\mathcal{M}$ under the Hausdorff distance. For each local mode $x_{j} \in \mathcal{M}$, we recall that the local mode $\widehat{x}_{j} \in \mathcal{M}_{n}$ is the closest local mode in $\mathcal{M}_{n}$ to $x_{j}$. An application of Taylor expansion up to the second order leads to
\begin{align*}
    0 = \nabla \funcest_{n, \radius}(\widehat{x}_{j}) & = \nabla \funcest_{n, \radius}(x_{j}) + (\widehat{x}_{j} - x_{j})^{\top} \nabla^2 \funcest_{n, \radius}(x_{j}) + R_{j},
\end{align*}
where $R_{j}$ is the Taylor remainder such that $\|R_{j}\| = o(\|x_{j} - \widehat{x}_{j}\|)$. Given the conditions~\eqref{eq:consistency_condition}, the matrix $\nabla^2 \funcest_{n, \radius}(x_{j})$ is invertible. Therefore, we have
\begin{align*}
    \| \widehat{x}_{j} - x_{j}\| \leq \| \nabla \funcest_{n, \radius}(x_{j}) + R_{j}\| \cdot \|\nabla^2 \funcest_{n, \radius}(x_{j})^{-1}\|_{op},
\end{align*}
where $\|.\|_{op}$ denotes the operator norm. Note that, $\|\nabla^2 \funcest_{n, \radius}(x_{j})^{-1}\|_{op}$ is bounded due to the conditions~\eqref{eq:consistency_condition}. To obtain the conclusion of part (b), it is sufficient to demonstrate that
\begin{align}
    \Prob \parenth{ \| \nabla \funcest_{n, \radius}(x_{j})\|_{\max} \geq c_{1} \radius^{\max\{2 - \alpha, 0\}} \exp \parenth{-C_{1} \radius^{\alpha}} + c_{2} \sqrt{\frac{\radius^{d + 1} \log(2/ \delta)}{n}}} \leq \delta, \label{eq:point_wise_first_order_derivatives}
\end{align}
where $c_{1}$ and $c_{2}$ are some universal constants. Note that, we can directly apply the uniform concentration bound in Theorem~\ref{theorem:uniform_bound_derivatives} to obtain the above point-wise concentration bound with an extra $log \radius$ term. However, here we do not want to have the $\log R$ in the point-wise concentration bound; therefore, we will need use the argument of Proposition~\ref{theorem:high_prob_Fourier_density} to remove the $\log \radius$ term.

Note that, $\nabla p_{0}(x_{j}) = 0$. An application of triangle inequality leads to
\begin{align*}
    \| \nabla \funcest_{n, \radius}(x_{j})\|_{\max} & = \| \nabla \funcest_{n, \radius}(x_{j}) - \nabla p_{0}(x_{j})\|_{\max} \\
    & \leq \| \nabla \funcest_{n, \radius}(x_{j}) - \Exs \brackets{\nabla \funcest_{n, \radius}(x_{j})}\|_{\max} + \| \Exs \brackets{\nabla \funcest_{n, \radius}(x_{j})} - \nabla p_{0}(x_{j}) \|_{\max}.
\end{align*}
In the right hand side of the above bound, the upper bound for the second term has been established in Theorem~\ref{theorem:bias_variance_Fourier_density_derivatives}; therefore, we only focus on bounding the first term. We first establish the concentration bound for $\abss{\parenth{ \nabla^{i} \funcest_{n, \radius}(x_{j})}_{u} - \parenth{ \Exs \brackets{ \nabla^{i} \funcest_{n, \radius}(x_{j})}}_{u}}$ for any $1 \leq u \leq d$. Following the proof of Theorem~\ref{theorem:uniform_bound_derivatives}, we denote $h(y, s) = \frac{\sin(\radius(y - s))}{y - s}$ for all $y, s \in \mathbb{R}$. Then, we can rewrite $\parenth{ \nabla \funcest_{n, \radius}(x_{j})}_{u}$ as follows:
\begin{align*}
    \parenth{ \nabla \funcest_{n, \radius}(x_{j})}_{u} = \frac{1}{n (2\pi)^{d}} \sum_{i = 1}^{n} \prod_{l = 1}^{d} \frac{\partial^{|B_{l}|}}{\partial{x_{jl}^{|B_{l}|}}} h(x_{jl}, X_{jl}),
\end{align*}
where $B_{l} = \mathbbm{1}_{\{u = l\}}$ for any $l \in [d]$ and for any given $1 \leq u \leq d$. We denote $Y_{i} = \prod_{l = 1}^{d} \frac{\partial^{|B_{l}|}}{\partial{x_{jl}^{|B_{l}|}}} h(x_{jl}, X_{jl})$ for all $i \in [n]$. Then, we have $|Y_{i}| \leq \radius^{d + 1}$ for all $i \in [n]$. Furthermore, $\var(Y_{i}) \leq C \radius^{d + 2}$ for all $i \in [n]$ where $C$ is some universal constant. For any $t \in (0, C]$, an application of Bernstein's inequality shows that
\begin{align*}
    \Prob \parenth{\abss{\frac{1}{n} \sum_{i = 1}^{n} Y_{i} - \Exs \brackets{Y_{1}}} \geq t} \leq 2 \exp \parenth{- \frac{n t^2}{2 C \radius^{d + 2} + 2 \radius^{d + 1} t/ 3}}.
\end{align*}
By choosing $t = \bar{C} \sqrt{\frac{\radius^{d + 2} \log(2/ \delta)}{n}}$ where $\bar{C}$ is some universal constant, we find that
\begin{align*}
    \Prob \parenth{\abss{\frac{1}{n} \sum_{i = 1}^{n} Y_{i} - \Exs \brackets{Y_{1}}} \geq t} \leq \delta.
\end{align*}
Collecting the above results together, we have
\begin{align*}
    \Prob \parenth{\| \nabla \funcest_{n, \radius}(x_{j}) - \nabla p_{0}(x_{j})\|_{\max} \geq \bar{C} d \sqrt{\frac{\radius^{d + 2} \log(2/ \delta)}{n}} } \leq \delta.
\end{align*}
Therefore, the concentration bound~\eqref{eq:point_wise_first_order_derivatives} is proved. As a consequence, we reach the conclusion of part (b).
%%%%%%%%%%%%%%%%%%%%%%%%%%%%%%%%%%%%%%%%%%%%%%%%%%%%%%%%%%%%%%%%%%%%%%%%%%%%%%%%%%%%%%%%%%%%%%%%%%%%%%%%%%%%%%%%%%%%%%%
\subsection{Proof of Proposition~\ref{theorem:mode_clustering_mixing_density}}
\label{subsec:proof:theorem:mode_clustering_mixing_density}
The proof of Proposition~\ref{theorem:mode_clustering_mixing_density} is similar to that of Proposition~\ref{theorem:mode_clustering_data_density}. Indeed, to obtain the conclusion of Proposition~\ref{theorem:mode_clustering_mixing_density}, it is sufficient to establish the uniform concentration bound for the derivatives of Fourier deconvolution estimator $\deconvest_{n, \radius}$.
\begin{lemma}
\label{lemma:concentration_derivative_deconvolution}
Assume that $f $ is lower-supersmooth density function of order $\alpha_{1} > 0$ and $\deconv \in \mathcal{C}^{r}(\Theta)$ is upper-supersmooth density function of order $\alpha_{2} > 0$ such that $\alpha_{2} \geq \alpha_{1}$ for some given $r \in \mathbb{N}$ and $\Theta$ is a bounded subset of $\mathbb{R}^{d}$. Then, there exist universal constants $\{C_{i}'\}_{i = 1}^{r}$, $C'$, $\bar{C}$ such that as long as $\radius \geq C'$ and $1 \leq i \leq r$, we have
\begin{align*}
    \Prob \biggr(\sup_{\theta \in \Theta} \| \nabla^{i} \deconvest_{n, \radius}(\theta) - \nabla^{i} g(\theta)\|_{\max} & \geq C_{i}' \radius^{\max \{i + 1 - \alpha_{2}, 0\}} \exp \parenth{- C_{1} \radius^{\alpha_{2}}} \nonumber \\
    & + \bar{C} \sqrt{\frac{\radius^{2(i + d) + \alpha_{1}}\exp(2C_{2}d\radius^{\alpha_{1}}) \log (2/\delta)}{n}} \biggr) \leq \delta,
\end{align*}
where $C_{1}$ and $C_{2}$ are the constants given in Definition~\ref{def:tail_Fourier}.
\end{lemma}
\begin{proof}
The proof of Lemma~\ref{lemma:concentration_derivative_deconvolution} proceeds in the similar way as that of Theorem~\ref{theorem:uniform_bound_derivatives}. Recall that the Fourier deconvolution estimator $\deconvest_{n, \radius}$ is given by:
\begin{align*}
    \deconvest_{n,\radius}(\theta) = \frac{1}{n (2 \pi)^{d}} \sum_{j = 1}^{n} \int_{[-\radius, \radius]^{d}} \frac{\cosfunc(s^{\top}( \theta - X_{j}))}{\widehat{f}(s)} ds.
\end{align*}
Without loss of generality, we assume that $i = 4l + 1$ for some $l \in \mathbb{N}$ (The proof argument for other cases of $i$ is similar). Then, from the proof of Theorem~\ref{theorem:deconvolution_bias_variance_supersmooth_derivatives}, we have 
\begin{align*}
    (\nabla^{i} \deconvest_{n,\radius}(\theta))_{u_{1}\ldots u_{i}} = - \frac{1}{n (2\pi)^{d}} \sum_{j = 1}^{n} \int_{[-\radius, \radius]^{d}} s_{u_{1}}\ldots s_{u_{i}} \cdot \frac{\sinfunc(s^{\top} (\theta - X_{j})) }{\widehat{f}(s)} d s,
\end{align*}
for all $1 \leq u_{1}, \ldots, u_{i} \leq d$. We denote $Y_{j} = -\frac{1}{(2\pi)^{d}} \int_{[-\radius, \radius]^{d}} s_{u_{1}}\ldots s_{u_{i}} \cdot \frac{\sinfunc(s^{\top} (\theta - X_{j})) }{\widehat{f}(s)} d s$ for any $j \in [n]$. Since $f$ is lower-supersmooth of order $\alpha_{1}$, we have $|Y_{j}| \leq C \radius^{i + d} \exp(C_{2}d\radius^{\alpha_{1}})$ and $\Exs \brackets{|Y_{j}|} \leq C \radius^{i + d} \exp(C_{2}d\radius^{\alpha_{1}})$ where $C$ is some universal constant and $C_{2}$ is a given constant in Definition~\ref{def:tail_Fourier} with lower-supersmooth density function. An application of Bernstein inequality leads to 
\begin{align*}
    \Prob \parenth{ \sup_{\theta \in \Theta} \abss{(\nabla^{i} \deconvest_{n,\radius}(\theta))_{u_{1}\ldots u_{i}}  - \parenth{ \Exs \brackets{ \nabla^{i} \deconvest_{n,\radius}(\theta)} }_{u_{1}\ldots u_{i}}} > t} \leq 4 \mathcal{N}_{[]} \parenth{ t/8, \mathcal{F}', \mathbb{L}_{1}(P)} & \\
    & \hspace{- 10 em} \times \exp \parenth{ - \frac{96 n t^2}{76 C^2 \radius^{2i + 2d} \exp(2 C_{2}d \radius^{\alpha_{1}})} },
\end{align*}
where $\mathcal{F}' = \{f_{\theta}: \mathbb{R}^{d} \to \mathbb{R}: f_{\theta}(t) = -\frac{1}{(2\pi)^{d}} \int_{[-\radius, \radius]^{d}} s_{u_{1}}\ldots s_{u_{i}} \cdot \frac{\sinfunc(s^{\top} (\theta - t)) }{\widehat{f}(s)} d s \ \text{for all} \ \theta \in \Theta, t \in \mathbb{R}^{d} \}$. For any $f_{\theta_{1}}, f_{\theta_{2}} \in \mathcal{F}'$, we can check that
\begin{align*}
    |f_{\theta_{1}}(t) - f_{\theta_{2}}(t)| \leq d \radius^{d + i + 1} \exp(C_{2}d\radius^{\alpha_{1}}) \enorm{\theta_{1} - \theta_{2}}.
\end{align*}
Therefore, we have the following upper bound on the bracketing entropy:
\begin{align*}
    \mathcal{N}_{[]} \parenth{ t/8, \mathcal{F}', \mathbb{L}_{1}(P)} \leq \parenth{ \frac{ 4 d \sqrt{d} \cdot \text{Diam}(\Theta) \radius^{d + i + 1}\exp(C_{2}d\radius^{\alpha_{1}})}{t}}^{d}.
\end{align*}
Collecting the above results, when $R$ is sufficiently large, by choosing $$t = C' \sqrt{\frac{\radius^{2(i + d) + \alpha_{1}}\exp(2C_{2}d\radius^{\alpha_{1}}) \log (2/\delta)}{n}},$$ we have
\begin{align*}
    \Prob \parenth{ \sup_{\theta \in \Theta} \abss{(\nabla^{i} \deconvest_{n,\radius}(\theta))_{u_{1}\ldots u_{i}}  - \parenth{ \Exs \brackets{ \nabla^{i} \deconvest_{n,\radius}(\theta)} }_{u_{1}\ldots u_{i}}} > t} \leq \delta.
\end{align*}
Taking an union bound over $u_{1}, \ldots, u_{i}$ with the above inequality and combining it with the result of Theorem~\ref{theorem:deconvolution_bias_variance_supersmooth_derivatives}, we obtain the conclusion of the lemma.
\end{proof}
%%%%%%%%%%%%%%%%%%%%%%%%%%%%%%%%%%%%%%%%%%%%%%%%%%%%%%%%%%%%%%%%%%%%%%%%%%%%%%%%%%%%%%%%%%%%%%%%%%%%%%%%%%%%%%%%%%%%%%%
\subsection{Proof of Proposition~\ref{theorem:mode_regression_super_smooth}}
\label{subsec:proof:theorem:mode_regression_super_smooth}
The proof of Proposition~\ref{theorem:mode_regression_super_smooth} follows the proof argument of Theorems 3 and 4 in~\cite{Yen_Chen_2016} with the main difference is in the uniform concentration bound of the Fourier density estimator $\funcestmode_{n, \radius}$ and its derivatives around the true joint density $p_{0}$. Here, we provide the main steps of the proof for part (a) for the completeness and the proof for part (b) can be argued similarly. 

From the proof of Proposition~\ref{theorem:mode_clustering_data_density}, for each $x \in \mathcal{X}$, under the Assumptions~\ref{assume:manifold_regression} and~\ref{assume:Hessian_manifold_regression} when $\sup_{y} |\frac{\partial^{i}{\funcestmode_{n, \radius}}}{\partial{y}^{i}} (x, y) - \frac{\partial^{i}{p_{0}}}{\partial{y}^{i}} (x, y)| \leq C$ for any $i \in \{0, 1, 2\}$ where $C$ is some universal constant depending on $\curvature$, then for each local mode of $\mathcal{M}(x)$ there exists a unique local mode of $\mathcal{M}_{n}(x)$ that is closest to it. Given this property, with the similar proof argument as that of Theorem 3 in~\cite{Yen_Chen_2016}, for each $x \in \mathcal{X}$ we obtain that
\begin{align*}
    \frac{1}{\mathcal{H}(\mathcal{M}_{n}(x), \mathcal{M}(x))} \abss{ \mathcal{H}(\mathcal{M}_{n}(x), \mathcal{M}(x)) - \max_{y \in \mathcal{M}(x)} \left\{\frac{\abss{\frac{\partial{\funcestmode_{n, \radius}}}{\partial{y}} (x, y)}}{\abss{\frac{\partial^2{p_{0}}}{\partial{y}^2} (x, y)}} \right\}} & \\
    & \hspace{- 9 em} = \mathcal{O}_{P} \parenth{\max_{0 \leq i \leq 2} \sup_{x, y} \abss{ \frac{\partial^{i}{\funcestmode_{n, \radius}}}{\partial{y}^{i}} (x, y) - \frac{\partial^{i}{p_{0}}}{\partial{y}^{i}} (x, y)}}.
\end{align*}
The above inequality leads to the following bound:
\begin{align*}
    \sup_{x \in \mathcal{X}} \mathcal{H}(\mathcal{M}_{n}(x), \mathcal{M}(x)) = \sup_{x \in \mathcal{X}, y \in \mathcal{M}(x)} \left\{\frac{\abss{\frac{\partial{\funcestmode_{n, \radius}}}{\partial{y}} (x, y)}}{\abss{\frac{\partial^2{p_{0}}}{\partial{y}^2} (x, y)}} \right\} & \\
    & \hspace{- 14 em} + \mathcal{O}_{P} \parenth{\left\{\max_{0 \leq i \leq 2} \sup_{x, y} \abss{ \frac{\partial^{i}{\funcestmode_{n, \radius}}}{\partial{y}^{i}} (x, y) - \frac{\partial^{i}{p_{0}}}{\partial{y}^{i}} (x, y)}\right\} \sup_{x \in \mathcal{X}} \mathcal{H}(\mathcal{M}_{n}(x), \mathcal{M}(x))}.
\end{align*}
Since $p_{0}$ is upper-supersmooth density function of order $\alpha > 0$ and $p_{0} \in \mathcal{C}^{2}(\mathcal{X} \times \mathcal{Y})$, from Theorem~\ref{theorem:uniform_bound_derivatives} we have
\begin{align*}
    \max_{0 \leq i \leq 2} \sup_{x, y} \abss{ \frac{\partial^{i}{\funcestmode_{n, \radius}}}{\partial{y}^{i}} (x, y) - \frac{\partial^{i}{p_{0}}}{\partial{y}^{i}} (x, y)} = \mathcal{O}_{P} \parenth{\radius^{\max \{3 - \alpha, 0\}} \exp(-C_{1} \radius^{\alpha}) + \sqrt{\frac{\radius^{d + 5} \log \radius}{n}}},
\end{align*}
where $C_{1}$ is a given constant in Definition~\ref{def:tail_Fourier}. This bound suggests that it is sufficient to upper bound $\sup_{x \in \mathcal{X}, y \in \mathcal{M}(x)} \left\{\frac{\abss{\frac{\partial{\funcestmode_{n, \radius}}}{\partial{y}} (x, y)}}{\abss{\frac{\partial^2{p_{0}}}{\partial{y}^2} (x, y)}} \right\}$ to obtain the conclusion of the proposition. In fact, by triangle inequality, we have
\begin{align*}
    \sup_{x \in \mathcal{X}, y \in \mathcal{M}(x)} \left\{\frac{\abss{\frac{\partial{\funcestmode_{n, \radius}}}{\partial{y}} (x, y)}}{\abss{\frac{\partial^2{p_{0}}}{\partial{y}^2} (x, y)}} \right\} & \leq \sup_{x \in \mathcal{X}, y \in \mathcal{M}(x)} \left\{\frac{\abss{\frac{\partial{\funcestmode_{n, \radius}}}{\partial{y}} (x, y) - \Exs \brackets{\frac{\partial{\funcestmode_{n, \radius}}}{\partial{y}} (x, y)}}}{\abss{\frac{\partial^2{p_{0}}}{\partial{y}^2} (x, y)}} \right\} \\
    & + \sup_{x \in \mathcal{X}, y \in \mathcal{M}(x)} \left\{\frac{\abss{\Exs \brackets{\frac{\partial{\funcestmode_{n, \radius}}}{\partial{y}} (x, y)} - \frac{\partial{p_{0}}}{\partial{y}} (x, y)}}{\abss{\frac{\partial^2{p_{0}}}{\partial{y}^2} (x, y)}} \right\}.
\end{align*}
Given Assumption~\ref{assume:Hessian_manifold_regression}, we have
\begin{align*}
    \sup_{x \in \mathcal{X}, y \in \mathcal{M}(x)} \left\{\frac{\abss{\Exs \brackets{\frac{\partial{\funcestmode_{n, \radius}}}{\partial{y}} (x, y)} - \frac{\partial{p_{0}}}{\partial{y}} (x, y)}}{\abss{\frac{\partial^2{p_{0}}}{\partial{y}^2} (x, y)}} \right\} & \leq \frac{1}{\curvature} \sup_{x \in \mathcal{X}, y \in \mathcal{M}(x)} \abss{\Exs \brackets{\frac{\partial{\funcestmode_{n, \radius}}}{\partial{y}} (x, y)} - \frac{\partial{p_{0}}}{\partial{y}} (x, y)} \\
    & \leq C' \radius^{\max \{2 - \alpha, 0\}} \exp ( - C_{1} \radius^{\alpha}),
\end{align*}
where $C'$ is some universal constant and the second inequality is due to Theorem~\ref{theorem:bias_variance_Fourier_density_derivatives}. Following the proof of Theorem~\ref{theorem:uniform_bound_derivatives}, we denote $W_{i} = \frac{\prod_{j = 1}^{d} \frac{\sinfunc(\radius(x_{j} - X_{ij}))}{x_{j} - X_{ij}} \cdot \frac{\partial}{\partial{y}} \frac{\sinfunc(\radius( y - Y_{i}))}{y - Y_{i}}}{\abss{\frac{\partial^2{p_{0}}}{\partial{y}^2} (x, y)}}$ for each $i \in [n]$. Then, it is clear that $|W_{i}| \leq \frac{\radius^{d + 2}}{\curvature}$ and $\Exs \brackets{|W_{i}|} \leq \frac{\radius}{\curvature}$ for all $i \in [n]$ and $x \in \mathcal{X}, y \in \mathcal{M}(x)$. Therefore, an application of Bernstein inequality leads to 
\begin{align*}
    \Prob \parenth{ \sup_{x \in \mathcal{X}, y \in \mathcal{M}(x)} \left\{\frac{\abss{\frac{\partial{\funcestmode_{n, \radius}}}{\partial{y}} (x, y) - \Exs \brackets{\frac{\partial{\funcestmode_{n, \radius}}}{\partial{y}} (x, y)}}}{\abss{\frac{\partial^2{p_{0}}}{\partial{y}^2} (x, y)}} \right\} > t} & \\ 
    & \hspace{- 12 em} \leq \parenth{ \frac{ 4 ( d + 1) \sqrt{d + 1} \cdot \text{Diam}(\mathcal{X} \times \mathcal{Y}) \radius^{d + 3}}{t \curvature}}^{d + 1} \exp \parenth{ - \frac{96 n t^2}{76 \radius^{d + 3}}}.
\end{align*}
By choosing $t = \bar{C} \sqrt{\frac{\radius^{d + 3} \parenth{\log(2/ \delta) + d (d + 3) \log \radius + d (\log d + \text{Diam}(\mathcal{X} \times \mathcal{Y})}}{n}}$ where $\bar{C}$ is some universal constant, we have
\begin{align*}
    \Prob \parenth{ \sup_{x \in \mathcal{X}, y \in \mathcal{M}(x)} \left\{\frac{\abss{\frac{\partial{\funcestmode_{n, \radius}}}{\partial{y}} (x, y) - \Exs \brackets{\frac{\partial{\funcestmode_{n, \radius}}}{\partial{y}} (x, y)}}}{\abss{\frac{\partial^2{p_{0}}}{\partial{y}^2} (x, y)}} \right\} > t} \leq \delta.
\end{align*}
Putting the above results together, there exists universal constant $C$ such that
\begin{align*}
    \Prob \parenth{ \sup_{x \in \mathcal{X}} \mathcal{H}(\mathcal{M}_{n}(x), \mathcal{M}(x)) \geq C \brackets{\radius^{\max \{2 - \alpha, 0\}} \exp ( - C_{1} \radius^{\alpha}) + \sqrt{\frac{\radius^{d + 3} \log \radius \log(2/ \delta)}{n}}}} \geq 1 - \delta.
\end{align*}
As a consequence, we reach the conclusion of the proposition.
\subsection{Proof of Proposition~\ref{prop:boostrap_Gaussian_process}}
\label{subsec:proof:prop:boostrap_Gaussian_process}
The proof of Proposition~\ref{prop:boostrap_Gaussian_process} follows from that of Proposition~\ref{prop:approx_Gaussian_process}. Here, we only provide the proof sketch. To facilitate the proof argument, we denote $P_{n} = \frac{1}{n} \sum_{i = 1}^{n} \delta_{X_{i}}$ the empirical measure associated with the data $X_{1}, \ldots, X_{n}$. Recall that, the functional space $\mathcal{F}$ in equation~\eqref{eq:empirical_process} is given by:
\begin{align*}
    \mathcal{F} = \{f_{x}: \mathbb{R}^{d} \to \mathbb{R}: f_{x}(t) = \frac{1}{\pi^{d}}\prod_{i = 1}^{d} \frac{\sin(\radius(x_{i} - t_{i}))}{\radius(x_{i} - t_{i})} \ \text{for all} \ x \in \mathcal{X}, t \in \mathbb{R}^{d} \}. 
\end{align*}
We denote the Gaussian process $\mathbb{B}'$ on $\mathcal{F}$ with the covariance matrix given by:
\begin{align}
    \text{cov}(\mathbb{B}'(f_{1}, f_{2})) = \Exs_{P_{n}} \brackets{f_{1}(X) f_{2}(X) } -  \Exs \brackets{f_{1}(X)} \Exs_{P_{n}} \brackets{ f_{2}(X) }, \label{eq:empirical_Gaussian_process}
\end{align}
for any $f_{1}, f_{2} \in \mathcal{F}$. Note that, the difference between the Gaussian process $\mathbb{B}$ with covariance matrix given in equation~\eqref{eq:Gaussian_process} and the Gaussian process $\mathbb{B}'$ is that the outer expectations in the covariance matrices of $\mathbb{B}$ are taken with respect to the unknown distribution $P$ while those of $\mathbb{B}'$ are taken with respect to the empirical distribution $P_{n}$.

For the remaining argument, we assume that $X_{1}^{n} = (X_{1}, \ldots, X_{n})$ is a fixed sample to simplify the presentation. Then, from the result of Proposition~\ref{prop:approx_Gaussian_process}, we have
\begin{align*}
     \sup_{t \geq 0} \abss{\Prob \parenth{\sqrt{\frac{n}{\radius^{d}}}\sup_{x \in \mathcal{X}} \abss{\funcest_{n, \radius}(x) - \Exs \brackets{ \funcest_{n, \radius}(x)}} < t \; \big| \; X_{1}^{n}} - \Prob \parenth{\bold{B}' < t \; \big| \; X_{1}^{n} }} \leq  C \frac{(\log n)^{(7+d)/8}}{n^{1/8}},
\end{align*}
where $\bold{B}' = \sqrt{\radius^{d}} \sup_{f \in \mathcal{F}} \abss{\mathbb{B}'(f)}$. 

Now, we proceed to bound $\sup_{t \geq 0} \abss{\Prob \parenth{\bold{B}' < t \; \big| \; X_{1}^{n} } - \Prob \parenth{\bold{B} < t}}$. Since $\mathcal{X}$ is a bounded subset of $\mathbb{R}^{d}$, as in the proof of Proposition~\ref{prop:approx_Gaussian_process}, we have $N= \sup_{P} \mathcal{N}_{2} \parenth{ \epsilon/8, \mathcal{F}, P} \leq \parenth{ \frac{ 4 d \sqrt{d} \cdot \text{Diam}(\mathcal{X}) \radius^2}{\epsilon}}^{d}$. Denote $\overline{\mathcal{F}} = \{\bar{f}_{1}, \ldots, \bar{f}_{N}\}$ as the set of $\epsilon$-covering of $\mathcal{F}$. An application of triangle inequality leads to
\begin{align*}
    \sup_{t \geq 0} \abss{\Prob \parenth{\bold{B}' < t \; \big| \; X_{1}^{n} } - \Prob \parenth{\bold{B} < t}} & \leq \sup_{t \geq 0} \abss{\Prob \parenth{\bold{B}' < t \; \big| \; X_{1}^{n} } - \Prob \parenth{\sup_{f \in \overline{\mathcal{F}}} \sqrt{\radius^{d}}|\mathbb{B}'(f)| < t \; \big| \; X_{1}^{n} }} \\
    & \hspace{-2 em} + \sup_{t \geq 0} \abss{\Prob \parenth{\sup_{f \in \overline{\mathcal{F}}} \sqrt{\radius^{d}}|\mathbb{B}'(f)| < t \; \big| \; X_{1}^{n} } - \Prob \parenth{\sup_{f \in \overline{\mathcal{F}}} \sqrt{\radius^{d}}|\mathbb{B}(f)| < t}} \\
    & \hspace{-2 em} + \sup_{t \geq 0} \abss{\Prob \parenth{\sup_{f \in \in \overline{\mathcal{F}}} \sqrt{\radius^{d}}|\mathbb{B}(f)| < t} - \Prob \parenth{\bold{B} < t}}.
\end{align*}
It is sufficient to bound $\sup_{t \geq 0} \abss{\Prob \parenth{\sup_{f \in \overline{\mathcal{F}}} \sqrt{\radius^{d}}|\mathbb{B}'(f)| < t \; \big| \; X_{1}^{n} } - \Prob \parenth{\sup_{f \in \overline{\mathcal{F}}} \sqrt{\radius^{d}}|\mathbb{B}(f)| < t}}$. From Theorem 2 in~\citep{Victor_2014c}, we find that
\begin{align*}
    \sup_{t \geq 0} \abss{\Prob \parenth{\sup_{f \in \overline{\mathcal{F}}} \sqrt{\radius^{d}}|\mathbb{B}'(f)| < t \; \big| \; X_{1}^{n} } - \Prob \parenth{\sup_{f \in \overline{\mathcal{F}}} \sqrt{\radius^{d}}|\mathbb{B}(f)| < t}} \leq C \Delta^{1/3} \parenth{1 \vee \log(N/ \Delta)},
\end{align*}
where $\Delta = \radius^{d} \max_{1 \leq i,j \leq N} \abss{ \text{cov}(\mathbb{B}'(\bar{f}_{i}, \bar{f}_{j})) - \text{cov}(\mathbb{B}(\bar{f}_{i}, \bar{f}_{j}))}$. Using the proof similar argument as that of Theorem~\ref{theorem:uniform_bound_Fourier_density}, we have
\begin{align*}
    \radius^{d} \max_{1 \leq i,j \leq N} \abss{ \text{cov}(\mathbb{B}'(\bar{f}_{i}, \bar{f}_{j})) - \text{cov}(\mathbb{B}(\bar{f}_{i}, \bar{f}_{j}))} = \mathcal{O}_{P} \parenth{\sqrt{\frac{\radius^{d} \log \radius}{n}}}.
\end{align*}
Putting all the above results together, we obtain the conclusion of the proposition.

\subsection{Proof of Proposition~\ref{prop:CI_nonparametric_regression}}
\label{subsec:proof:prop:CI_nonparametric_regression}
From the definition of $\nonpreg(x)$ in equation~\eqref{nonregfou}, we have
\begin{align}
    \nonpreg(x) = m(x) + \frac{\widehat{a}_{1}(x)}{\funcest_{n, \radius}(x)} + \frac{\widehat{a}_{2}(x)}{\funcest_{n, \radius}(x)}, \label{eq:CI_nonparametric_regression_zero}
\end{align}
where 
$$\widehat{a}_{1}(x) = \frac{1}{n \pi^{d}} \sum_{i = 1}^{n} (m(X_{i}) - m(x)) \prod_{j = 1}^{d} \frac{\sin(R (x_{j} - X_{ij}))}{x_{j} - X_{ij}}$$ and 
$$\widehat{a}_{2}(x) = \frac{1}{n \pi^{d}} \sum_{i = 1}^{n} \epsilon_{i} \prod_{j = 1}^{d} \frac{\sin(R (x_{j} - X_{ij}))}{x_{j} - X_{ij}}.$$ 
Since $\Exs \brackets{\widehat{a}_{2}(x)} = 0$, from the central limit theorem, we have
\begin{align*}
    \frac{\sqrt{n} \widehat{a}_{2}(x)}{\sqrt{n \var (\widehat{a}_{2}(x))}} \overset{d}{\to} \mathcal{N}(0, 1).
\end{align*}
Direct algebra shows that $\var (\widehat{a}_{2}(x)) = \frac{\sigma^2}{n \pi^{2d}} \Exs \brackets{ \prod_{j = 1}^{d} \frac{\sin^2(R (x_{j} - X_{.j}))}{(x_{j} - X_{.j})^2}}$ where $X = (X_{.1}, \ldots, X_{.d}) \sim p_{0}$. As $p_{0} \in \mathcal{C}^{2}(\mathcal{X})$, with the similar argument as that in Section~\ref{sec:point_wise_CI} we have $$\lim_{\radius \to \infty} \frac{n \var (\widehat{a}_{2}(x))}{\radius^{d}} = \frac{\sigma^2 p_{0}(x)}{\pi^{d}}.$$
Since $\funcest_{n, \radius}(x) \overset{p}{\to} p_{0}(x)$ as $n \to \infty$ and $\radius \to \infty$, we have
\begin{align}
    \sqrt{\frac{n}{\radius^{d}}} \frac{\widehat{a}_{2}(x)}{\funcest_{n, \radius}(x)} \overset{d}{\to} \mathcal{N}\parenth{0, \frac{\sigma^2}{p_{0}(x) \pi^{d}}}. \label{eq:CI_nonparametric_regression_first}
\end{align}
Moving to $\widehat{a}_{1}(x)$, we have
\begin{align*}
    \frac{\sqrt{n}\widehat{a}_{1}(x)}{\sqrt{n \var (\widehat{a}_{1}(x))}} \overset{d}{\to} \mathcal{N}(0, 1) + \frac{\Exs \brackets{\widehat{a}_{1}(x)}}{\sqrt{\var(\widehat{a}_{1}(x))}}.
\end{align*}
Since $\radius^{\alpha} = \mathcal{O}(\log n)$, from the argument of Theorem~\ref{theorem:random_nonparametric}, we have $\frac{\Exs \brackets{\widehat{a}_{1}(x)}}{\sqrt{\var(\widehat{a}_{1}(x))}} \to 0$ as $n \to \infty$. For the variance term $\var(\widehat{a}_{1}(x))$, direct calculation yields that
\begin{align*}
    \var(\widehat{a}_{1}(x)) = \frac{1}{n \pi^{2d}} \var \parenth{(m(X) - m(x)) \prod_{j = 1}^{d} \frac{\sin(R (x_{j} - X_{.j}))}{(x_{j} - X_{.j})}},
\end{align*}
where $X = (X_{.1}, \ldots, X_{.d}) \sim p_{0}$. We can check that $\Exs^2 \brackets{(m(X) - m(x)) \prod_{j = 1}^{d} \frac{\sin(R (x_{j} - X_{.j}))}{(x_{j} - X_{.j})}} \leq 2 \|p_{0}\|_{\infty}^2 \|m\|_{\infty}^2$ and 
\begin{align*}
    \Exs \brackets{(m(X) - m(x))^2 \prod_{j = 1}^{d} \frac{\sin^2(R (x_{j} - X_{.j}))}{(x_{j} - X_{.j})^2}} =  \frac{p_{0}^2(x)}{\radius^{2}} + o \parenth{ \frac{1}{\radius^2}},
\end{align*}
where the final equality is due to Taylor expansion up to the first order. Putting these results together, we have
\begin{align}
    \sqrt{\frac{n}{\radius^{d}}} \frac{\widehat{a}_{1}(x)}{\funcest_{n, \radius}(x)} \overset{p}{\to} 0. \label{eq:CI_nonparametric_regression_second}
\end{align}
Combining the results from equations~\eqref{eq:CI_nonparametric_regression_zero},~\eqref{eq:CI_nonparametric_regression_first}, and~\eqref{eq:CI_nonparametric_regression_second}, we obtain the conclusion of the proposition.
%%%%%%%%%%%%%%%%%%%%%%%%%%%%%%%%%%%%%%%%%%%%%%%%%%%%%%%%%%%%%%%%%%%%%%%%%%
\bibliographystyle{plainnat}
\bibliography{Nhat}
\end{document}

%% file: final_macros.tex
\newcommand{\brackets}[1]{\left[ #1 \right]}
\newcommand{\parenth}[1]{\left( #1 \right)}

\newcommand{\abss}[1]{\left| #1 \right |}

\newcommand{\radius}{\ensuremath{R}}
\newcommand{\funcest}{\ensuremath{\widehat{f}}}
\newcommand{\funcestmode}{\ensuremath{\widehat{f}}}
\newcommand{\sinfunc}{\ensuremath{\text{sin}}}
\newcommand{\cosfunc}{\ensuremath{\text{cos}}}

\newcommand{\modclus}{\ensuremath{\eta}}
\newcommand{\curvature}{\ensuremath{\lambda^{*}}}
\newcommand{\transitionmat}{\ensuremath{\widehat{p}}}

\newcommand{\deconvest}{\ensuremath{\widehat{g}}}
\newcommand{\deconv}{\ensuremath{g}}

\newcommand{\nonpreg}{\ensuremath{\widehat{m}}}

\newcommand{\vecnorm}[2]{\| #1 \|_{#2}}

\newcommand{\enorm}[1]{\vecnorm{#1}{2}} % euclidean norm

\newcommand{\Exs}{\ensuremath{{\mathbb{E}}}}
\newcommand{\Prob}{\ensuremath{{\mathbb{P}}}}

\DeclareMathOperator{\var}{var}
\DeclareMathOperator{\cov}{cov}
\DeclareMathOperator{\trace}{trace}

\newtheoremstyle{named}{}{}{\itshape}{}{\bfseries}{.}{.5em}{\thmnote{#3's }#1}
\theoremstyle{named}

\theoremstyle{plain}

\newtheorem{theorem}{Theorem}
\newtheorem{proposition}{Proposition}
\newtheorem{lemma}{Lemma}

\newtheorem{corollary}{Corollary}
\newtheorem{definition}{Definition}

\newtheorem{assumption}{Assumption}

\newlength{\widebarargwidth}
\newlength{\widebarargheight}
\newlength{\widebarargdepth}

\makeatletter
\long\def\@makecaption#1#2{
        \vskip 0.8ex
        \setbox\@tempboxa\hbox{\small {\bf #1:} #2}
        \parindent 1.5em  %% How can we use the global value of this???
        \dimen0=\hsize
        \advance\dimen0 by -3em
        \ifdim \wd\@tempboxa >\dimen0
                \hbox to \hsize{
                        \parindent 0em
                        \hfil
                        \parbox{\dimen0}{\def\baselinestretch{0.96}\small
                                {\bf #1.} #2
                                }
                        \hfil}
        \else \hbox to \hsize{\hfil \box\@tempboxa \hfil}
        \fi
        }
\makeatother

\long\def\comment#1{}
\definecolor{battleshipgrey}{rgb}{0.52, 0.52, 0.51}
\definecolor{darkgray}{rgb}{0.66, 0.66, 0.66}
\definecolor{darkgreen}{rgb}{0.0, 0.2, 0.13}
\definecolor{darkspringgreen}{rgb}{0.09, 0.45, 0.27}
\definecolor{dukeblue}{rgb}{0.0, 0.0, 0.61}
\definecolor{olivedrab7}{rgb}{0.24, 0.2, 0.12}
\definecolor{darkblue}{rgb}{0.0, 0.0, 0.55}
\definecolor{darkscarlet}{rgb}{0.34, 0.01, 0.1}
\definecolor{candyapplered}{rgb}{1.0, 0.03, 0.0}
\definecolor{ao(english)}{rgb}{0.0, 0.5, 0.0}
\definecolor{applegreen}{rgb}{0.55, 0.71, 0.0}